\documentclass[a4paper,reqno]{amsart}
\pdfoutput=1

\usepackage{lmodern}
\usepackage{xargs} 
\usepackage{microtype}

\usepackage[english]{babel}

\usepackage[T1]{fontenc}
\usepackage[utf8]{inputenc}
\usepackage{amsmath,amsthm,amssymb}
\usepackage{mathtools,thmtools} 

\usepackage{bm}
\renewcommand{\mathbf}[1]{\bm{\mathrm{#1}}}

\newcommand{\forcebold}[1]{#1}

\usepackage{mathrsfs,esint}
\usepackage{enumerate}

\DeclareUnicodeCharacter{211D}{$\mathbb{R}$}
\DeclareUnicodeCharacter{1D54A}{$\mathbb{S}$}

\usepackage[top=3cm, bottom=3cm, inner=2.7cm, outer=2.7cm,headsep=0.75cm,footskip=1.25cm]{geometry}

\usepackage{csquotes}
\usepackage[shortlabels]{enumitem}

\usepackage{hyperref}
\hypersetup{unicode,hypertexnames=false,colorlinks=true,linkcolor=blue,citecolor=blue,pdfauthor={Goldman-Merlet-Pegon},pdftitle={Uniform C1,α-regularity for almost-minimizers of some nonlocal perturbations of the perimeter}}
\usepackage[capitalize,nameinlink,noabbrev]{cleveref}

\usepackage[backend=biber,style=numeric,isbn=false,alldates=year,url=false,doi=false,eprint=false,giveninits=true,maxbibnames=50]{biblatex}
\renewbibmacro{in:}{} 
\AtEveryBibitem{\clearlist{location}} 
\AtEveryBibitem{\clearfield{urlyear}} 
\AtEveryBibitem{\clearfield{pagetotal}} 

\renewbibmacro{volume+number+eid}{%
    \textbf{
	\printfield{volume}%
    \setunit{\adddot}%
    \printfield{number}%
    \setunit{\addcomma\space}%
    \printfield{eid}}
}
\AtBeginBibliography{%

}

\DeclareFieldFormat
  [article,inproceedings,incollection,report,unpublished,online]
  {title}{#1\addcomma}
\DeclareFieldFormat
  [book,inbook,incollection,unpublished]
  {date}{(#1)}
\DeclareFieldFormat
  [book,inbook]
  {title}{\textit{#1}\addcomma}
\DeclareFieldFormat
  [article,inbook,book,incollection,inproceedings,report,unpublished,online]
  {pages}{pp. #1}

\DeclareDelimFormat[bib]{nametitledelim}{\space:\space}

\DeclareNameAlias{labelname}{given-family} 

\usepackage[final,color]{showkeys}
\renewcommand*\showkeyslabelformat[1]{
  \fbox{\parbox[t]{\marginparwidth}{\raggedright\normalfont\path{#1}}}}

\usepackage[dvipsnames,svgnames]{xcolor}
\makeatletter \def\SK@refcolor{\color{OrangeRed}} \makeatother
\makeatletter \def\SK@labelcolor{\color{Aquamarine}} \makeatother

\usepackage{graphicx}
\usepackage{caption}
\usepackage[mode=buildnew]{standalone}
\usepackage{tikz}
\usetikzlibrary{fpu}
\usetikzlibrary{math}
\usepackage[colorinlistoftodos,textsize=small,disable=true]{todonotes}
\newcommandx{\unsure}[2][1=]{\todo[linecolor=red,backgroundcolor=red!25,bordercolor=red,#1]{#2}}
\newcommandx{\change}[2][1=]{\todo[linecolor=blue,backgroundcolor=blue!25,bordercolor=blue,#1]{#2}}
\newcommandx{\info}[2][1=]{\todo[linecolor=OliveGreen,backgroundcolor=OliveGreen!25,bordercolor=OliveGreen,#1]{#2}}
\newcommandx{\improvement}[2][1=]{\todo[linecolor=Plum,backgroundcolor=Plum!25,bordercolor=Plum,#1]{#2}}

\usepackage{comment}

\makeatletter
\newcommand{\specialcell}[1]{\ifmeasuring@#1\else\omit$\displaystyle#1$\ignorespaces\fi}
\makeatother

\newcommand{\nhphantom}[1]{\sbox0{#1}\hspace{-\the\wd0}}

\newcommand{\numberset}[1]{{\mathbb{#1}}}
\newcommand{\N}{\numberset{N}}

\newcommand{\R}{\numberset{R}}

\newcommand{\dd}{\mathop{}\mathopen{}\mathrm{d}}

\newcommand{\dhx}{\mathop{}\mathopen{}\mathrm{d}\hat{x}} 
\newcommand{\dx}{\mathop{}\mathopen{}\mathrm{d}x} 
\newcommand{\dy}{\mathop{}\mathopen{}\mathrm{d}y} 
\newcommand{\dz}{\mathop{}\mathopen{}\mathrm{d}z} 
\newcommand{\dt}{\mathop{}\mathopen{}\mathrm{d}t} 
\newcommand{\dr}{\mathop{}\mathopen{}\mathrm{d}r} 
\newcommand{\ds}{\mathop{}\mathopen{}\mathrm{d}s}

\newcommand{\dH}{\mathop{}\mathopen{}\mathrm{d}\mathcal{H}} 
\newcommand{\dHn}{\mathop{}\mathopen{}\mathrm{d}\mathcal{H}^{n-1}} 
\newcommand{\ddt}{\mathop{}\mathopen{}\frac{\mathrm{d}}{\mathrm{d}t}}

\newcommand{\bbs}{\mathbb{S}}
\newcommand{\bbsn}{\mathbb{S}^{n-1}}

\newcommand{\calF}{{\mathcal{F}}}
\newcommand{\calH}{{\mathcal{H}}}

\newcommand{\scrE}{{\mathscr{E}}}
\newcommand{\scrF}{{\mathscr{F}}}

\def\lt{\left}
\def\rt{\right}

\newcommand{\f}{\mathbf{f}}

\DeclareMathOperator{\Sha}{Sh}

\newcommand{\cy}{\mathbf{C}}
\newcommand{\ky}{\mathbf{K}}

\newcommand{\unD}{\mathrm{1D}}

\newcommand{\Per}{\mathit{P}}
\newcommand{\hn}{{\calH^{n-1}}}
\newcommand{\dhn}{\dH^{n-1}}
\newcommand{\e}{{\mathbf{e}}}

\newcommand{\ind}{{\mathbf{1}}}
\newcommand{\Rpn}{{\R^n}} 

\newcommand{\Om}{\Omega}
\newcommand{\om}{\omega}
\newcommand{\eps}{\varepsilon}
\newcommand{\vphi}{\varphi}

\newcommand{\wt}{\widetilde}

\newcommand{\loc}{{\mathrm{loc}}}

\newcommand{\id}{{\mathrm{Id}}}
\newcommand{\compl}{{\mathrm{c}}}
\newcommand{\Lip}{{\mathrm{Lip}}}
\newcommand{\grad}{\nabla}
\newcommand{\subsq}{\subset}
\newcommand{\csubset}{\subset\subset} 
\newcommand{\LL}{\mathop{\hbox{\vrule height 6pt width .5pt depth 0pt \vrule height .5pt width 3pt depth 0pt}}\nolimits}

\newcommand{\Citeauthorsc}[1]{\textsc{\Citeauthor{#1}}}

\DeclareMathOperator{\dvg}{\mathrm{div}}
\DeclareMathOperator{\dist}{\mathrm{d}}
\DeclareMathOperator{\spt}{\mathrm{spt}}

\DeclarePairedDelimiter{\abs}{\lvert}{\rvert}

\DeclarePairedDelimiter{\norm}{\lVert}{\rVert}

\numberwithin{equation}{section} 

\declaretheorem[name=Theorem,numberwithin=section]{mainres}

\declaretheorem[name=Theorem,numberwithin=section]{thm}
\declaretheorem[name=Lemma,numberlike=thm]{lem}
\declaretheorem[name=Proposition,numberlike=thm]{prp}

\declaretheorem[name=Definition,numberlike=thm,style=definition]{dfn}
\declaretheorem[name=Remark,numberlike=thm,style=remark]{rmk}

\crefname{equation}{}{}
\crefname{enumi}{}{}
\crefname{thm}{Theorem}{Theorems}
\crefname{mainthm}{Theorem}{Theorems}
\crefname{mainres}{Theorem}{Theorems}
\crefname{maindfn}{Definition}{Definitions}
\crefname{lem}{Lemma}{Lemmas}
\crefname{mainapp}{Application}{Applications}
\crefname{prp}{Proposition}{Propositions}
\crefname{cor}{Corollary}{Corollaries}
\crefname{dfn}{Definition}{Definitions}
\crefname{rmk}{Remark}{Remarks}
\crefname{conj}{Conjecture}{Conjectures}
\crefname{ex}{Example}{Examples}
\crefname{appsec}{Appendix}{Appendices}

\def\Xint#1{\mathchoice
{\XXint\displaystyle\textstyle{#1}}%
{\XXint\textstyle\scriptstyle{#1}}%
{\XXint\scriptstyle\scriptscriptstyle{#1}}%
{\XXint\scriptscriptstyle
\scriptscriptstyle{#1}}%
\!\int}
\def\XXint#1#2#3{{%
\setbox0=\hbox{$#1{#2#3}{\int}$}
\vcenter{\hbox{$#2#3$}}\kern-.5\wd0}}

\def\dashint{\Xint-}

\renewcommand{\iint}{\int}

\renewcommand{\leq}{\leqslant}
\renewcommand{\geq}{\geqslant}

\newcommand{\les}{\lesssim}

\title[Uniform regularity for nonlocal perturbations of the perimeter]{Uniform \forcebold{$C^{1,\alpha}$}-regularity for almost-minimizers of some nonlocal perturbations of the perimeter}
\author{M. Goldman}
\address{CMAP, CNRS, \'Ecole polytechnique, Institut Polytechnique de Paris, 91120 Palaiseau, France}
\email{michael.goldman@cnrs.fr}
\author{B. Merlet}
\author{M. Pegon}
\address{Univ. Lille, CNRS, UMR 8524, Inria - Laboratoire Paul Painlev\'e, F-59000 Lille}
\email{benoit.merlet@univ-lille.fr}
\email{marc.pegon@univ-lille.fr}
\date{}
\begin{document}

\begin{abstract}
In this paper, we establish a $C^{1,\alpha}$-regularity theorem for almost-minimizers of the
functional $\mathcal{F}_{\varepsilon,\gamma}=P-\gamma P_{\varepsilon}$, where $\gamma\in(0,1)$ and
$P_{\varepsilon}$ is a nonlocal energy converging to the perimeter as $\varepsilon$ vanishes.
Our theorem provides a criterion for $C^{1,\alpha}$-regularity at a point of the boundary which is
\textsl{uniform} as the parameter $\varepsilon$ goes to $0$. Since the two  terms in the energy are
of the  same order when $\varepsilon$ is small, we are considering here much  stronger nonlocal
interactions than those considered in most related works.
As a consequence of our regularity result, we obtain that, for $\varepsilon$ small enough, volume-constrained minimizers of
$\mathcal{F}_{\varepsilon,\gamma}$ are balls. For small $\varepsilon$, this minimization problem
corresponds to the large mass regime for a Gamow-type problem where the nonlocal repulsive term is
given by an integrable kernel $G$ with sufficiently fast decay at infinity.
\end{abstract}

\maketitle

\tableofcontents

\section{Introduction}
\addtocontents{toc}{\protect\setcounter{tocdepth}{1}}
The aim of this paper is to complete the program started in~\cite{Peg2021,MP2021} regarding the behavior of minimizers 
of a variant of Gamow's liquid drop model (see~\cite{CMT2017}) in the regime of large mass. 
We are interested in the minimization problem
\begin{equation}\label{cminpb}\tag{$\mathcal{P}$}
\min~ \Big\{ \calF_{\gamma,\eps}(E)\coloneqq P(E)-\gamma P_\eps(E)~:~ E\subsq\R^n \text{ with } \abs{E}=\abs{B_1}\Big\},
\end{equation}
where $n\ge 2$, $\gamma\in (0,1)$, $P$ is the Euclidean perimeter and $P_\eps$ is a nonlocal
perimeter functional such that $P_\eps\to P$ (both pointwise and in the sense of $\Gamma$-convergence) as $\eps\to 0$.
More precisely, given a  radial function $G:\R^n\mapsto(0,\infty)$ with finite first moment, we define the rescaled
kernels $G_\eps$ by~$G_\eps(z)\coloneqq \eps^{-(n+1)}G(\eps^{-1}z)$ for all $z\in\R^n$ and the nonlocal
perimeter $P_\eps$ by
\[
P_\eps(E)\coloneqq \iint_{\Rpn\times\Rpn} \abs{\ind_E(x)-\ind_E(y)} G_\eps(x-y)\dx\dy =
2\int_{E\times E^\compl} G_\eps(x-y)\dx\dy.
\] 
It is well-known that when $G$ is integrable,~\cref{cminpb} is indeed equivalent to Gamow's model
after appropriate rescaling (see~\cite{Peg2021} for instance). In particular the regime of small
$\eps$ in~\cref{cminpb} corresponds to \emph{large mass} in Gamow's model.

The main contribution of this paper is a $C^{1,\alpha}$-regularity theorem for almost-minimizers of
$\calF_{\eps,\gamma}$ which is
\textsl{uniform} as $\eps$ goes to $0$, under suitable assumptions on the kernel $G$. This result is stated further on as~\cref{mainthm:holdreg}. To the best of our knowledge, this latter is the first uniform regularity result (in dimension higher than 2) for
a problem involving the competition of two local/nonlocal perimeters, where at the macroscopic scale, \emph{neither of the terms is
negligible in front of the other}.

In combination with the Fuglede type computations done in~\cite{MP2021}, we obtain the following characterization of minimizers of~\cref{cminpb} for
small~$\eps$. This extends to any arbitrary space dimension the two-dimensional result~\cite[Theorem
A]{MP2021} which is not based on regularity theory.

\begin{mainres}[Minimality of the unit ball]\label{mainthm:minball}
Assume that $n\ge 2$, $\gamma\in(0,1)$ and that $G$ satisfies the assumptions~\crefrange{Hposrad}{Hdec} described below. Then,
there exists $\eps_{\mathrm{ball}}=\eps_{\mathrm{ball}}(n,G,\gamma)>0$, such that, for every
$\eps\le\eps_{\mathrm{ball}}$, the unit ball is the unique minimizer of~\cref{cminpb}, up to
translations.
\end{mainres}

In terms of Gamow's problem~\cref{cminpb}, this result  provides a very general set of assumptions on the
kernel $G$ for which minimizers of
\[
\min~ \Big\{ P(E)+\int_{E\times E} G(x-y)\dx\dy~:~ E\subsq\R^n \text{ with } \abs{E}=m\Big\}
\]
for \emph{large masses $m$} are balls: namely, if $G$ satisfies~\cref{Hposrad,,Hmoment2,,Hint0,,Hdec} and its first
moment is below the explicit treshold of~\cref{fixedmoment1}. This is in sharp contrast with the case
of Riesz kernels $G(x)=\frac{1}{\abs{x}^{\alpha}}$, $\alpha\in(0,n)$, where we observe the following
dichotomy: for \emph{small masses} the nonlocal term is \emph{negligible} in front of the
perimeter, and minimizers are balls, while for \emph{large masses}, the problem does not admit
minimizers (at least if $\alpha\in (0,2]$).

For our $C^{1,\alpha}$-regularity theorem, we work with a classical notion of almost-minimality for
$\calF_{\eps,\gamma}$.

\begin{dfn}[Almost-minimizers]\label{def:almostmin}
Let $\gamma\in(0,1)$ and $\eps>0$. For any positive constants $\Lambda$ and~$r_0$, we say that $E$
is a $(\Lambda,r_0)$-minimizer of $\calF_{\eps,\gamma}$ if for every set of finite perimeter
$F\subsq\R^n$ such that $E\triangle F\csubset B_r(x)$ with $0<r\le r_0$ and $x\in\R^n$, we have
\[
\calF_{\eps,\gamma}(E)\le\calF_{\eps,\gamma}(F)+\Lambda\abs{E\triangle F}.
\]
\end{dfn}

We will see in~\cref{prp:equivpbs} that (volume constrained) minimizers of~\cref{cminpb} are indeed (unconstrained) $(\Lambda,r_0)$-minimizers
of $\calF_{\eps,\gamma}$ for any $r_0$ and some constant $\Lambda$, not depending on $\eps$. This type of relaxation of the volume constraint is standard for this kind of problems (see e.g.
\cite{Rig2000,FE2011,GN2012,FFM+2015}).

For $k\in\N$ and a general kernel $K$, it is convenient to introduce the $k$-th moment of $K$,
which is defined by
\begin{equation}\label{defIk}
 I_K^k\coloneqq \int_{\R^n} |z|^k |K(z)| \dz.
\end{equation}
In this work, $G$ always satisfies the following hypotheses:
\begin{enumerate}[(H1),topsep=6pt,itemsep=4pt]
\item\label{Hposrad} $G$ is a measurable, nonnegative, radial function, that is, there exists $g:(0,\infty)\to [0,\infty)$ such that $G(z)=g(\abs{z})$ for every
$z\in\Rpn\setminus\{0\}$;
\item\label{Hmoment1} $z\mapsto \abs{z}G(z)\in L^1(\Rpn)$ and the first moment is normalized by
\begin{equation}\label{fixedmoment1}
I_G^1  = \frac{1}{\mathbb{K}_{1,n}},
\end{equation}
where $\displaystyle\mathbb{K}_{1,n}\coloneqq \dashint_{\bbsn} \abs{x_n} \dHn$.
\end{enumerate}
We will also use the following additional assumptions on $G$:
\begin{enumerate}[(H1),resume,topsep=6pt,itemsep=4pt]
\item\label{Hmoment2} $G\in W^{1,1}_{\loc}(\Rpn\setminus\{0\})$, $\displaystyle I_{\abs{\nabla G}}^2<\infty$, and $|g'(r)|\le \frac{C}{r^{n+1}}$ for $r\ge 1$	;
\item\label{Hint0} $\displaystyle\int_{B_1 \setminus B_r}
G(z)\dz \le \frac{C}{r^{s_0}}$ for every $r\in(0,1)$, for some constants $C>0$ and $s_0\in (0,1)$;
\item\label{Hdec} 
Denoting 
$Q:\R_+\to\R_+$ the rate function defined by
\begin{equation}\label{eq:defQ}
Q(r)\coloneqq \int_{\R^n\setminus B_r} \abs{z}G(z)\dz,\qquad\forall r\in[0,\infty),
\end{equation} 
there holds $\displaystyle Q(r)\le \frac{C}{r^{n-1+p_0}}$ for every $r>0$, for some
constants $C>0$ and $p_0>0$.
\end{enumerate}

Let us briefly comment on these assumptions.
\begin{enumerate}[(i),topsep=6pt,itemsep=4pt]
\item~\cref{Hposrad,Hmoment1} guarantee that $P_\eps$ is well-defined on sets of finite perimeter and converges to the standard perimeter. It is shown in~\cite{Peg2021} that these two assumptions are sufficient to ensure both the existence of minimizers for small $\eps$ and their convergence to the ball as $\eps$ vanishes.
\item~\cref{Hmoment2} (in the form of $I^2_{|\nabla G|}<\infty$) is used to compute the first variation of $\calF_{\eps,\gamma}$ (see~\cref{lem:firstvarf}).
It is also needed in its full version in order to apply the stability result~\cite[Theorem~3]{MP2021} for nearly spherical sets.
\item~\cref{Hint0} states that for small scales the perimeter is dominant over $P_\eps$, leading to regularity at small scales by classical regularity
theory for almost minimizers of the perimeter (see~\cref{prp:almostregsmallscale} and~\cref{prp:smallscales_decexcess}). The hypothesis is weaker than assuming $G$ to be locally integrable. It roughly states that, near $0$, $G$ behaves at most like the kernel of the $s_0$-fractional perimeter (see for instance~\cite{CV2013,DNRV2015,FFM+2015}).
\item Finally,~\cref{Hdec} limits the contributions of long-range interactions in $P_\eps$. We do not believe this condition to be sharp. 
\end{enumerate}
We now also comment on the restrictions these conditions impose on the kernel and give a few examples where  these assumptions are satisfied.
\begin{enumerate}[(i),topsep=6pt,itemsep=4pt]
\item With~\cref{Hmoment1}, one can check that assumption~\cref{Hdec} is equivalent to
$I_G^{n+q_0}<\infty$ for some positive~$q_0$ (possibly different from $p_0$).
\item If $G$ is a power law function near the origin, that is, $G\propto \abs{z}^{-\alpha}$ for some
$\alpha>0$ in a neighborhood of $0$, then~\cref{Hint0} states that $\alpha\le n+s_0$.
Notice that in that particular example, $\abs{\cdot}^2\grad G$ is integrable near the origin, which
is a part of~\cref{Hmoment2}.
\item If $G$ is a power law function at infinity, that is $G\propto \abs{z}^{-\beta}$ at infinity,~\cref{Hdec} states that $\beta \ge 2n+p_0$. In that
particular example, $\abs{\cdot}^2\grad G$ is integrable at infinity, and $|g'(r)|\le\frac{C}{r^{n+1}}$ when
$r\to\infty$, which is the other part of~\cref{Hmoment2}.
\item From the two previous points, we see that the kernel $G$ defined by
\[
G(z)\propto \min\lt(\frac{1}{\abs{z}^{n+s_0}},\frac{1}{\abs{z}^{2n+p_0}}\rt)
\]
with $s_0\in(0,1)$, $p_0>0$, satisfies assumptions
\crefrange{Hposrad}{Hdec}.\\
Other admissible kernels are multiples of the Bessel kernels $\mathcal{B}_{\alpha,\kappa}$, defined
for any $\alpha>0$ and~$\kappa>0$ as the fundamental solution of the operator
$(\id-\kappa\Delta)^{\frac{\alpha}{2}}$. Indeed, Bessel kernels are smooth away from zero, decay
exponentially at infinity and, near the origin
\[
\mathcal{B}_{\alpha,\kappa}\propto
\begin{cases}
\frac{1}{\abs{z}^{n-\alpha}}&\qquad\text{ for }\alpha\in(0,n),\\
-\log(\abs{z})&\qquad\text{ for }\alpha=n,\\
1&\qquad\text{ for }\alpha>n.\\
\end{cases}
\]
Let us point out that they correspond to screened Coulomb kernels in the case $\alpha=2$, see
\cite{KMN2020}. 
Eventually, our paper covers the case of integrable and  compactly supported kernels (with the
extra assumption~\cref{Hmoment2}), which was  first studied in~\cite{Rig2000}.
\end{enumerate}

\vspace{2pt}

To state our $C^{1,\alpha}$-regularity theorem and sketch its proof, we need to introduce the
notion of (spherical) excess, which measures the variation of the normal vector to the boundary of a
set near a point.

For a set of finite perimeter $E$, we always implicitly assume that $E$ denotes a well-chosen
representative in the class of Borel sets $F$ such that $|F\Delta E|=0$, that is, we assume that  $\partial E$ satisfies (see e.g.
\cite[Proposition~12.19]{Mag2012})
\[
\partial E=\spt |D\ind_E| =\Big\{x\in \Rpn~:~ 0< \abs{E\cap B_r(x)}<\abs{B_r(x)} \text{ for all }
r>0\Big\}.
\]
We denote by $\partial^* E$ the reduced boundary of $E$, and by $\nu_E(x)$ the outer unit normal to
$\partial^* E$ at $x$.

\begin{dfn}[Spherical excess]\label{dfn:sphericalexcess}
For any set of finite perimeter $E\subsq\R^n$ we define the spherical excess (or simply excess) of
$E$ in $x\in\partial E$ at scale $r>0$ by
\[
\e(E,x,r)\coloneqq \inf_{\nu\in\bbsn} \frac{1}{r^{n-1}}\int_{\partial^* E\cap B_r(x)}
\frac{\abs{\nu-\nu_E(y)}^2}{2}\dhn_y,
\]
where we used the short-hand notation $\dhn_y$ for $\dhn(y)$.
\end{dfn}

 When $x=0$ we simply denote $\e(E,r)=\e(E,0,r)$ (which we will usually assume by translation invariance of the statements). We can now state our main \enquote{epsilon-regularity} theorem.

\begin{mainres}\label{mainthm:holdreg}
Assume that $G$ satisfies~\crefrange{Hposrad}{Hdec}, and let $\gamma\in(0,1)$ and
$\Lambda>0$.
Then there exist positive constants $\tau_{\mathrm{reg}}$, $\eps_{\mathrm{reg}}$,
$\beta\in(0,1)$, and $\alpha\in(0,1)$ depending only on $n$, $G$ and $\gamma$ such that the
following holds.
Let  $E$ be a $(\Lambda,r_0)$-minimizer of $\calF_{\eps,\gamma}$ with $\eps\in(0,\eps_{\mathrm{reg}})$
and  $0\in\partial E$. Assume that for some $R\in[\eps^{1-\beta}, r_0]$,
\[
\e(E,R)+\Lambda R\le \tau_{\mathrm{reg}},
\]
then, up to a rotation, $\partial E$ coincides in $B_{R/2}$ with the graph of a $C^{1,\alpha}$
function $u:\R^{n-1}\mapsto \R$. Moreover,
\[
[\nabla u]_{\alpha, \frac{R}{2}}^2 \le C\lt( \frac{1}{R^{2\alpha}}(\e(E,R)+\Lambda R) +1\rt)
\]
for some $C=C(n,G,\gamma)>0$. 
Here $[\cdot]_{\alpha,R}$ denotes the H\"older semi-norm  in the ball of radius $R$ in~$\R^{n-1}$.
\end{mainres}

Let us notice that by~\cref{Hint0}, we already know that  $(\Lambda,r_0)$-minimizers with small excess at scale $\eps$ are $C^{1,\alpha}$.  This follows from the classical regularity theory of almost-minimizers of the perimeter. The main point of~\cref{mainthm:holdreg}
is that the estimate holds at a fixed scale $R$ \emph{uniformly as $\eps$ goes to 0}.

\begin{rmk}
We could generalize the definition~\cref{def:almostmin} of $(\Lambda,r_0)$-minimizers to an open subset $\Om\subsq\R^n$, imposing that competitors $F$ differ from $E$ only in balls $B_r(x)\subsq \Om$. 
We would obtain a uniform regularity result for $(\Lambda,r_0)$-minimizers in compact subsets of $\Om$. This applies for instance to sets $E$ which are prescribed outside $\Om$ and minimize $\calF_{\eps,\gamma}$ locally in $\Om$. Our arguments work just the same.
\end{rmk}

We now give the main steps of the proof of~\cref{mainthm:holdreg} or more precisely of the excess decay stated in~\cref{thm:decexcess}. While the overall strategy follows the classical regularity theory for minimizers of the perimeter
as pioneered by De Giorgi, Federer, Almgren and Allard (to name a few), the presence of a
\emph{competing nonlocal term} makes it quite delicate to adapt. In particular, the arguments
described below in Steps 2 \& 4 are new. Indeed, as we go from large scales to
smaller scales the nonlocal nature of $P_\eps$ is increasingly important. Even at scales $r\gg\eps$
the nonlocal contributions at distances $s\gtrsim r$ has to be precisely estimated. Besides, for the
Caccioppoli inequality we introduce a new strategy based on the slicing techniques of~\cite{MP2021,Lud2014}
and on the study of the new \emph{shadow} functional $\Sha$.\medskip

Here we follow the order of the presentation of the classical theory of~\cite{Mag2012} and we use the same notation.

We first establish density upper and lower bounds, both for the volume and the perimeter. 
This is  a direct  consequence of a  weak quasi-minimality property for $(\Lambda,r_0)$-minimizers
$E$ of $\calF_{\eps,\gamma}$ (see~\cite[Theorem~5.6]{Fus2015} or~\cite{DS1995}). Indeed, we prove
in~\cref{prp:wquasimin} that $E$ satisfies
\[
P(E;B_r(x))\le C P(F;B_r(x))
\]
for every  $F$ such that $E\triangle F\csubset B_r(x)$ with $x\in\R^n$, 
$0<\Lambda r\le 1-\gamma$ and~$C$ depends only on $n$, $G$ and~$\gamma$.

We then prove the excess decay itself. We argue differently for scales $r\ge r_+\gg \eps$ and $r\le r_-\ll \eps$. 

The decay of the excess in the situation $r\gg \eps$ is done in~\cref{prp:largescales_decexcess} and represents most of the work.
The proof goes through a Campanato iteration scheme which relies on the improvement of the excess by tilting proven in~\cref{lem:excess_improv}. In turn this Lemma states that if
the excess is small at some scale $r\gg \eps$ then up to tilting and an error of the order of $Q(r/\eps)$ (with $Q$ defined by~\eqref{eq:defQ}), the excess is much smaller at a scale $\lambda r$ for some $\lambda\ll1$. 
Let us stress again that as opposed to the usual applications of this idea, here the
error term gets \emph{larger as $r$ decreases}.\\
Roughly speaking, the idea of the proof of the tilt Lemma is that for $r\gg \eps$, $P_\eps$
is close to the perimeter functional (at least for sufficiently smooth sets). Writing $\calF_{\eps,\gamma}= (1-\gamma) P +\gamma (P-P_\eps)$ we may hope to reproduce the classical strategy for the excess decay by treating $P-P_\eps$ as a negligible term. 
This is actually quite delicate. Indeed, formal computations show that on smooth sets the energy  $P-P_\eps$ penalizes curvature rather than volume (or even perimeter). The remainder $P-P_\eps$ is therefore a  singular perturbation of the main term $(1-\gamma) P$.\\ 
%
%
We now sketch the four main steps of this strategy. We first need to introduce the cylindrical
excess and some more notation.
Let
\begin{equation}\label{eq:cyl}
\cy(x,r,\nu)=x+\Big\{ y+t\nu~:~ y\in \nu^\perp
\text{ with } \abs{y}<r \text{ and } -r<t<r\Big\}
\end{equation}
denote the (truncated) cylinder centered at $x\in\R^n$ with direction $\nu\in\bbs^{n-1}$, basis
radius $r$ and height~$2r$. 

\begin{dfn}[Cylindrical excess]\label{dfn:cylindricalexcess}
For any set of finite perimeter $E\subsq\R^n$ and any cylinder $\cy(x,r,\nu)$ centered at
$x\in\partial E$ we define the cylindrical excess of $E$ in $\cy(x,r,\nu)$ by
\begin{equation}\label{eq:defcylexcess}
\e(E,x,r,\nu)\coloneqq \frac{1}{r^{n-1}}\int_{\partial^* E\cap \cy(x,r,\nu)}
\frac{\abs{\nu-\nu_E(y)}^2}{2}\dhn_y.
\end{equation}
\end{dfn}
As above, if $x=0$ we simply denote $\cy(r,\nu)=\cy(0,r,\nu)$ and $\e(E,r,\nu)=\e(E,0,r,\nu)$. We can now proceed with the sketch of the proof for large scales.\medskip\\
\textit{Step 1.}
We show in~\cref{thm:lipapprox1} that if the excess of a $(\Lambda,r_0)$-almost minimizer $E$ of $\calF_{\eps,\gamma}$ is
small in a cylinder $\cy(4r,\nu)$, then $\partial E\cap\cy(2r,\nu)$ is almost flat and almost
entirely covered by the graph of a Lipschitz function $u$. As observed in~\cite{DHV2022}, this is based on the so-called \textsl{height bound} (see~\cref{prp:heightbound}) which relies only
on the density estimates so that we can directly appeal to~\cite{Mag2012}.  \medskip\\
\textit{Step 2.} In~\cref{thm:lipapprox2}, we show that the function $u$ almost satisfies an equation of the form
\[
\lt(\Delta-\gamma \Delta_{G_{\eps}}\rt)u=0\qquad \text{ in }\ \cy(r,\nu),
\] 
where $\Delta_{G_{\eps}}$ is a nonlocal
operator converging to $\Delta$ as $\eps\to 0$. For this part, we proceed as follows.
\begin{enumerate}[1.,topsep=6pt,itemsep=4pt]
\item In~\cref{lem:firstvarf}, we write the Euler--Lagrange equation associated with deformations of $E$ in the direction of $\nu$. 
\item In~\cref{lem:lipapprox2_step1}, we localize the equation to the cylinder $\cy(2r,\nu)$ by carefully discarding the negligible long-range interaction terms. 
\item In~\cref{lem:lipapprox2_step2}, we pass the equation on $\partial E$ to the graph of $u$ using their proximity.
\item We linearize the equation.
\item 
Eventually, since $r$ is much larger than $\eps$, we have formally
$(\Delta-\gamma\Delta_{G_{\eps}})\simeq(1-\gamma)\Delta$ in $\cy(r,\nu)$, so that $u$ is close almost harmonic in this set, see~\cref{prp:harmapprox}.
\end{enumerate}
\noindent
\textit{Step 3.} Using that $u$ is close to a harmonic function, we show that the \textsl{flatness} of
$E$ (see~\cref{dfn:flatness}) at some smaller scale $\lambda r$ is much smaller than the excess
at scale $4r$, up to tilting the direction.  This part is relatively  standard.\medskip\\
%
\textit{Step 4.} By analogy with functions, one should think of the excess of $E$ as the Dirichlet
energy of $u$, and think of the flatness of $E$ as the $L^2$ norm of $u$. To transfer the smallness
of the flatness at scale $\lambda r$ to the excess, we prove in~\cref{prp:strongcaccio} a Caccioppoli-type inequality (or
Reverse Poincaré), stating roughly
\[
\e(E,\lambda r/2,\nu)\, \lesssim\, \f(E,\lambda r,\nu)
+\left(\frac{\eps}{\lambda r}\right)^{\theta}\e(E,\lambda r,\nu)+\text{ \enquote{smaller terms}}
\]
whenever $\lambda r$ is still much larger than $\eps$, and where $\theta\in(0,1)$.
Our proof of the Caccioppoli inequality relies on an improved quasi-minimality condition when the
set $E$ is already known to be sufficiently flat (see~\cref{prp:benoitmin}).
Here the most delicate point is to bound the short-range nonlocal
interactions. For this, we use a 1D slicing method from ~\cite{MP2021,Lud2014}. It consists in
writing the nonlocal perturbation $(P-P_\eps)(E)$ as the integral over all lines $L$ of the  1D
energies $(P^{1D}-P^{1D}_\eps)(E\cap L)$. We are led to study an auxiliary optimization problem
which involves the \enquote{shadow funtional} $\Sha$. Although the latter is a fairly natural variant
of the integral geometric measure, we are not aware of any work about~it. We will comment more on
this functional after the conclusion of this sketch of proof.

\medskip
The above method establishes the decay of the excess from the macroscopic scales down  to a scale $r_+\gg\eps$. To prove excess decay for smaller scales, we first jump from $r_+$ to a  scale $r_-\ll\eps$ by using a crude estimate, see~\cref{prp:scaleexcess}. To compensate for the loss introduced at this step we need the excess to be already small enough at scale $r_+$. This explains the introduction of both hypothesis~\cref{Hdec} and condition $R\ge \eps^{1-\beta}$ in~\cref{mainthm:holdreg}. Eventually, at smaller scales $r\le r_-\ll\eps$, the nonlocal perturbation $ P_\eps$ behaves as a volume term and we are able to rely on the classical regularity theory for almost-minimizers of the perimeter, see~\cref{prp:smallscales_decexcess}.

\bigskip

 We now comment on the shadow functional $\Sha$. Denoting $B'$ the unit $(n-1)$-ball and $C':=B'\times\R$, we optimize $\Sha(\partial^*F \cap C')$  over the  sets $F$ obstructing the cylinder $C'$. Specifically:
\begin{enumerate}[$\circ$]
\item The obstruction condition reads
\[
B'\times (-\infty,-1/4)\ \subset\ F\ \subset\ B'\times (-\infty,1/4).
\]
\item For a Borel set $\Sigma\subset\R^n$, $\Sha(\Sigma)$ is the average over all lighting directions of the area of the shadow of $\Sigma$, namely,
\[
 \Sha(\Sigma):=\dfrac1{2\om_{n-1}}\int_{\bbsn} \calH^{n-1}\big(\pi_{\sigma^\perp}\Sigma\big) \dhn_\sigma,
\]
where $\pi_{\sigma^\perp}$ denotes the orthogonal projection on the hyperplane $\sigma^\perp$. By Fubini, we may equivalently integrate over all possible lines $L$ and count one if $L$ intersects  $\Sigma$ and 0 in the other cases:
\[
 \Sha(\Sigma)=\dfrac1{2\om_{n-1}}\int_{\bbsn} \int_{\sigma^\perp}\chi\lt((x+\R\sigma)\cap\Sigma\rt) \,\mathrm{d}\calH^{n-1}_x\,\dhn_\sigma,
\]
where $\chi(\emptyset)=0$ and  $\chi(A)=1$ for nonempty subsets $A\subset\R^n$.
\end{enumerate}
The prefactor is chosen so that $\Sha(P)=\calH^{n-1}(P)$ for Borel sets $P$ lying in a
$(n-1)$-affine subspace of~$\R^n$. Notice that substituting the counting measure $\calH^0$ for
$\chi$, we obtain the integral geometric  measure  $\mathcal{I}^{n-1}$, hence, $\Sha(\Sigma)$ is not
larger (and in general smaller) than $\mathcal{I}^{n-1}(\Sigma)$.\\
We prove in~\cref{lem_shadow} that  $B'\times(-\infty, 0)$ minimizes $\Sha(\partial^*F\cap [B'\times\R])$ among  obstructing sets $F$. 
However, substituting $\Sha$ for $\mathcal{H}^{n-1}$ in the Plateau problem, the solutions (if any) are in general not minimal surfaces anymore. The obvious reason is that, denoting 
\[
\wt\Sigma :=\{x\in\R^n : (x+\R \sigma) \cap\Sigma\ne \emptyset\text{ for every }\sigma\in\bbsn\},
\]
there holds $\Sha(\wt\Sigma)=\Sha(\Sigma)$  but for nonplanar hypersurfaces, $\calH^n(\wt\Sigma)>0$.  Considering for instance the Steiner problem associated with the set of vertices $V$ of some convex polygon $P\subset\R^2$, the minimizers of $\Sha$  among closed connected sets containing $V$ are exactly the closed connected sets $\Sigma$ such that $V\subset \Sigma\subset P$. This includes sets with Hausdorff dimensions ranging from 1 to 2. The situation may be more subtle in higher dimension, as minimizers of $\Sha$  among  sets spanning a given boundary $\Gamma$ may not contain a solution of the corresponding Plateau problem. For instance, for $n=3$, if $\Gamma=\bbs^1\times\{-t,t\}$ for some $t>0$ then for $t$ small enough the minimizer of $\Sha$ is the (empty) cylinder $\bbs^1\times(-t,t)$.



\subsection*{Motivation and related results}
As already alluded to, under the additional hypothesis that  $G\in L^1$ and after rescaling,
$\calF_{\eps,\gamma}$ is equivalent to the  generalized Gamow functionals (see~\cite{CFP2021}),
\begin{equation}\label{probGamow}
\min~ \Big\{ P(E)+\gamma \int_{E\times E} G(x-y) \dx\dy~:~ \abs{E}=m\Big\}.
\end{equation}
Besides the case of compactly supported kernels studied in~\cite{Rig2000} and for which existence of minimizers holds for any $m$, the main example studied in the literature is the case of Riesz
interaction energies, $G(z)= |z|^{-(n-\alpha)}$ with $\alpha\in(0,n)$. The case $n=3$ and $\alpha=2$ corresponding to Gamow's
liquid drop model for the atomic nucleus, see~\cite{CMT2017} for a short overview on this problem.
In this problem it has been shown in~\cite{KM2014,FFM+2015} that, for small $m$, minimizers
are balls while for large $m$ there is non-existence of minimizers under the assumption that $\alpha\in [n-2,n)$, see~\cite{KM2014,LO2014,FN2021}. The question of the existence of minimizers in the case $\alpha\in (0,n-2)$  is still open.
The proof of the rigidity of balls for small $m$  in~\cite{KM2014,FFM+2015} (see also
\cite{CFP2021} for the case of quite general kernels $G$) is of the same spirit as for
\cref{mainthm:minball} and goes through a Selection Principle along the lines of~\cite{CL2012}.
Notice however that in that case one can directly rely on the classical regularity theory for quasi-minimizers of the perimeter, see~\cref{eq:difPe2}. 
Let us point out that  in a related direction, it has been shown in~\cite{ABTZ2021} that if we
replace the Euclidean perimeter in~\cref{probGamow} by a weighted perimeter $P_a(E) \coloneqq
\int_{\partial^* E} a(x)\dH_x^{n-1}$, with a weight $a$ growing fast enough at infinity, then balls
are the unique minimizers for large $m$.\\

Recently, there has been a growing interest for related nonlocal isoperimetric problems which do not fall within the standard regularity theory for perimeter almost-minimizers.
A first example comes from a variational model for charged liquid drops where the kernel is still given by the Riesz interaction kernel but
now the charge is not assumed to be uniformly distributed on $E$. This leads in general to much more singular interactions, see~\cite{GNR2015}. 
However, introducing either an additional penalization of the charge as in~\cite{DHV2022}
or restricting to $\alpha\le 1$ as in~\cite{GNR2022}, it is possible to obtain an epsilon-regularity theorem in the same spirit as the one for minimal surfaces.
A major difference between our setting and~\cite{DHV2022,GNR2022} is that in the charged liquid drop model, it is possible to show that for smooth
sets the nonlocal term is actually a volume term (while for us it penalizes curvature). Another
example studied in~\cite{MS2019} and which is strongly related to~\cref{cminpb},
is formally~\cref{probGamow} with the Riesz kernel but for $n=2$ and $\alpha=-1$. This is 
motivated by dipolar repulsion. In order to make the model meaningful, a small-scale cut-off has to be introduced (otherwise the energy is always infinite).
This cut-off plays a similar role in that model as our parameter $\eps$.  In particular, just like in our case, in the limit of vanishing cut-off length and after proper renormalization,
the nonlocal term converges to the perimeter. Among many other things, it is shown in~\cite{MS2019} that as in~\cref{mainthm:minball}, in the sub-critical regime (in our language $\gamma<1$)
minimizers are disks for small but finite cut-off lengths. Just like in our problem, the main issue
in~\cite{MS2019} is to obtain regularity estimates which hold at a macroscopic scale.
However, our strategy to obtain these estimates is very different from~\cite{MS2019}. Indeed, while we propagate regularity from the macroscopic scale down to the microscopic scale (in the form of excess decay), 
\cite{MS2019} relies on  the Euler--Lagrange equation to bootstrap the regularity from the microscopic scale up to the macroscopic scale.
Let us point out that on the one hand, the proof in~\cite{MS2019} does not seem to be easily adapted to dimensions higher than two and
that on the other hand the logarithmic scaling in~\cite{MS2019} allows to directly pass (in our notation) from a scale $r\gg \eps$ to a scale $r\ll \eps$. 
We refer to~\cite{CN2020a,KS2023,DKP2022} for other related models where however rigidity of the ball has not been investigated. \\      

In conclusion, besides~\cite{MS2019} which concerns a two dimensional problem,~\cref{mainthm:holdreg} is the first uniform regularity theorem for quasi-minimizers of a functional built with two competing local/nonlocal perimeters which remain asymptotically of the same order.\\ 
One may compare our regularity result with the one of \cite{CV2013}. Therein, the authors establish a uniform $C^{1,\alpha}$-regularity result for local
minimizers of the $s$-perimeter which is uniform in $s$ as $s\to 1^-$.
However, due to the lack of a competing term, the problem and its analysis are rather different from
the ones of the present work.

\subsection*{Outline of the paper}

The structure of the paper is as follows. In~\cref{sec:prelim}, we recall and prove a few facts
about nonlocal perimeters as well as some useful results from~\cite{Peg2021,MP2021} on minimizers of
\cref{cminpb}. We then show that minimizers of~\cref{cminpb} are almost-minimizers of
$\calF_{\eps,\gamma}$, and establish uniform density estimates for almost-minimizers of
$\calF_{\eps,\gamma}$. Eventually, we recall some basic properties of the excess and argue that
almost-minimizers satisfy the height bound property.  In~\cref{sec:lipapprox} we prove the Lipschitz
approximation theorem at scales much larger than $\eps$
(\cref{thm:lipapprox1,thm:lipapprox2}).
In~\cref{sec:caccio}, we establish the Caccioppoli inequality for $(\Lambda,r_0)$-minimizers of
$\calF_{\eps,\gamma}$.
Finally, building upon~\cref{sec:lipapprox,sec:caccio},~\cref{sec:reg} is devoted to establishing
power decay of the excess from large scales down to arbitrarily small scales.

\subsection*{Notation.}\label{subsec:notation}~

We write any point $x\in \R^n$ as $x=(x',x_n)$.
We denote by $B_r(x)\subsq \R^n$ the open ball of radius $r$ in~$\R^n$
centered at $x$. When $x=0$ we simply write $B_r$ for $B_r(0)$.  For open balls in $\R^{n-1}$, we
write $D_r(x')$ and $D_r$ when $x'=0$.
For any $m\in\N$, $\om_m$ denotes the $m$-volume of the unit ball in $\R^m$, that is, its
Lebesgue measure in $\R^m$.\\
For any set $E\subsq\R^n$, we denote its complement by $E^\compl\coloneqq \R^n\setminus E$ and
by $\abs{E}$ its volume whenever it is measurable.
For any $m\in\N$ we denote by $\calH^m$ the $m$-dimensional Hausdorff measure in $\R^n$. When
integrating with respect to the measure $\calH^m$ in a variable $x$, we use the compact notation $\dH_x^m$
instead of the standard $\dH^m(x)$. If $A$ is of dimension $m$ and $f$ is
$\calH^m$-measurable, we may simply write
\[
 \int_A f\coloneqq \int_A f(x) \dH^m_x.
\]
Similarly, we sometimes use the notation $f_x\coloneqq f(x)$.\\
Recall the definition~\cref{eq:cyl} of the truncated cylinder
\[
\cy(x,r,\nu)=x+\Big\{ y+t\nu~:~ y\in \nu^\perp
\text{ with } \abs{y}<r \text{ and } -r<t<r\Big\}.
\]
When $\nu=e_n$ (the $n$-th vector of the canonical basis of $\R^n$) we write
$\cy_r(x)$ for the cylinder $\cy(x,r,\nu)$ (recall~\cref{eq:cyl}) and simply $\cy_r$ if in addition $x=0$.
We also write $\e_n(E,x,r)$ for $\e(E,x,r,e_n)$ (recall~\cref{eq:defcylexcess}) and for $x=0$, $\e_n(E,r)=\e_n(E,0,r)$.

\vspace{2mm}
In addition to the excess~\cref{eq:defcylexcess}, we introduce in~\cref{sec:caccio} other quantities.
We also gather their definitions here for the convenience of the reader.\\
The flatness of $\partial E$ in the cylinder $\cy(x,r,\nu)$ with $x\in\partial E$ is defined by
\[
\f(E,x,r,\nu) \coloneqq \inf_{c\in\R}~ \frac{1}{r^{n-1}}\int_{\partial^* E\cap\cy(x,r,\nu)}
\frac{\abs{(y-x)\cdot\nu-c}^2}{r^2} \dhn_y.
\]
When $\nu=e_n$, we write $\f_n(E,x,r)$ for $\f(E,x,r,e_n)$ and we write $\f_n(E,r)$ for
$\f_n(E,0,r)$.\\
We denote by $\ky_t(z)$ the cylinder $D_t(z)\times\left(-1,1\right)$, and we write $\ky_t$ when $z=0$.
For any cylinder $\ky_t(z)$, any set of locally finite perimeter $E$, and any constant $c\in \R$, we define
the quantities
\[
\scrF(E,\ky_t(z),c)\coloneqq \int_{\ky_t(z)\cap\partial^* E} \frac{(x_n-c)^2}{t^2}\dhn
\]
and
\[
\scrE(E,\ky_t(z))\coloneqq P(E;\ky_t(z))-\calH^{n-1}(D_t(z)).
\]
When $z=0$, we write $\scrE(E,t)=\scrE(E,\ky_t(0))$ and $\scrF(E,t,c)=\scrE(E,\ky_t(0),c)$.
Recalling the definition of $Q$ in~\cref{eq:defQ}, we shall also make use of the
function~$Q_{1-\theta}$, with $\theta\in[0,1]$, defined by
\[
Q_{1-\theta}(t)\coloneqq Q(t^{1-\theta}),\qquad\forall t>0.
\]

\addtocontents{toc}{\protect\setcounter{tocdepth}{2}}
\section{Preliminary}\label{sec:prelim}

\subsection{Nonlocal perimeter and first variation}

In this section we recall some basic properties of the nonlocal perimeter depending on our
assumptions on $G$.
The following proposition is a consequence of~\cite{Dav2002} and our choice of $I_G^1$. It ensures
that $P_\eps$ is well-defined on sets of finite perimeter and is bounded from above by the standard
perimeter. We also state it for a general kernel $K$, not necessarily normalized, since we will
often use it with other kernels than $G$. We use the notation
\[
 \Per_{K}(E)\coloneqq \iint_{\Rpn\times\Rpn} \abs{\ind_E(x)-\ind_E(y)} K(x-y)\dx\dy.
\]

\begin{prp}[Upper bound]\label{prp:boundpg}
Assume that $K:\Rpn\to [0,\infty)$ satisfies~\cref{Hposrad} and $x\mapsto\abs{x}K(x)\in L^1(\R^n)$.
Then, for every set of finite perimeter $E$ in $\R^n$, we have
\begin{equation}\label{boundpk}
\Per_{K}(E)\le \mathbb{K}_{1,n}I_K^1 P(E).
\end{equation}
In particular, for the kernels $G_\eps$, we have
\begin{equation}\label{boundpg}
\Per_{\eps}(E)\le P(E),\qquad\forall \eps>0.
\end{equation}
\end{prp}

Let us recall that $\Per_\eps$ is continuous with respect to the $L^1$ topology along sequences with
bounded perimeter.

\begin{lem}[Continuity]\label{lem:contpg}
Assume that $G$ satisfies~\cref{Hposrad,,Hmoment1}.
Let $E_k$ be a sequence of sets of finite perimeter in $\R^n$ and $E\subsq\R^n$ such that
\[
\sup_k~ P(E_k)<\infty
\qquad\text{and}\qquad
E_k \xrightarrow{L^1} E.
\]
Then, for every $\eps>0$, we have
\[
\lim_k P_\eps(E_k) = P_\eps(E).
\]
\end{lem}

\begin{proof}
Let  $C\coloneqq \sup_k~ P(E_k)<\infty$.
Setting for $E\subsq \R^n$,
\[
u_k(z)\coloneqq \int_{\R^n} \abs{\ind_{E_k}(x+z)-\ind_{E_k}(x)}\dx
\qquad\text{ and }\qquad
u(z)\coloneqq \int_{\R^n} \abs{\ind_{E}(x+z)-\ind_{E}(x)}\dx,
\]
by the $L^1$ convergence of $E_k$ to $E$, for every $z\in\R^n$, $u_k$ converges to $u$ uniformly.
In addition, we have
\[
P_\eps(E_k)=\int_{\R^n} u_k(z)G_\eps(z)\dz
\]
and
\[
u_k(z)G_\eps(z)\le P(E_k)\abs{z}G_\eps(z)\le C\abs{z}G_\eps(z)\in L^1(\R^n).
\]
Hence by dominated convergence, $\lim_k P_\eps(E_k)=P_\eps(E)$.
\end{proof}

Depending on the integrability assumptions on $G$, we may estimate the difference
$P_\eps(E)-P_\eps(F)$ from above by a perimeter term, a volume term, or an interpolation of the two.
This type of estimates is relatively standard in the context of nonlocal perimeters (see for
instance~\cite[Lemma 5.3]{DNRV2015} for a similar statement in the case of $s$-perimeters).
The last interpolation estimate below will allow us to show a useful quasi-minimality property at
small scales for $(\Lambda,r_0)$-minimizers of $\calF_{\eps,\gamma}$
(see~\cref{prp:almostregsmallscale}). We will not use~\cref{eq:difPe2} but include it for
completeness.

\begin{lem}\label{lem:diffPeps}
Let $E,F\subsq\Rpn$ be two measurable sets with finite perimeter, and let $\eps>0$. We have:
\begin{enumerate}[(i)]
\item\label{diffPeps:i} if $G$ satisfies~\cref{Hposrad,Hmoment1}, then
\begin{equation}\label{eq:difPe}
P_\eps(E)-P_\eps(F) \le P_\eps(E\triangle F)\le P(E\triangle F);
\end{equation}
\item\label{diffPeps:ii} if $G$ satisfies~\cref{Hposrad} and $G\in L^1(\R^n)$, then  
\begin{equation}\label{eq:difPe2}
P_\eps(E)-P_\eps(F)\le \frac{2I_G^0}{\eps} \abs{E\triangle F};
\end{equation}	
\item\label{diffPeps:iii} if $G$ satisfies~\cref{Hposrad,Hmoment1,Hint0}, then there exists $C=C(n,G)>0$ such that
\begin{equation}\label{eq:difPe3}
P_\eps(E)-P_\eps(F)\le C\left(\frac{\abs{E\triangle F}}{\eps}\right)^{1-s_0}P(E\triangle F)^{s_0}.
\end{equation}
\end{enumerate}
\end{lem}

\begin{proof}
We decompose the proof in two steps.\\
\textit{Step 1. We establish $P_\eps(E)-P_\eps(F)\le P_\eps(E\triangle F)$.}
Let us note  for $A,B\subsq \R^n$,
\[
 \Phi_\eps(A,B)\coloneqq \iint_{A\times B} G_\eps(x-y) \dx \dy
\]
so that $P_\eps(E)=2\Phi_\eps(E,E^c)$. It is readily checked that
\[
 \begin{aligned}
  \Phi_\eps(E,E^c)-\Phi_\eps(F,F^c)
  &=\Phi_\eps(E\cap F,F\backslash E)+\Phi_\eps(E\backslash F, E^c\cap F^c)\\
  &\mathrel{\hphantom{=}}\hphantom{\Phi_\eps(E\cap F,F\backslash E)}-\Phi_\eps(E\cap F, E\backslash F) \hfill-\Phi_\eps(F\backslash E, F^c\cap E^c)\\
  &=\Phi_\eps(E\triangle F, (E\triangle F)^c)-2 \lt[\Phi_\eps(E\cap F, E\backslash F)+ \Phi_\eps(F\backslash E, F^c\cap E^c)\rt]\\
  &\le \Phi_\eps(E\triangle F, (E\triangle F)^c).
 \end{aligned}
\]
This concludes the first step.\medskip\\
\textit{Step 2. We deduce the different cases.}
Case~\cref{diffPeps:i} is a direct consequence of Step~1 and~\cref{boundpg}. If $G\in L^1(\R^n)$
then
\[
P_\eps(E)\le 2\norm{G_\eps}_{L^1(\R^n)}\abs{E}
\]
which gives~\cref{diffPeps:ii}.
For~\cref{diffPeps:iii}, let us write, for any $R>0$ and any $E\subsq \R^n$,
\[
P_{\eps}(E)
=\int_{\R^n\backslash B_R} G_\eps(z)\int_{\R^n} \abs{\chi_{E}(x+z)-\chi_{E}(x)}\dx\dz
+\int_{B_R}G_\eps(z)\int_{\R^n}\abs{\chi_{E}(x+z)-\chi_{E}(x)}\dx\dz.
\]
Using
\[
\int_{\R^n} \abs{\chi_{E}(x+z)-\chi_{E}(x)}\dx\le 2\abs{E}
\]
and
\[
\int_{\R^n} \abs{\chi_{E}(x+z)-\chi_{E}(x)}\dx\le \abs{z}P(E),
\]
we deduce
\begin{equation}\label{diffPeps:eq0}
\begin{aligned}
P_{\eps}(E)
&\le 2\abs{E}\int_{\R^n\backslash B_R} G_\eps(z)\dz
+P(E)\int_{B_R} \abs{z}G_\eps(z)\dz\\
&= \frac{2\abs{E}}{\eps}\int_{\R^n\backslash B_{R/\eps}} G(z)\dz
+P(E)\int_{B_{R/\eps}} \abs{z}G(z)\dz.
\end{aligned}
\end{equation}
Next, we claim that~\cref{Hint0} implies
\begin{equation}\label{diffPeps:claim1}
\int_{\R^n\setminus B_r} G(x)\dx \le \frac{C}{r^{s_0}},\qquad\forall r>0
\end{equation}
and
\begin{equation}\label{diffPeps:claim2}
\int_{B_r} \abs{x}G(x)\dx \le C r^{1-s_0},\qquad\forall r>0,
\end{equation}
for some $C=C(n,G)>0$.   It is of course enough to check these statements for either small or large~$r$. We start with~\cref{diffPeps:claim1}. Thanks to~\cref{Hint0}, it holds for small $r$. If instead  $r\ge 1$, 

\begin{equation*}\label{diffPeps:eq1}
\int_{\R^n\setminus B_r} G(x)\dx
\le \frac{1}{r} \int_{\R^n\setminus B_r} \abs{x}G(x)\dx
\le \frac{I^1_G}{r}
\le \frac{C}{r^{s_0}}.
\end{equation*}
We now turn to~\cref{diffPeps:claim2}. By~\cref{Hmoment1} it is enough to prove it for $r\in (0,1)$. In this case, we have
\begin{multline*}
\int_{B_r} \abs{x}G(x)\dx
=\sum_{k=0}^\infty \int_{B_{2^{-k}r}\setminus B_{2^{-(k+1)}r}} \abs{x}G(x)\dx
\le \sum_{k=0}^\infty \frac{r}{2^k} \int_{B_1\setminus B_{2^{-(k+1)}r}} G(x)\dx\\
\stackrel{\cref{diffPeps:claim1}}{\le} C\sum_{k=0}^\infty \frac{r}{2^k} \left(\frac{2^k}{r}\right)^{s_0}
= C r^{1-s_0}\sum_{k=0}^\infty \frac{1}{2^{k(1-s_0)}}
\le C r^{1-s_0},
\end{multline*}
 proving~\cref{diffPeps:claim2}.\\
Plugging~\cref{diffPeps:claim1,diffPeps:claim2} into~\cref{diffPeps:eq0} yields
\[
P_{\eps}(E)
\le C\lt(\frac{\abs{E}}{\eps}\left(\frac{\eps}{R}\right)^{s_0}
+P(E)\left(\frac{R}{\eps}\right)^{1-s_0}\rt).
\]
Finally choosing $R=\frac{\abs{E}}{P(E)}$, we get
\[
P_{\eps}(E)
\le C\left(\frac{\abs{E}}{\eps}\right)^{1-s_0}P(E)^{s_0}.
\]
 This concludes the proof of~\cref{diffPeps:iii}.
\end{proof}

We will use the following computation from~\cite[Lemma~2.3]{MP2021} to estimate the derivative
of the nonlocal perimeter under rescaling.

\begin{lem}\label{lem:derivpg}
Assume that $G$ satisfies~\cref{Hposrad,,Hmoment1,,Hmoment2}.
Then, for any set of finite perimeter $E\subsq\R^n$, the function $t\mapsto P_\eps(tE)$
is locally Lipschitz continuous in $(0,+\infty)$, and for almost every $t$, we have
\[
\begin{aligned}
\frac{\dd}{\dt} \left[P_\eps(tE)\right]
&= \frac{n}{t}P_\eps(tE)-\frac{1}{t}\wt{P}_{\eps}(tE),
\end{aligned}
\]
where $\wt{P}_{\eps}(E)$ is defined by 
\[
\wt{P}_{\eps}(E):= 2\int_E \int_{\partial^* E} G_\eps(x-y)\,(y-x)\cdot\nu_E(y)\dH^{n-1}_y\dx.
\]
\end{lem}

We now compute the first variation of the energy.

\begin{lem}\label{lem:firstvarf}
Assume that $G$ satisfies~\cref{Hposrad,,Hmoment1,,Hmoment2}.
Let $T\in C^1_c(\Rpn;\Rpn)$ be a compactly supported vector field, and let us define $f_t\coloneqq
\id_{\Rpn}+tT$.  Then for any set of finite perimeter $E\subsq\Rpn$, $\eps>0$,
$\gamma\in(0,1)$ and $\Lambda\ge 0$, the function $t\mapsto \calF_{\eps,\gamma}(f_t(E))$ is differentiable at $t=0$ with $\delta\calF_{\eps,\gamma}(E)[T]
\coloneqq \left[\ddt\calF_{\eps,\gamma}(f_t(E))\right]_{|t=0}$  given by
\[
\begin{aligned}
&\delta\calF_{\eps,\gamma}(E)[T]
=\int_{\partial^* E} \dvg_E T\dhn\\
&\qquad -2\gamma\left(\iint_{E\times E^\compl} \dvg T(x)G_\eps(x-y)
\dx\dy+\int_{\partial^* E}\int_E G_\eps(x-y)\left(T(x)-T(y)\right)\cdot\nu_E(y)\dx\dhn_y\right)
\end{aligned}
\]
where $\dvg_E T$ is the boundary divergence of $T$ on $E$, defined by
\[
\dvg_E T(x)\coloneqq \dvg T(x)-\nu_E(x)\cdot\grad T(x)\nu_E(x),
\quad\forall x\in\partial^* E.
\]
\end{lem}
\begin{proof}
 Since the computation of the first variation of the perimeter is standard, see
 e.g.~\cite[Theorem~17.5]{Mag2012}, we only compute the first variation of $P_\eps$.  We show that
 (recall the notation~$T_x=T(x)$),
 \begin{multline*}
\left[\ddt\Per_\eps(f_t(E))\right]_{|t=0}
= 2\iint_{E\times E^\compl} \dvg T(x)G_\eps(x-y)\dx\dy\\
\qquad+2\int_{\partial^* E}\int_E G_\eps(x-y)\left(T_x-T_y\right)\cdot\nu_E(y)\dx\dhn_y.
\end{multline*}
Notice that using~\cref{boundpg} and the fact that $T$ is Lipschitz continuous,~\cref{Hposrad,,Hmoment1} imply that both terms on the right-hand side are well-defined. Since $\eps$ does not play any role we may assume without loss of generality that $\eps=1$.  We set  $ F_G(t)\coloneqq \frac{1}{2} P_1(f_t(E)) $. Note that by choosing $t_0\le 1/\norm{\grad T}_{L^\infty}$, $f_t$ is a diffeomorphism of $\R^n$ for
every $t$ such that $\abs{t}< t_0$.
In particular $f_t(E)$ is a set of finite perimeter (see e.g.~\cite[Proposition~17.1]{Mag2012}).
Thus, $F_G(t)$ is well-defined for every $t\in(-t_0,t_0)$. We then set (for the moment this is just a notation) 
\[
F'_G(0)\coloneqq \iint_{E\times E^\compl} \dvg T(x)G(x-y)\dx\dy+\int_{\partial^* E}\int_E G(x-y)\left(T_x-T_y\right)\cdot\nu_E(y)\dx\dhn_y.
\]
We claim that as $t\to 0$,
\begin{equation}\label{claimderiv}
 F_G(t) -F_G(0)-t F_G'(0)=o(t).
\end{equation}
This would show that $F_G$ is differentiable in $0$ with derivative $F'_G(0)$, concluding the proof.    
 Changing variables, for any $t$ small enough we have
\[
F_G(t)
=\iint_{E\times E^c} G(f_t(x)-f_t(y)) \det Df_t(x)\,\det Df_t(y)\dx\dy.
\]
Since $\det D f_t (x)=1+ t \dvg T(x) +O(t^2)$, we find 
\[
 F_G(t)=\iint_{E\times E^c} G(f_t(x)-f_t(y))(1+t\dvg T(x)+ t\dvg T(y)+ O(t^2))\dx\dy.  
\]
Notice that    by the reverse change of variables and~\cref{boundpg}, 
\[
 \iint_{E\times E^c} G(f_t(x)-f_t(y))\dx \dy\le C\iint_{f_t(E)\times f_t(E)^c} G(x-y) \dx \dy\le C P(f_t(E))\le C P(E).
\]
Therefore 
\[
 F_G(t)=\iint_{E\times E^c} G(f_t(x)-f_t(y))(1+t\dvg T(x)+ t\dvg T(y))\dx\dy + O(t^2).
\]
Now, using that 
\[
 G(f_t(x)-f_t(y))-G(x-y)=t\int_0^1 \nabla G(f_{st}(x)-f_{st}(y))\cdot (T_x-T_y) \ds
\]
 and the Lipschitz continuity of $T$, we have 
\[
\begin{aligned}
 &\lt|\iint_{E\times E^c} G(f_t(x)-f_t(y))\dvg T(x)\dx\dy - \iint_{E\times E^c} G(x-y)\dvg T(x)\dx\dy\rt|\\
 &\qquad\le C\abs{t} \int_0^1 \iint_{E\times E^c}|\nabla G(f_{st}(x)-f_{st}(y))||x-y|\dx \dy \ds\\
 &\qquad\le C\abs{t} \int_0^1 \iint_{E\times E^c}|\nabla G(f_{st}(x)-f_{st}(y))||f_{st}(x)-f_{st}(y)|\dx \dy \ds\\
 &\qquad\le C\abs{t} \int_0^1 \iint_{f_{st}(E)\times (f_{st}(E))^c}|\nabla G(x-y)||x-y|\dx \dy \ds\\
 &\qquad\le C I_{|\nabla G|}^2 \abs{t} \int_0^1 P(f_{st}(E)) \ds\le C I_{|\nabla G|}^2 \abs{t},
\end{aligned}
\]
where we used again~\cref{boundpk} but for the kernel $K=|\cdot| |\nabla G|$. Since the same holds with $\dvg T(x)$ replaced by $\dvg T(y)$, in order to prove 
~\cref{claimderiv} it is thus enough to show   
\begin{multline}\label{eq:toproveclaimderiv}
 \iint_{E\times E^c} G(f_t(x)-f_t(y)) \dx \dy- \iint_{E\times E^c} G(x-y) \dx \dy \\
 + t\lt(\iint_{E\times E^c} G(x-y)\dvg T(y)\dx\dy- \int_{\partial^* E}\int_E G(x-y)\left(T_x-T_y\right)\cdot\nu_E(y)\dx\dhn_y\rt)=o(t).
\end{multline}
Writing as above that 
\begin{multline*}
\iint_{E\times E^c} G(f_t(x)-f_t(y)) \dx \dy- \iint_{E\times E^c} G(x-y) \dx \dy \\
 =t \int_0^1 \int_{E\times E^c} \nabla G(f_{st}(x)-f_{st}(y))\cdot (T_x-T_y)  \dx \dy \ds 
\end{multline*}
we reduce it further to the proof of 
\begin{equation}\label{reducproofvar1}
 \lim_{t\to 0}\int_0^1 \int_{E\times E^c} \nabla G(f_{st}(x)-f_{st}(y))\cdot (T_x-T_y)  \dx \dy \ds = \int_{E\times E^c} \nabla G(x-y)\cdot (T_x-T_y)  \dx \dy 
\end{equation}
together with the integration by parts formula 
\begin{multline}\label{reducproofvar2}
 \int_{E\times E^c} \nabla G(x-y)\cdot (T_x-T_y) +G(x-y)\dvg T(y) \dx \dy\\
 = \int_{\partial^* E}\int_E G(x-y)\left(T_x-T_y\right)\cdot\nu_E(y)\dx\dhn_y.
\end{multline}
Assuming~\cref{reducproofvar1,,reducproofvar2}, we see that~\cref{eq:toproveclaimderiv} holds true,
which proves the lemma.
Notice that these identities would be easy to prove if $G$ were a smooth kernel with compact support. However,
since our assumptions on $G$ seem too weak to prove these directly we argue by approximation. Let~$G_k$ be a sequence of smooth compactly supported radial kernels with 
\[
 \lim_{k\to \infty} \int_{\R^n} |z| |(G-G_k)(z)|\dz=0 \qquad\qquad \textrm{and} \qquad\qquad \lim_{k\to \infty} \int_{\R^n} |z|^2 |\nabla (G-G_k)(z)|\dz=0.
\]
Since we assumed that $I^1_G+I^2_{|\nabla G|}<\infty$ it is not difficult to construct such a sequence. We start with~\cref{reducproofvar1}. For every fixed $s\in [0,1]$, we have 
\[
 \begin{aligned}
  &\lt|\int_{E\times E^c} \nabla G(f_{st}(x)-f_{st}(y))\cdot (T_x-T_y)  \dx \dy-\int_{E\times E^c} \nabla G_k(f_{st}(x)-f_{st}(y))\cdot (T_x-T_y)  \dx \dy\rt|\\
  &\qquad\le C \int_{E\times E^c} |(\nabla G- \nabla G_k)( f_{st}(x)-f_{st}(y))| |x-y| \dx \dy\\
  &\qquad\le C \int_{f_{st}(E)\times f_{st}(E)^c} |\nabla (G-G_k)( x-y)| |x-y| \dx \dy\\
  &\qquad\le C \lt(\int_{\R^n} |z|^2 |\nabla (G-G_k)(z)| \dz\rt) P(f_{st}(E))\\
  &\qquad\le C \int_{\R^n} |z|^2 |\nabla (G-G_k)(z)| \dz,
 \end{aligned}
\]
where we used~\cref{boundpk} with $K= |\cdot| |\nabla (G-G_k)|$ (which is radially symmetric). Integrating in $s$ and using a simple diagonal argument, this proves~\cref{reducproofvar1}. We now turn to~\cref{reducproofvar2}. Since 
\[-\dvg_y( G(x-y)(T_x-T_y))=\nabla G(x-y)\cdot (T_x-T_y)+ G(x-y) \dvg T(y),\]
the integration by parts formula~\cref{reducproofvar2} holds with $G$ replaced by $G_k$. By the previous computations it is therefore enough to observe that on the one hand
\[
 \lt|\int_{E\times E^c} G(x-y)\dvg T(y) \dx \dy-\int_{E\times E^c} G_k(x-y)\dvg T(y) \dx \dy\rt|\le C P(E) \int_{\R^n} |z| |(G-G_k)(z)|\dz
\]
and on the other hand,
\begin{multline*}\lt|\int_{\partial^* E}\int_E G(x-y)\left(T_x-T_y\right)\cdot\nu_E(y)\dx\dhn_y-\int_{\partial^* E}\int_E G_k(x-y)\left(T_x-T_y\right)\cdot\nu_E(y)\dx\dhn_y\rt|\\\le C \int_{\partial^* E} \int_{E} |x-y||(G-G_k)(x-y)| \dx\dhn_y\le C P(E)  \int_{\R^n} |z| |(G-G_k)(z)|\dz.
\end{multline*}
This proves the lemma.
\end{proof}

\subsection{Perimeter quasi-minimizing properties of minimizers}

 We recall from~\cite[(4.2)]{Peg2021} that using~\cref{boundpg} it follows that if $E$
 satisfies $\calF_{\eps,\gamma}(E)\le \calF_{\eps,\gamma}(B_1)$, then 
\begin{equation}\label{eq:boundpermin}
 P(E)\le P(B_1) +\frac{\gamma}{1-\gamma} \lt(P(B_1)-P_\eps(B_1)\rt)\le \frac{1}{1-\gamma} P(B_1).
\end{equation}

We now use the scaling properties given in~\cref{lem:derivpg} to prove the equivalence between
\cref{cminpb} and the unconstrained minimization problem
\begin{equation}\label{minpb}\tag{$\mathcal{P}'$}
\min~ \Big\{ \calF_{\eps,\gamma}(E)+\Lambda\big|\abs{E}-\abs{B_1}\big|~:~ E\subsq\R^n \text{
measurable}\Big\}
\end{equation}
if $\Lambda$ is large enough, not depending on $\eps$. As a consequence, minimizers of
\cref{cminpb} are $(\Lambda,r_0)$-minimizers of~$\calF_{\eps,\gamma}$.

\begin{prp}\label{prp:equivpbs}
Assume that $G$ satisfies~\cref{Hposrad,Hmoment1} and let $\gamma\in(0,1)$. There exists $C=C(n)>0$ such that  for every
$\gamma\in(0,1)$, $\eps>0$ and $\Lambda\ge C/(1-\gamma)$,  problems~\cref{cminpb} and
\cref{minpb} are equivalent, in the sense that~\cref{minpb} admits a minimizer if and only if~\cref{cminpb}
does, and their minimizers coincide. In particular, any minimizer of~\cref{cminpb} is a
$(\Lambda,r_0)$-minimizer of $\calF_{\eps,\gamma}$ for any $\Lambda\ge C/(1-\gamma)$ and any
$r_0>0$.
\end{prp}

\begin{proof}
Let us set
\[
\Lambda_0 \coloneqq \frac{2}{1-\gamma}\left(1+\Big(n+\frac{2}{\mathbb{K}_{1,n}}\Big)\right)
\frac{P(B_1)}{\abs{B_1}}.
\]
Since 
\[
 \inf_{|E|=|B_1|} \calF_{\eps,\gamma}(E) \ge \inf_{E} \calF_{\eps,\gamma,\Lambda}(E),
\]
it is enough to prove that for $\Lambda\ge \Lambda_0$,  the converse inequality holds and that any
set minimizing $\calF_{\eps,\gamma,\Lambda}$ has measure $\omega_n$.
This in turn is equivalent to the claim that if $E$ is such that 
\[
\abs{E}\neq \om_n\qquad\qquad\text{ and }\qquad\qquad \calF_{\eps,\gamma,\Lambda}(E)\le \inf_{|E|=|B_1|} \calF_{\eps,\gamma}(E) 
\]
then $\Lambda< \Lambda_0$. Let  $E$ be such a set.  Recall that $E$ satisfies~\cref{eq:boundpermin}.
In the case $\abs{E} \leq \abs{B_1}/2$, testing the minimality of $E$ against $B_1$, we readily obtain
\[
\frac{\Lambda \om_n}{2} \leq \frac{P(B_1)}{1-\gamma},
\]
so that $\Lambda < \Lambda_0$.

Now, in the case $\abs{E}> \abs{B_1}/2$, let $\lambda>0$ be such that $|\lambda E|=|B_1|$. In particular, $\lambda^n>1/2$.
We then have 
\[
 P(E)-\gamma P_\eps(E)+\Lambda| |E|-|B_1||= \calF_{\eps,\gamma,\Lambda}(E)\le \calF_{\eps,\gamma}(\lambda E)= \lambda^{n-1} P(E)-\gamma P_\eps(\lambda E).
\]
Reorganizing terms we find 
\begin{equation}\label{proof:Lambda1}
 \Lambda \omega_n \lambda^{-n} |1-\lambda^n|\le (\lambda^{n-1}-1) P(E)+ \gamma |P_\eps(E)-P_\eps(\lambda E)|.
\end{equation}
By~\cref{lem:derivpg} and~\cref{boundpg}, for any 
$t>0$ we have
\[
\abs*{\ddt \left[P_\eps(tE)\right]}
\le \frac{1}{t}\left(nP_\eps(tE)+\abs{\wt{P}_{G_\eps}(tE)}\right)
\le \frac{1}{t}\left(nP(tE)+2P(tE)I_G^1\right)
\le t^{n-2}\left(n+\frac{2}{\mathbb{K}_{1,n}}\right),
\]
thus
\[
\abs{P_\eps(E)-P_\eps(\lambda E)}\le
\abs*{\int_{\lambda}^1 \ddt\left[P_\eps(tE)\right]\dt}
\le C_1 \abs{\lambda^{n-1}-1} P(E),
\]
where $C_1\coloneqq \left(n+\frac{2}{\mathbb{K}_{1,n}}\right)$. Inserting this into~\cref{proof:Lambda1} and using~\cref{eq:boundpermin}, this leads to 
\[
  \Lambda \omega_n \lambda^{-n}|1-\lambda^n|\le \frac{1}{1-\gamma}\lt(1 +  C_1\gamma \rt)P(B_1) |\lambda^{n-1}-1|.
\]
Since $1/2<\lambda^{-n}$ and $|\lambda^{n-1}-1|< |\lambda^n-1|$, we conclude that $\Lambda< \Lambda_0$.
\end{proof}

We recall the following elementary scaling properties of the energy which we will heavily use in the paper.

\begin{prp}\label{prp:rescaling}
For any set of finite perimeter $E$, any $\eps>0$ and any $r>0$ we have
\[
\calF_{\eps,\gamma}(E)=r^{n-1}\calF_{\eps/r,\gamma}(E/r).
\]
In particular $E$ is a $(\Lambda,r_0)$-minimizer of $\calF_{\eps,\gamma}$ if and only if $E/r$ is a
$(\Lambda r,\frac{r_0}{r})$-minimizer of $\calF_{\eps/r,\gamma}$.
\end{prp}

We now prove that $(\Lambda,r_0)$-minimizers of $\calF_{\eps,\gamma}$ are quasi-minimizers of the
perimeter and thus have density bounds which are uniform in $\eps$. 

\begin{prp}[Weak quasi-minimality]\label{prp:wquasimin}
Assume that $G$ satisfies~\cref{Hposrad,Hmoment1} and let $\gamma\in(0,1)$, $\eps>0$, $\Lambda>0$
and $r_0>0$ with $\Lambda r_0\le 1-\gamma$.
Then, for  any $(\Lambda,r_0)$-minimizer $E$ of $\calF_{\eps,\gamma}$ and every set $F$ with
$E\triangle F\csubset B_r(x)$ with $0< r \le r_0$ we have 
\begin{equation}\label{wquasimin:ineq}
P(E; B_r(x)) \le \frac{4}{1-\gamma} P(F; B_r(x)).
\end{equation}
As a consequence, there exists $C= C(n)>0$  such that for every $x\in \partial E$ and every $0<r\le
r_0$,
\begin{equation}\label{eq:densityestim}
\left(\frac{1-\gamma}{4}\right)^n \le \frac{\abs{E\cap B_r(x)}}{r^n}
\le 1-\left(\frac{1-\gamma}{4}\right)^n
\quad\text{ and }\quad
 \frac{(1-\gamma)^{n-1}}{C} \le \frac{P(E;B_r(x))}{r^{n-1}} \le \frac{C}{1-\gamma}.
\end{equation}
In particular, we have
\begin{equation}\label{equivtopoboundary}
\calH^{n-1}(\partial E\setminus \partial^* E)=0.
\end{equation}
\end{prp}

\begin{proof}
We only prove~\cref{wquasimin:ineq} since it is standard that weak quasi-minimality implies density
upper and lower bounds (see~\cite[Theorem~5.6]{Fus2015}), which then imply~\cref{equivtopoboundary}. To obtain the correct scaling in $\gamma$, one can repeat the proof in 
\cite[Theorem~21.11]{Mag2012}. By scaling and translation, we may assume that $r=1$ and $x=0$.
Testing the $(\Lambda,r_0)$-minimality of $E$ against $F$, we have 
\begin{align*}
P(E;B_1)
&~\le~\, P(F;B_1)+\gamma \left(P_\eps(E)-P_\eps(F)\right)+\Lambda\abs{E\triangle F}\\
&\stackrel{\mathclap{\cref{eq:difPe}\&\cref{boundpg}}\vphantom{\big|}}{~\le~}
\,P(F;B_1)+\gamma P(E\triangle F)+\Lambda\abs{E\triangle F}.
\end{align*}
We now argue as in~\cite[Remark~21.7]{Mag2012} and use  the isoperimetric
inequality to infer
\[
\abs{E\triangle F}
=\abs{E\triangle F}^{\frac{1}{n}}\abs{E\triangle F}^{1-\frac{1}{n}}
\le \frac{1}{n}P(E\triangle F).
\]
We thus find 
\[
 P(E;B_1)\le P(F;B_1)+\lt(\gamma+ \frac{\Lambda}{n}\rt)P(E\triangle F)\le P(F;B_1)+\lt(\gamma+ \frac{\Lambda}{n}\rt)(P(E;B_1)+P(F;B_1)).
\]
Rearranging terms and using that $\Lambda/n\le (1-\gamma)/2$ yields~\cref{wquasimin:ineq}.
\end{proof}
\begin{rmk}\label{rem:reducE}
 Thanks to~\cref{equivtopoboundary} when $E$ is a $(\Lambda,r_0)$-minimizer of $\calF_{\eps,\gamma}$
 with $\Lambda r_0\le 1-\gamma$, we do not distinguish anymore between $\partial E$ and
 $\partial^* E$ when integrating.
\end{rmk}

Under hypothesis~\cref{Hint0} we prove that $(\Lambda,r_0)$-minimizers of $\calF_{\eps,\gamma}$ are also almost-minimizers at scales which are small compared to $\eps$. 
\begin{prp}\label{prp:almostregsmallscale}
 Assume that $G$ satisfies~\cref{Hposrad,Hmoment1,Hint0}. Then there exists $C=C(n,G,\gamma)>0$ such
 that for every $\gamma\in(0,1)$, $\eps>0$, $\Lambda>0$ and $r_0>0$ with $\Lambda r_0\le 1-\gamma$,
 every $(\Lambda,r_0)$-minimizer $E$ of~$\calF_{\eps,\gamma}$ 
 and every set $F$ with $E\triangle F\csubset B_r(x)$ and $r\le r_0$, we have 
\[
P(E;B_r(x))
\le P(F;B_r(x))+\left(\frac{C}{\eps^{1-s_0}}\right)r^{n-s_0} +\Lambda |E\triangle F|.
\]
\end{prp}
\begin{proof}
We may assume that $P(F;B_r(x))\le P(E;B_r(x))$ otherwise there is nothing to prove. Arguing as
above using the $(\Lambda,r_0)$-minimality of $E$ we have 
 \begin{equation*}
\begin{aligned}
P(E;B_r(x))
&\le P(F;B_r(x))+\gamma\big(P_\eps(E)-P_\eps(F)\big)
+\Lambda |E\triangle F |\\
&\stackrel{\mathclap{\cref{eq:difPe3}}}{\le} P(F;B_r(x))+C\gamma\left(\frac{\abs{E\triangle F}}{\eps}\right)^{1-s_0}
P(E\triangle F)^{s_0} +\Lambda\abs{E\triangle F}
\end{aligned}
\end{equation*}
Then using $\abs{E\triangle F}\leq \abs{B_r(x)}=\om_n r^n$ and
\[
P(E\triangle F)\leq P(E;B_r(x))+P(F;B_r(x))\leq 2P(E;B_r(x))\leq C r^{n-1}
\]
by~\cref{eq:densityestim} yields the result.
\end{proof}
\begin{rmk}\label{rmk:regsmallscales}
~\cref{prp:almostregsmallscale} indeed yields classical almost-minimality for the perimeter at scales smaller than $\eps$ since letting $r= \eps \hat{r}$ and   $E=x+\eps \hat{E}$,
 we find for every $\hat{F}\triangle \hat{E}\csubset B_{\hat{r}}$,
 \[
  P(\hat{E};B_{\hat{r}})
\le P(\hat{F};B_{\hat{r}})+C \hat{r}^{n-s_0} +\Lambda \eps |\hat{E}\triangle \hat{F}|.
 \]

\end{rmk}

\subsection{Basic properties of the excess}

The cylindrical excess and spherical excess are respectively defined in~\cref{dfn:sphericalexcess}
and~\cref{dfn:cylindricalexcess}.
We refer to~\cite[Chapter~22.1]{Mag2012} for more details on the excess, however we recall two basic
properties that we use extensively in the rest of the paper.

\begin{prp}[Scaling properties]\label{prp:scaleexcess}
 Let $E\subsq\R^n$ be a set of finite perimeter, $x\in\partial E$, $\nu\in\bbsn$ and $0<r<R$. Then we
have
\[
\e(E,x,r,\nu) \le \left(\frac{R}{r}\right)^{n-1}\e(E,x,R,\nu)\qquad\text{ and }
\]
In addition, setting $E_{x,r}\coloneqq r^{-1}(E-x)$ we have
\[
\e(E_{x,r},0,1,\nu) =\e(E,x,R,\nu).
\]
\end{prp}

Note that this property holds for the spherical excess as well.

\begin{prp}[Changes of direction]\label{prp:changedir}
Let $\gamma\in(0,1)$ and $\eps>0$.  There exists $C=C(n,\gamma)>0$ such that for every
$(\Lambda,r_0)$-minimizer $E$ of $\calF_{\eps,\gamma}$ with $\Lambda r_0\le 1-\gamma$, every
$\nu,\nu_0\in\bbsn$, $x\in\partial E$ and $r>0$ such that $\sqrt{2} r \le r_0$, we have
\[
\e(E,x,r,\nu)\le C\left(\e(E,x,\sqrt{2}r,\nu_0)+\abs{\nu-\nu_0}^2\right).
\]
\end{prp}

The proof is identical to the one in~\cite[Proposition~22.5]{Mag2012} and relies only
on the density estimates for minimizers. Since it is very short, we write it for the reader's
convenience.

\begin{proof}
Using the inequality $\abs{\nu-\nu_{E}(y)}^2\le
2\abs{\nu-\nu_0}^2+2\abs{\nu_0-\nu_{E}(y)}^2$, and the facts that $\cy(x,r,\nu)\subsq
\cy(x,\sqrt{2}r,\nu_0)$ (recall the definition~\cref{eq:cyl}) and $\cy(x,r,\nu)\subsq B_{\sqrt{2}r}(x)$, we have
\[
\e(E,x,r,\nu)
\le \frac{2}{r^{n-1}}\int_{\partial E\cap C(x,\sqrt{2}r,\nu_0)}
\frac{\abs{\nu_0-\nu_E(y)}^2}{2}\dhn_y+ \frac{P(E;B_{\sqrt{2}r}(x))}{r^{n-1}}\abs{\nu-\nu_0}^2.
\]
The results follows from~\cref{eq:densityestim}.
\end{proof}

\subsection{The height bound}

Thanks to the density estimates of~\cref{prp:wquasimin}, $(\Lambda,r_0)$-minimizers of
$\calF_{\eps,\gamma}$ satisfy the same \enquote{height bound} property as quasi-minimizers of the
perimeter (see~\cite[Theorem~22.8]{Mag2012}). This property is a crucial tool for the Lipschitz approximation theorem and the Caccioppoli inequality.

\begin{prp}[The height bound]\label{prp:heightbound}
Let $\eps>0$, $\gamma\in(0,1)$, $\Lambda>0$ and $r_0>0$ with $\Lambda r_0\le 1-\gamma$.
There exist positive constants $\tau_{\rm height}=\tau_{\rm height}(n,\gamma)$ and $C=C(n,\gamma)$ such that the following holds.
For every $(\Lambda,r_0)$-minimizer $E$ of $\calF_{\eps,\gamma}$, every $x\in\partial E$,
$\nu\in\bbsn$ and $r>0$ with $2r\le r_0$, if
\[
\e_n(x,2r) < \tau_{\rm height},
\]
then
\[
\sup \Big\{\abs{x_n-y_n}~:~ (y',y_n)\in\partial E\cap \cy_r(x)\Big\}
\le Cr \e_n(x,2r)^{\frac{1}{2(n-1)}}.
\]
\end{prp}

\begin{proof}
As recalled by \Citeauthorsc{Mag2012}, the only step where the almost-minimality with respect to the
perimeter is used in the proof of~\cite[Theorem~22.8]{Mag2012} is to obtain the
\enquote{small-excess position} of Lemma~22.10 therein. In fact, this lemma holds as long as we have
density estimates on the perimeter for $E$, as  shown in~\cite[Lemma~7.2]{DHV2022}. Hence, thanks
to~\cref{eq:densityestim}, the same height bound holds for $(\Lambda,r_0)$-minimizers of
$\calF_{\eps,\gamma}$, whenever $\Lambda r_0\le 1-\gamma$ and $2r\le r_0$.
\end{proof}

\section{Lipschitz approximation theorem}\label{sec:lipapprox}

This section is devoted to the proof of the Lipschitz approximation theorem for
$(\Lambda,r_0)$-minimizers of $\calF_{\eps,\gamma}$, which is divided into two parts.
We first state that a small excess of an almost-minimizer $E$ in a cylinder implies that
the boundary of $E$ in that cylinder is almost entirely covered by the graph of a Lipschitz
function~$u$. Secondly, we state that the aforementioned function $u$ is close to a harmonic
function as long as the scale is much larger than $\eps$.

\subsection{Lipschitz approximation and harmonic comparison}

Since the first part of the Lipschitz approximation theorem relies only on standard properties on
the excess, the density estimates and the height bound, by~\cref{prp:wquasimin,prp:heightbound},
the proof can be reproduced almost verbatim from Steps~1 to~4 of the proof of
\cite[Theorem~23.7]{Mag2012}. 

\begin{thm}[Lipschitz approximation I]\label{thm:lipapprox1}
Assume that $G$ satisfies~\cref{Hposrad,,Hmoment1,,Hmoment2}.
Let $\eps>0$, $\gamma\in(0,1)$, $\Lambda>0$ and $r_0>0$ with $\Lambda r_0\le 1-\gamma$.
There exist positive constants $\tau_{\mathrm{lip}}=\tau_{\mathrm{lip}}(n,\gamma)$, $\delta_0=\delta_0(n,\gamma)$ and
$C=C(n,\gamma)$ such that the following holds. If $E$ is a $(\Lambda,r_0)$-minimizer of
$\calF_{\eps,\gamma}$ with $0\in\partial E$ and, for some $r$ such that $4\Lambda r\le r_0$,
\[
\e_n(4r) \le \tau_{\mathrm{lip}},
\]
then, setting
\[
M \coloneqq \partial E\cap \cy_{2r},
\]
 there exists a $\frac{1}{2}$-Lipschitz function $u:\R^{n-1}\to\R$ such that:
\begin{enumerate}[(i)]
\item\label{lipapprox1:i} $\norm{u}_{L^\infty}\le
Cr\e_n(4r)^{\frac{1}{2(n-1)}}<r/4$;\\[4pt]
\item\label{lipapprox1:iv} $\hn(M\triangle\Gamma_u)\le C \e_n(4r)r^{n-1}$;
\item\label{lipapprox1:v} $\displaystyle\frac{1}{r^{n-1}}\int_{D_{2r}} \abs{\grad u}^2 \le
C\e_n(4r)$.
\end{enumerate}
\end{thm}

We show that the function $u$ in the conclusion of~\cref{thm:lipapprox1} is almost a
solution to a nonlocal linear equation of the form $(\Delta-\gamma \Delta_{G_\eps}) u=0$ in $D_r$.

\begin{thm}[Lipschitz approximation II]\label{thm:lipapprox2}
There exists $C=C(n,\gamma, I^2_{\nabla G})>0$ such that under the same assumptions as~\cref{thm:lipapprox1}, the function $u$ satisfies for every $\vphi\in C^1_c(D_r)$.
\begin{multline*}
\frac{1}{r^{n-1}}
\left( \int_{D_r} \grad u\cdot\grad\vphi
-\gamma \iint_{D_{2r}\times D_{2r}} (u(x')-u(y'))(\vphi(x')-\vphi(y'))G_\eps(x'-y',0)
\dx'\dy'\right)\\
\le C \norm{\grad\vphi}_{L^\infty}
\left(\e_n(4r)+Q\left(\frac{r}{4\eps}\right)+\Lambda r\right).
\end{multline*}
\end{thm}

By scaling it is enough to prove~\cref{thm:lipapprox2} for $r=1$. Since the proof is quite long, we  postpone it to the next section and show first how it leads to a harmonic approximation result.

\begin{prp}[Harmonic approximation]\label{prp:harmapprox}
Let $\gamma\in(0,1)$ and assume that $G$ satisfies~\cref{Hposrad,Hmoment1}.
There exists $\eps_{\mathrm{harm}}\in(0,1)$ such that for every $\tau>0$, there exists
$\sigma=\sigma(n,G,\gamma,\tau)>0$ with the following property. If for some $\eps\in(0,\eps_{\mathrm{harm}})$,
$u\in H^1(D_2)$ satisfies
\[
\int_{D_2} \abs{\grad u}^2\le 1
\]
and, for all $\vphi\in C^1_c(D_{1})$,
\[
\abs*{\int_{D_1} \grad u\cdot\grad\vphi
-2\gamma\iint_{D_2\times D_2} (u(x')-u(y'))(\vphi(x')-\vphi(y'))G_\eps(x'-y',0)\dx'\dy'}
\le\norm{\grad\vphi}_{L^\infty} \sigma,
\]
then there exists a harmonic function $v$ on $D_1$ such that
\[
\int_{D_1} \abs{\grad v}^2\le 1\quad\text{ and }\quad
\int_{D_1} \abs{u-v}^2 \le\tau.
\]
\end{prp}

\begin{proof}

As there is no risk of confusion, to simplify the notation we use $x,y$ instead of $x',y'$ for
points in $\R^{n-1}$, and write $G_\eps(x)$ instead of $G_\eps(x',0)$.
Arguing by contradiction, let us assume that there exist vanishing sequences $(\eps_k)\subsq
(0,1)$ and $(\sigma_k)\subsq(0,1)$, a positive constant $\tau>0$ and a sequence $(u_k)\subsq H^1(D_2)$
such that the following holds:
\begin{enumerate}[(i)]
\item\label{harmapprox:eq0}
$\displaystyle\int_{D_2} \abs{\grad u_k}^2\le 1$ for all $k\in\N$;
\item\label{harmapprox:eq1} for every $k$ and every $\vphi\in C^\infty_c(D_1)$, we have
\[
\abs*{\int_{D_1} \grad u_k\cdot\grad\vphi
-2\gamma\iint_{D_2\times D_2} (u_k(x)-u_k(y))(\vphi(x)-\vphi(y))G_{\eps_k}(x-y)\dx\dy}
\le \sigma_k\norm{\grad\vphi}_{L^\infty};
\]
\item\label{harmapprox:eq2} there exist arbitrarily large $k$ such that there is no harmonic function $u$ on $D_1$ such that
\[
\int_{D_1} \abs{\grad u}^2\le 1\quad\text{ and }\quad
\int_{D_1} \abs{u_k-u}^2 \le\tau.
\]
\end{enumerate}
Without loss of generality, up to adding a constant to each $u_k$, one may assume that $\int_{D_2}
u_k=0$, so that by Poincaré--Wirtinger inequality, we have
\[
\int_{D_2} \abs{u_k}^2 \le C\int_{D_2} \abs{\grad u_k}^2 \le C,\qquad\forall k\in\N.
\]
In particular, $(u_k)$ is bounded in $H^1(D_2)$. Thus, up to extraction, there exists $u\in H^1(D_2)$ such that $u_k$ converges strongly to $u$ in $L^2(D_2)$ and
$\grad u_k$ converges weakly to~$\grad u$ in $L^2(D_2;\R^{n-1})$.
We claim that for every $\vphi\in C^\infty_c(D_1)$,
\begin{equation}\label{harmapprox:claim}
\lim_k \iint_{D_2\times D_2} (u_k(x)-u_k(y))(\vphi(x)-\vphi(y))G_{\eps_k}(x-y)\dx\dy
= \frac{1}{2}\int_{D_1} \grad u\cdot \grad\vphi,
\end{equation}
which we prove further below.
By the weak convergence of $\grad u_k$ to $\grad u$, the fact that $\gamma\neq 1$ and
\cref{harmapprox:eq1}, this implies
\[
\int_{D_1} \grad u\cdot\grad \vphi=0,\qquad\forall\vphi\in C^\infty_c(D_1);
\]
in other words, $u$ is harmonic.
By~\cref{harmapprox:eq0} and lower semicontinuity with respect to the weak $H^1$ convergence, we have
\begin{equation}\label{harmapprox:eq3}
\int_{D_1} \abs{\grad u}^2\le 1,
\end{equation}
and since $u_k$ converges to $u$ in $L^2(D_1)$, we have, for every $k$ large enough,
\[
\int_{D_1} \abs{u_k-u}^2 \le\tau.
\]
With~\cref{harmapprox:eq3}, this contradicts~\cref{harmapprox:eq2}.\\
We now prove~\cref{harmapprox:claim}.
Using the change of variable $z=x-y$, we have
\begin{equation}\label{harmapprox:eq4}
\begin{aligned}
&\iint_{D_2\times D_2} (u_k(x)-u_k(y))(\vphi(x)-\vphi(y))G_{\eps_k}(x-y)\dx\dy\\
&~=\int_0^1\int_0^1 \iint_{D_2\times D_2} \big(\grad u_k(x+t(y-x))\cdot (x-y)\big)
\big(\grad \vphi(x+s(y-x))\cdot (x-y)\big)
G_{\eps_k}(x-y)\dx\dy\ds\dt\\
&~=\int_0^1\int_0^1 \int_{\R^{n-1}} \int_{D_2}\ind_{D_2}(x+z)
\left(\grad u_k(x+tz)\cdot z\right)
\left(\grad \vphi(x+sz)\cdot z\right)
 G_{\eps_k}(z)\dx\dz\ds\dt.
\end{aligned}
\end{equation}
Let us set $g_\eps(r)\coloneqq \eps^{-(n+1)}g(\eps^{-1}r)$ for every $r>0$ (recall $G(x)=g(\abs{x})$
for every $x\in\R^n\setminus\{0\}$). Then for each $s,t\in(0,1)$ and each $x\in D_2$, using polar
coordinates, we have
\begin{multline}\label{harmapprox:eq5}
 \int_{\R^{n-1}} \ind_{D_2}(x+z)
\left(\grad u_k(x+tz)\cdot z\right)
\left(\grad \vphi(x+sz)\cdot z\right)
 G_{\eps_k}(z)\dz\\
\quad=\int_0^4 r^ng_{\eps_k}(r) \int_{\bbs^{n-2}}\ind_{D_2}(x+r\sigma)
\big(\grad u_k(x+tr\sigma)\cdot\sigma\big)\big(\grad \vphi(x+sr\sigma)\cdot\sigma\big)
\dH^{n-2}_\sigma\dr.
\end{multline}
Using the fact that for every $s,t\in(0,1)$ we have
$\abs{\grad\vphi(x+sr\sigma)-\grad\vphi(x+tr\sigma)}\le r\norm{D^2\vphi}_{L^\infty}$ and
Cauchy--Schwarz inequality, we deduce that for every $s,t\in(0,1)$ and every $\sigma\in\bbs^{n-2}$,
\begin{multline*}\label{harmapprox:eq6}
\left|\int_0^4 r^ng_{\eps_k}(r)\int_{D_2}\ind_{D_2}(x+r\sigma)
\big(\grad u_k(x+tr\sigma)\cdot\sigma\big)\big(\grad \vphi(x+sr\sigma)\cdot\sigma\big)
\dx\dr\right.\\
\left.-\int_0^4 r^ng_{\eps_k}(r)\int_{D_2}\ind_{D_2}(x+r\sigma)
\big(\grad u_k(x+tr\sigma)\cdot\sigma\big)\big(\grad
\vphi(x+tr\sigma)\cdot\sigma\big)\dx\dr\right|\\
\le C\norm{D^2\vphi}_{L^\infty}\left(\int_0^4 r^{n+1}g_{\eps_k}(r)\dr\right)\left(\int_{D_2} \abs{\grad
u}^2\right)^{\frac{1}{2}}.
\end{multline*}
Notice that
\[
\lim_{k\to\infty}~ \int_0^4 r^{n+1}g_{\eps_k}(r)\dr = 0
\]
since $r\mapsto r^ng(r)\in L^1(\R)$ and
\[
\begin{aligned}
\int_0^4 r^{n+1}g_{\eps_k}(r)\dr
&=\int_0^{4/\eps_k} (\eps_k r)r^{n}g(r)\dr\\
&=\int_0^{4/\sqrt{\eps_k}} (\eps_k r)r^{n}g(r)\dr
+\int_{4/\sqrt{\eps_k}}^{4/\eps_k} (\eps_k r)r^{n}g(r)\dr\\
&\le 4\sqrt{\eps_k}\int_0^\infty r^ng(r)\dr
+4\int_{4/\sqrt{\eps_k}}^{\infty} r^{n}g(r)\dr.
\end{aligned}
\]
Therefore, in view of~\cref{harmapprox:eq4,harmapprox:eq5}, in order to prove~\cref{harmapprox:claim}, we only  need to compute the limit of
\[
\begin{aligned}
&\int_0^1\int_0^4 r^ng_{\eps_k}(r) \int_{\bbs^{n-2}}\int_{D_2}
\ind_{D_2}(x+r\sigma) \big(\grad u_k(x+tr\sigma)\cdot\sigma\big)\big(\grad
\vphi(x+tr\sigma)\cdot\sigma\big) \dx\dH^{n-2}_\sigma\dr\dt\\
&\qquad=\int_0^1\int_0^4 r^ng_{\eps_k}(r)
\int_{\bbs^{n-2}}\int_{\R^{n-1}}\ind_{D_2}(y-tr\sigma)\ind_{D_2}(y+(1-t)r\sigma)
\big(\grad u_k(y)\cdot\sigma\big)\big(\grad \vphi(y)\cdot\sigma\big)\\
&\hphantom{\qquad=\int_0^1\int_0^4 r^ng_{\eps_k}(r)
\int_{\bbs^{n-2}}\int_{\R^{n-1}}\ind_{D_2}(y-tr\sigma)\ind_{D_2}(y+(1-t)r\sigma) \big(\grad
u_k(y)\cdot\sigma\big)}
\dy\dH^{n-2}_\sigma\dr\dt\\
&\qquad=\int_0^1\int_0^{\frac{4}{\eps_k}} r^n g(r)\int_{\bbs^{n-2}}\int_{D_1} 
\ind_{D_2}(y-t\eps_k r\sigma)\ind_{D_2}(y+(1-t)\eps_k r\sigma)
\big(\grad u_k(y)\cdot\sigma\big)\big(\grad \vphi(y)\cdot\sigma\big)\\
&\hphantom{\qquad=\int_0^1\int_0^4 r^ng_{\eps_k}(r)
\int_{\bbs^{n-2}}\int_{\R^{n-1}}\ind_{D_2}(y-tr\sigma)\ind_{D_2}(y+(1-t)r\sigma) \big(\grad
u_k(y)\cdot\sigma\big)}
\dy\dH^{n-2}_\sigma\dr\dt,
\end{aligned}
\]
where we  used a  change of variables and the fact that $\vphi\in C^\infty_c(D_1)$.
By the weak convergence of $\grad u_k$ to $\grad u$, for any $r>0$, $t\in(0,1)$ and
$\sigma\in\bbs^{n-2}$, we have
\[
\lim_k~\int_{D_1} \big(\grad u_k\cdot\sigma\big)\big(\grad \vphi\cdot\sigma\big)
=\int_{D_1} \big(\grad u\cdot\sigma)\big(\grad\vphi\cdot\sigma)
\]
and
\[
\begin{aligned}
&\left|\int_{D_1} \ind_{D_2}(y-t\eps_k r\sigma)\ind_{D_2}(y+(1-t)\eps_k r\sigma)
\big(\grad u_k(y)\cdot\sigma\big)\big(\grad \vphi(y)\cdot\sigma\big)\dy\right.
-\left.\int_{D_1} 
\big(\grad u_k\cdot\sigma\big)\big(\grad \vphi\cdot\sigma\big)\right|\\
&\qquad\le \int_{D_1\setminus\big(D_2(t\eps_k r)\cup D_2((1-t)\eps_k r\big)} \abs{\grad
u_k}\abs{\grad\vphi}\\
&\qquad\le \norm{\grad u_k}_{L^2(D_1)} \left(\int_{D_1\setminus\big(D_2(t\eps_k r)\cup D_2((1-t)\eps_k
r\big)} \abs{\grad\vphi}^2\right)^{\frac{1}{2}}
\xrightarrow{k\to\infty} 0,
\end{aligned}
\]
where we used the inequality $\norm{\grad u_k}_{L^2(D_1)}\le 1$ to pass to the limit. Thus,
for any $r>0$, $t\in(0,1)$ and~$\sigma\in\bbs^{n-2}$, we have
\[
\lim_k \int_{D_1}\ind_{D_2}(y-t\eps_k r\sigma)\ind_{D_2}(y+(1-t)\eps_k r\sigma)\
\big(\grad u_k(y)\cdot\sigma\big)\big(\grad \vphi(y)\cdot\sigma\big)\dy
=\int_{D_1} \big(\grad u\cdot\sigma\big)\big(\grad \vphi\cdot\sigma\big).
\]
Hence, using once more $\norm{\grad u_k}_{L^2(D_1)}\le 1$ and the Cauchy--Schwarz inequality,
applying the dominated convergence theorem yields
\begin{multline}\label{harmapprox:eq9}
\lim_k \int_0^1\hspace{-1pt}\int_0^{\frac{4}{\eps_k}} r^n g(r)\int_{\bbs^{n-2}}\int_{D_1} 
\ind_{D_2}(y-t\eps_k r\sigma)\ind_{D_2}(y+(1-t)\eps_k r\sigma)
\big(\grad u_k(y)\cdot\sigma\big)\big(\grad \vphi(y)\cdot\sigma\big)\dy\dH^{n-2}_\sigma\dr\dt\\
=\int_0^\infty r^n g(r)\int_{\bbs^{n-2}}\int_{D_1} 
\big(\grad u(y)\cdot\sigma\big)\big(\grad \vphi(y)\cdot\sigma\big)\dy\dH^{n-2}_\sigma\dr.
\end{multline}
This concludes the proof of~\cref{harmapprox:claim} in view of the normalization~\cref{fixedmoment1}
and the following identity. For every $x,y\in\R^{n-1}$, there holds
\begin{multline*}
\int_{\bbs^{n-2}} (x\cdot\sigma)(y\cdot\sigma)\dH^{n-2}_\sigma
=x\cdot  \left(\int_{\bbs^{n-2}} \sigma\otimes\sigma \dH^{n-2}_\sigma\right) y\\
=\lt(\int_{\bbs^{n-2}}\abs{\sigma_1}^2 \dH^{n-2}_\sigma\right) x\cdot y
= \frac{1}{2} \lt(\int_{\bbs^{n-1}} |\sigma_1| \dhn_\sigma\rt) x\cdot y,
\end{multline*}
where the last equality comes from a direct computation (see~\cite[Lemma~3.13]{Peg2021}).
\end{proof}

\subsection{Proof of \texorpdfstring{\cref{thm:lipapprox2}}{the Lipschitz approximation theorem II}}
\label{subsec:prooflipapprox2}

We start by \enquote{localizing} the Euler--Lagrange equation implied by the
$(\Lambda,r_0)$-minimality condition and the first variation of $\calF_{\eps,\gamma}$ given by
\cref{lem:firstvarf}.

\begin{lem}\label{lem:lipapprox2_step1}
Under the assumptions of~\cref{thm:lipapprox1}, there exists $C=C(n,\gamma)>0$ such that for every
$\vphi\in C^1_c(D_{1})$, denoting also $\vphi$ the function $\R^n\ni (x',x_n)\mapsto \vphi(x')$  (with a
slight abuse of notation),
\begin{multline}\label{lipapprox2:resstep1}
\Bigg|\int_{\partial E\cap \cy_2}  (\grad \vphi \cdot \nu_E)(\nu_{E}\cdot e_n)
+2\gamma \int_{\partial E\cap \cy_{2}} \int_{E\cap \cy_2}
G_\eps(x-y)(\vphi(x)-\vphi(y)) (\nu_{E}(y)\cdot e_n)  \dx\dhn_y \Bigg|\\
\le C \lt(Q\left(\frac{1}{4\eps}\right)+ \Lambda\rt)\norm{\grad
\vphi}_{L^\infty}.
\end{multline}
\end{lem}

\begin{proof}
By~\cref{prp:heightbound} we may choose $\tau_{\mathrm{lip}}=\tau_{\mathrm{lip}}(n,\gamma)$ small
enough so that
\begin{equation}\label{lipapprox2_step1:Eflat}
\left\{  x_n<-\frac{1}{4}\right\}\cap \cy_{2}
\ \subsq\ E\cap \cy_{2}
\ \subsq\ \left\{  x_n<\frac{1}{4}\right\}\cap \cy_2.
\end{equation}
To simplify notation we write $\nu$ for $\nu_E$ and recall the convention $T_x$ for $T(x)$. We may assume without loss of generality that $\norm{\grad
\vphi}_{L^\infty}=1$. We start with the following simple observation.
For every measure~$\mu$ and every sets $A, B$ such that  $A\times B\subsq \{(x,y)\in \R^{n}\times \R^n \, : \ |x-y|> 1/4\}$, 
\begin{equation}\label{eqmu}
 \int_{A\times B} G_\eps(x-y) \mathrm{d}\mu(x) \dy\le  4\mu(A) \int_{\R^n\backslash B_{\frac{1}{4}}} |z| G_\eps(z) \dz= 4\mu(A) Q\lt(\frac{1}{4\eps}\rt).
\end{equation}
 Let now $\alpha\in C^1_c((-1,1);[0,1])$ be such that  $\alpha\equiv 1$ in $(-\frac{1}{2},\frac{1}{2})$ and
$\norm{\alpha'}_{L^\infty}\le 4$. We then consider the vector field $T\in C^1_c(\cy_{1})$
defined by $T(x)=\vphi(x')\alpha(x_n)e_n$ for all $x\in\Rpn$. We first claim that 
\begin{equation}\label{lipapprox2:resstep1claim}
\Bigg|\int_{\partial E} \nu\cdot(\grad T\nu)
+2\gamma \int_{\partial E\cap \cy_{2}} \int_{E\cap \cy_2}
G_\eps(x-y)(T_x-T_y) \cdot\nu_y \dx\dhn_y \Bigg|
\le C\lt(Q\left(\frac{1}{4\eps}\right)+ \Lambda\rt).
\end{equation}
This would conclude the proof of~\cref{lipapprox2:resstep1} since $T(x)=\vphi(x') e_n$ in $D_{2}\times
(-\frac{1}{2},\frac{1}{2})$ and 
\begin{multline*}
 \lt|\int_{\partial E\cap \cy_2}\int_{E\cap\big(\cy_2\backslash\left(D_{2}\times
(-\frac{1}{2},\frac{1}{2})\right)\big)}
G_\eps(x-y)(T_x-\vphi_{x'} e_n) \cdot\nu_y \dx\dhn_y\rt|\\
\le C \int_{\partial E\cap \cy_2}\int_{E\cap\big(\cy_2\backslash\left(D_{2}\times
(-\frac{1}{2},\frac{1}{2})\right)\big)}
G_\eps(x-y) \dx\dhn_y\stackrel{\cref{eqmu}}{\le} C Q\lt(\frac{1}{4\eps}\rt),
\end{multline*}
where we used~\cref{eqmu} with $\mu= \hn\LL \partial E$ and the fact that  $P(E; \cy_2)\le C$ by
\cref{eq:densityestim}.\\
We thus prove~\cref{lipapprox2:resstep1claim}. Notice that $\dvg T(x)=\vphi(x')\alpha'(x_n)$.  In particular, by~\cref{lipapprox2_step1:Eflat}, $\dvg T$ vanishes in
\[
(\partial E\cap \cy_1)\cup \lt(E\cap \lt\{
x_n\le -1 \text{ or }x_n\ge -\frac{1}{2}\rt\}\rt).
\]
By $(\Lambda,r_0)$-minimality of $E$, setting $f_t(x)=x+tT(x)$ we have
\begin{equation}\label{lipapprox2:eq0}
\calF_{\eps,\gamma}(E)\le \calF_{\eps,\gamma}(f_t(E))+\Lambda\abs{E\triangle f_t(E)}.
\end{equation}
On the one hand, it is standard that for any $\abs{t}$ small enough
\[
\abs{E\triangle f_t(E)}\le 2\abs{t}\abs*{\int_{\partial E} T\cdot\nu}
\,\stackrel{\mathclap{\cref{eq:densityestim}}\vphantom{\big|}}{~\le~}\, C\abs{t}.
\]
On the other hand, for any $\abs{t}$ small enough, we have
\[
\calF_{\eps,\gamma}(f_t(E))
\le \calF_{\eps,\gamma}(E)+t\big(\delta\calF_{\eps,\gamma}(E)[T]\big)+o(t).
\]
Hence, by~\cref{lem:firstvarf},~\cref{lipapprox2:eq0} implies, for any $\abs{t}$ small enough
\begin{multline*}
-t 
\left[\int_{\partial^* E} \dvg_E T\dhn
-2\gamma\left(\iint_{E\times E^\compl} \dvg T(x)G_\eps(x-y) \dx\dy\right.\right.\\
\left.\left.+\int_{\partial^* E}\int_E
G_\eps(x-y)\left(T(x)-T(y)\right)\cdot\nu_E(y)\dx\dhn_y\right)\right]
\le C\abs{t}(\Lambda+o(1)).
\end{multline*}
Since this holds for $\pm t$ and for arbitrary small $\abs{t}$, in terms of $T$ this gives
\begin{multline*}
\left|\int_{\partial E} \nu\cdot(\grad T \nu)
+2\gamma\left(\int_{E\cap \cy_1\cap \{x_n\le-\frac{1}{2}\}} \int_{E^\compl} \dvg
T(x)G_\eps(x-y)\dy\dx\right.\right.\\
\left.\left.+ \int_{\partial E} \int_{E} G_\eps(x-y)(T_x-T_y) \cdot\nu_y \dx\dhn_y \right)\right|
\le C\Lambda.
\end{multline*}
Using again that  $P(E; \cy_2)\le C$ by~\cref{eq:densityestim} and~\cref{eqmu} with $A=E\cap
\cy_1\cap \{x_n\le-\frac{1}{2}\}$, $B=E^c$ and $\mu$ the Lebesgue measure (recall that $E^c\cap
\cy_2\subsq \{(x',x_n)~:~ x_n\ge -\frac{1}{4}\}$) we see that in order to prove
\cref{lipapprox2:resstep1claim} it is enough to show that
\begin{multline}\label{lipapprox2:claimreduced}
 \lt|\int_{\partial E} \int_{E} G_\eps(x-y)(T_x-T_y) \cdot\nu_y \dx\dhn_y\rt.\\
\lt.-\int_{\partial E\cap \cy_{2}} \int_{E\cap \cy_2}
G_\eps(x-y)(T_x-T_y) \cdot\nu_y \dx\dhn_y \rt|\le CQ\lt(\frac{1}{4\eps}\rt).
\end{multline}
Recalling that $T=0$ in $\cy_1^c$ we write
\begin{multline*}
 \int_{\partial E} \int_{E} G_\eps(x-y)(T_x-T_y) \cdot\nu_y \dx\dhn_y-\int_{\partial E\cap \cy_2} \int_{E\cap \cy_2} G_\eps(x-y)(T_x-T_y) \cdot\nu_y \dx\dhn_y\\
 = -\int_{\partial E\cap \cy_1} \int_{E\backslash \cy_2} G_\eps(x-y)T_y \cdot\nu_y \dx\dhn_y +\int_{\partial E\backslash \cy_2} \int_{E\cap \cy_1} G_\eps(x-y) T_x\cdot\nu_y \dx\dhn_y.  
\end{multline*}
Considering the last term on the right-hand side and using integration by parts we have 
\begin{multline*}
  \int_{\partial E\backslash \cy_2} \int_{E\cap \cy_1} G_\eps(x-y) T_x\cdot\nu_y \dx\dhn_y\\
  =\int_{E\cap \partial \cy_2}\int_{E\cap \cy_1} G_\eps(x-y) T_x\cdot\nu_{\cy_2}(y) \dx\dhn_y -\int_{E\backslash \cy_2}\int_{E\cap \cy_1} \nabla G_\eps(x-y)\cdot T_x \dx \dy.
  \end{multline*}
  Using Fubini and integration by parts again leads to 
  \begin{equation*}
  \begin{aligned}
  &\int_{E\backslash \cy_2}\int_{E\cap \cy_1} \nabla G_\eps(x-y)\cdot T_x \dx \dy\\
  &\qquad=\int_{E\cap \cy_1} \int_{E\backslash \cy_2} \nabla G_\eps(x-y)\cdot T_x \dy \dx\\
  &\qquad=\int_{\partial E\cap \cy_1} \int_{E\backslash \cy_2}  G_\eps(x-y) T_x\cdot \nu_x  \dy \dhn_x -\int_{E\cap \cy_1} \int_{E\backslash \cy_2} G_\eps(x-y) \dvg T(x) \dy \dx.
 \end{aligned}
\end{equation*}
Putting everything together we find
\begin{multline*}
 \int_{\partial E} \int_{E} G_\eps(x-y)(T_x-T_y) \cdot\nu_y \dx\dhn_y-\int_{\partial E\cap \cy_2} \int_{E\cap \cy_2} G_\eps(x-y)(T_x-T_y) \cdot\nu_y \dx\dhn_y\\
 = -2\int_{\partial E\cap \cy_1} \int_{E\backslash \cy_2} G_\eps(x-y)T_y \cdot\nu_y \dx\dhn_y+\int_{E\cap \partial \cy_2}\int_{E\cap \cy_1} G_\eps(x-y) T_x\cdot\nu_{\cy_2}(y) \dx\dhn_y\\
 +\int_{E\cap \cy_1} \int_{E\backslash \cy_2} G_\eps(x-y) \dvg T(x) \dy \dx.
\end{multline*}
Using~\cref{eqmu} with either $A=\partial E\cap \cy_1$, $B=E\backslash \cy_2$, $\mu= \hn\LL \partial E$ (and $P(E;\cy_1)\le C$), $A=E\cap \partial \cy_2$, $B=E\cap \cy_1$,
$\mu=\hn\LL \partial \cy_2$ or $A=E\cap \cy_1$, $B=E\backslash \cy_2$ and $\mu$ the Lebesgue measure we conclude the proof of~\cref{lipapprox2:claimreduced}.
\end{proof}

We now transfer this information to the graph of $u$.

\begin{lem}\label{lem:lipapprox2_step2}
Under the assumptions of~\cref{thm:lipapprox1}, there exists $C=C(n,\gamma)>0$ such that for every
$\vphi\in C^1_c(D_{1})$ we have
\begin{multline}\label{lipapprox2_step2:res}
\Bigg|\int_{\Gamma_u} (\nabla \vphi\cdot \nu_{E_u})(\nu_{E_u}\cdot e_n)+2\gamma \int_{\Gamma_u} \int_{E_u}
G_\eps(x-y)(\vphi(x)-\vphi(y)) (\nu_{E_u}(y)\cdot e_n) \dx\dhn_y \Bigg|\\
\qquad\le C \lt(\e_n(4)+Q\left(\frac{1}{4\eps}\right)\rt),
\end{multline}
where
\[
E_u \coloneqq \Big\{ (x',x_n)~:~ x'\in D_{2} \text{ and } x_n<u(x') \Big\}.
\]

\end{lem}

\begin{proof}
As above we may assume without loss of generality that $\norm{\grad
\vphi}_{L^\infty}=1$. To simplify notation we write $\nu$ for $\nu_E$ and $\nu^u$ for $\nu_{E_u}$ and will use the convention $T_x=T(x)$. We recall that $M=\partial E\cap \cy_2$. Since it is classical (see e.g. the proof of~\cite[Theorem 23.7]{Mag2012}) that 
\[
 \Bigg|\int_{\Gamma_u} (\nabla \vphi\cdot \nu_{E_u})(\nu_{E_u}\cdot e_n)-\int_{M}  (\grad \vphi \cdot \nu_E)(\nu_{E}\cdot e_n)\Bigg|\le C \e_n(4),
\]
from~\cref{lipapprox2:resstep1} it is enough to prove that 
\begin{multline}\label{eq:toprovelipapprox2_step2}
 \Bigg|\int_{\Gamma_u} \int_{E_u}
G_\eps(x-y)(\vphi_x-\vphi_y) (\nu^u_y\cdot e_n) \dx\dhn_y-\int_{M} \int_{E\cap \cy_2}
G_\eps(x-y)(\vphi_x-\vphi_y) (\nu_{y}\cdot e_n)  \dx\dhn_y \Bigg|\\
\qquad\le C  \e_n(4) \norm{\grad
\vphi}_{L^\infty}.
\end{multline}
To this aim we write 
\begin{multline*}
 \int_{\Gamma_u} \int_{E_u}
G_\eps(x-y)(\vphi_x-\vphi_y) (\nu^u_y\cdot e_n) \dx\dhn_y-\int_{M} \int_{E\cap \cy_2}
G_\eps(x-y)(\vphi_x-\vphi_y) (\nu_{y}\cdot e_n)  \dx\dhn_y \\
=\int_{\Gamma_u} \lt(\int_{E_u} G_\eps(x-y)(\vphi_x-\vphi_y) \dx-\int_{E\cap \cy_2} G_\eps(x-y)(\vphi_x-\vphi_y)\dx\rt)(\nu^u_y\cdot e_n)\dhn_y\\
+\int_{\Gamma_u} \int_{E\cap \cy_2}
G_\eps(x-y)(\vphi_x-\vphi_y) (\nu^u_y\cdot e_n) \dx\dhn_y\\
-\int_{M} \int_{E\cap \cy_2}
G_\eps(x-y)(\vphi_x-\vphi_y) (\nu_{y}\cdot e_n)  \dx\dhn_y.
\end{multline*}
We claim that 
\begin{equation}\label{eq:claimlip2step2_1}
\int_{\Gamma_u} \int_{E_u\triangle (E\cap \cy_2)} G_\eps(x-y)|\vphi_x-\vphi_y|\dx\dhn_y\le C \e_n(4)
\end{equation}
and 
\begin{multline}\label{eq:claimlip2step2_2}
 \lt|\int_{\Gamma_u} \int_{E\cap \cy_2}
G_\eps(x-y)(\vphi_x-\vphi_y) (\nu^u_y\cdot e_n) \dx\dhn_y\rt.\\
\lt.-\int_{M} \int_{E\cap \cy_2}
G_\eps(x-y)(\vphi_x-\vphi_y) (\nu_{y}\cdot e_n)  \dx\dhn_y\rt| \le C \e_n(4),
\end{multline}
from which~\cref{eq:toprovelipapprox2_step2} would follow. We start with
\cref{eq:claimlip2step2_1}.\\
By~\cref{lipapprox1:iv} of~\cref{thm:lipapprox1}, there exists a set $A\subsq D_{2}$ such that
$\calH^{n-1}(A)\le C\e_n(4)$ and
\[
E\cap \big\{(x',t)~:~ t\in(-2,2)\big\} = E_u\cap\big\{(x',t)~:~t\in(-2,2)\big\},
\quad\forall x'\in D_{2}\setminus A,
\]
since
\[
 \Big\{y'\in D_{2}~:~
\Pi_n^{-1}(\{y'\})\cap E\cap \cy_{2} \neq\Pi_n^{-1}(\{y'\})\cap
\Gamma_u^{-}\cap \cy_{2}\Big\}=\Pi_n(M\triangle \Gamma_u),
\]
where $\Pi_n: (y',y_n)\mapsto y_n$.
Thus, since $\vphi$  and $u$ are Lipschitz continuous we find
\begin{equation*}
 \begin{aligned}
  &\int_{\Gamma_u} \int_{E_u\triangle (E\cap \cy_2)} G_\eps(x-y)|\vphi_x-\vphi_y|\dx\dhn_y\\
  &\qquad \le \int_{\Gamma_u} \int_{A} \int_{-2}^2 G_\eps( (x',t)-y)|(x',t)-y| \dt \dx'\dhn_y\\
  &\qquad \le C \int_{A} \int_{D_2} \int_{\R} G_\eps( (x',t)-u_{y'})|(x',t)-u_{y'}| \dt \dy' \dx'\\
  &\qquad \le C \hn(A) \int_{\R^n} |z| G(z) \dz\le C \hn(A).
 \end{aligned}
\end{equation*}
We now turn to~\cref{eq:claimlip2step2_2}. Notice that $\hn$-a.e. on $\Gamma_u\cap M$ we have $\nu^u=\pm \nu$. Moreover, setting $\Gamma_1\coloneqq \Gamma_u\cap M\cap\{\nu^u=\nu\}$ and 
arguing exactly as in~\cite[(23.51)]{Mag2012}, we have $\hn((M\cap \Gamma_u)\backslash \Gamma_1)\le C \e_n(4)$. Recalling that by~\cref{lipapprox1:iv} of~\cref{thm:lipapprox1}, 
$\hn(M\triangle \Gamma_u)\le C \e_n(4)$, we find 
\begin{multline*}
   \lt|\int_{\Gamma_u} \int_{E\cap \cy_2}
G_\eps(x-y)(\vphi_x-\vphi_y) (\nu^u_y\cdot e_n) \dx\dhn_y\rt.\\
\lt.-\int_{M} \int_{E\cap \cy_2} G_\eps(x-y)(\vphi_x-\vphi_y) (\nu_{y}\cdot e_n)  \dx\dhn_y\rt|\\
\begin{aligned}
&\qquad\le \int_{(M\triangle \Gamma_u)\cup ((M\cap \Gamma_u)\backslash \Gamma_1)} \int_{E\cap \cy_2} G_\eps(x-y)|x-y| \dx \dhn_y\\
&\qquad\le C \e_n(4) \int_{\R^n} |z|G_\eps(z)\dz\le C\e_n(4).
 \end{aligned}
\end{multline*}
This proves the lemma.
\end{proof}

In order to conclude the proof of~\cref{thm:lipapprox2}, we are left with the linearization of~\cref{lipapprox2_step2:res}.

\begin{proof}[Proof of~\cref{thm:lipapprox2}]
Since arguing verbatim as in~\cite[Theorem 23.7]{Mag2012} we have 
\[
 \lt|\int_{D_1} \nabla u\cdot \nabla \vphi -\int_{\Gamma_u} (\nabla \vphi\cdot \nu_{E_u})(\nu_{E_u}\cdot e_n)\rt|\le C \e_n(4),
\]
by~\cref{lem:lipapprox2_step2} it is enough to prove that (recall the notation $u_{x'}=u(x')$)

\begin{multline}\label{lipapprox2:claimstep3}
\left|\int_{D_{2}}\int_{D_{2}} \int_{-2}^{u_{x'}}
G_\eps(x'-y',t-u_{y'})(\vphi_{x'}-\vphi_{y'})\dt\dx'\dy'\rt.\\
\lt.-\iint_{D_{2}\times D_{2}}(u_{x'}-u_{y'})(\vphi_{x'}-\vphi_{y'})G_\eps(x'-y',0)\dx'\dy'\right|\le C \lt(\e_n(4)+Q\lt(\frac{1}{4\eps}\rt)\rt).
\end{multline}
For $x'\neq y'$ we have
\begin{multline}\label{lipapprox2:eq15}
\int_{-2}^{u_{x'}} G_\eps(x'-y',t-u_{y'})\dt
=\int_{-2-u_{y'}}^{u_{x'}-u_{y'}} G_\eps(x'-y',s)\ds\\
\quad=\int_{-2-u_{y'}}^{-1} G_\eps(x'-y',s)\ds
+\int_{-1}^0 G_\eps(x'-y',s)\ds
+\int_0^{u_{x'}-u_{y'}} G_\eps(x'-y',s)\ds.
\end{multline}
On the one hand we observe that 
\begin{equation}\label{lipapprox2:eq16}
\iint_{D_{2}\times D_{2}} (\vphi_{x'}-\vphi_{y'}) \int_{-1}^0
G_\eps(x'-y',s)\ds\dx'\dy'
=0.
\end{equation}
On the other hand, since  $\norm{u}_{L^\infty}\le 1$, for any $y'$ we have $-2-u_{y'}<-1$. Thus,
using the fact that $\vphi$ is $1$-Lipschitz,  we compute
\begin{align}
\nonumber
&\abs*{\int_{D_2}\int_{D_2} \int_{-2-u_{y'}}^{-1}
G_\eps(x'-y',s)(\vphi_{x'}-\vphi_{y'})\ds\dx'\dy'}\\
\nonumber
&\phantom{\abs*{\int_{D_2}\int_{D_2} \int_{-2-u_{y'}}^{-1}}G_\eps(x'-y',s)}
\le \int_{D_2}  \int_{-2-u_{y'}}^{-1} \int_{D_2}
\abs{x'-y'} G_\eps(x'-y',s) \dx'\ds\dy'\\
\label{lipapprox2:eq17}
&\phantom{\abs*{\int_{D_2}\int_{D_2} \int_{-2-u_{y'}}^{-1}}G_\eps(x'-y',s)}
\le \int_{D_2}  \int_{\R^n\backslash B_{1}} 
\abs{z}G_\eps(z) \dz\dy'\le C Q\left(\frac{1}{4\eps}\right).
\end{align}
Combining~\cref{lipapprox2:eq15,,lipapprox2:eq16,,lipapprox2:eq17} yields
\begin{multline}\label{lipapprox2:eq18}
\left|\int_{D_2}\int_{D_2} \int_{-2}^{u_{x'}}
G_\eps(x'-y',t-u_{y'})(\vphi_{x'}-\vphi_{y'})\dt\dx'\dy'\right.\\
\left.-\iint_{D_2\times D_2} \int_0^{u_{x'}-u_{y'}}
G_\eps(x'-y',t)(\vphi_{x'}-\vphi_{y'}) \dt\dx'\dy'\right|\le CQ\left(\frac{1}{4\eps}\right).
\end{multline}
Using again  that $\vphi$ is $1$-Lipschitz and Fubini's theorem, we
estimate
\begin{align}
\nonumber
&\abs*{\iint_{D_2\times D_2}\int_0^{u_{x'}-u_{y'}}
(G_\eps(x'-y',t)-G_\eps(x'-y',0))(\vphi_{x'}-\vphi_{y'}) \dt\dx'\dy'}\\
\nonumber
&\phantom{\iint_{D_2\times D_2}\int_0^{u_{x'}}}
\le \int_0^1\iint_{D_2\times D_2} \int_0^{\abs{u_{x'}-u_{y'}}} t\,\abs{x'-y'}
\abs{\grad G_\eps(x'-y',st)} \dt\dx'\dy'\ds\\
\nonumber
&\phantom{\iint_{D_2\times D_2}\int_0^{u_{x'}}}
=\int_0^1 \int_0^{1}  t\iint_{D_2\times D_2} |u_{x'}-u_{y'}|^2 \,\abs{x'-y'}
\abs{\grad G_\eps(x'-y',st |u_{x'}-u_{y'}|)} \dx'\dy'\dt\ds\\
\label{lipapprox2:eq19}
&\phantom{\iint_{D_2\times D_2}\int_0^{u_{x'}}}
\le \int_0^1 \int_0^{1} \iint_{D_2\times D_2} |u_{x'}-u_{y'}|^2 \,\abs{x'-y'}
\abs{\grad G_\eps(x'-y',st |u_{x'}-u_{y'}|)} \dx'\dy'\dt\ds.
\end{align}
Set $\widetilde{G}_\eps\coloneqq \eps^{-(n+1)} \widetilde{G}(\cdot/\eps)$ where $\widetilde{G}\coloneqq |\cdot| |\nabla G|$ and  $\Phi_{stu}(x',y')\coloneqq(x'-y',st(u_{x'}-u_{y'}))$. 
Observing that $\abs{\Phi_{stu}(x',y')}\ge |x'-y'|$ we have for every fixed $s,t$,
\begin{multline*}
 \iint_{D_2\times D_2} |u_{x'}-u_{y'}|^2 \,\abs{x'-y'}
\abs{\grad G_\eps(x'-y',st |u_{x'}-u_{y'}|)} \dx'\dy'\\\le \iint_{D_2\times D_2} |u_{x'}-u_{y'}|^2 \widetilde{G}_\eps(\abs{\Phi_{stu}(x',y')}) \dx'\dy'.
\end{multline*}
Observing that $ I^1_{\widetilde{G}}=I^2_{|\nabla G|}$,~\cref{lem:techlem1} below yields 
\[
 \iint_{D_2\times D_2} |u_{x'}-u_{y'}|^2 \widetilde{G}_\eps(\abs{\Phi_{stu}(x',y')}) \dx'\dy'\le C I^2_{|\nabla G|} \int_{D_2} |\nabla u|^2\le C I^2_{|\nabla G|} \e_n(14),
\]
where we used that by~\cref{lipapprox1:v} of~\cref{thm:lipapprox1}, $\int_{D_2} |\nabla u|^2\le C \e_n(4)$. Combining this with~\cref{lipapprox2:eq19} and~\cref{lipapprox2:eq18} concludes the proof of~\cref{lipapprox2:claimstep3}.

\end{proof}

In the proof of~\cref{thm:lipapprox2} above, we used the following technical lemma.

\begin{lem}\label{lem:techlem1}
Let $G: \R^n\mapsto \R^+$ be a radial kernel such that (recall definition~\cref{defIk}) $
I_G^1 <\infty$.
For $u\in\Lip(D_2)$, we define the map~$\Phi_u:D_2\times D_2\to \R^{n-1}\times\R$ by
\[
\Phi_u(x',y')=(x'-y',u(x')-u(y')).
\]
There exists a  constant $C=C(n)>0$ such that if  $\norm{\grad u}_{L^\infty(D_2)}\le \frac{1}{2}$ then
\begin{equation}\label{techlem1:skernel}
\iint_{D_2\times D_2} (u(x')-u(y'))^2 G(\Phi_{u}(x',y'))\dx'\dy'
\le CI_G^1\int_{D_2} \abs{\grad u}^2.
\end{equation}
\end{lem}

\begin{proof}
 We start by estimating 
\[
\begin{aligned}
&\iint_{D_2\times D_2}(u(x')-u(y'))^2 G(\Phi_u(x',y')) \dx'\dy'\\
&\quad\le \int_0^1 \iint_{\R^{n-1}\times \R^{n-1}} \ind_{D_2}(x')\ind_{D_2}(y')
\abs{\grad u(x'+t(y'-x'))}^2 \abs{x'-y'}^2G(\Phi_u(x',y'))\dx'\dy'\dt\\
&\quad= \int_0^1 \iint_{\R^{n-1}\times \R^{n-1}} \ind_{D_2}(\hat{x}'-tz')\ind_{D_2}(\hat{x}'+(1-t)z')
\abs{\grad u(\hat{x})}^2 \abs{z'}^2G(\Phi_u(\hat{x}'-tz',\hat{x}'+(1-t)z'))\\
&\hphantom{\quad= \int_0^1 \iint_{\R^{n-1}\times \R^{n-1}} \ind_{D_2}(\hat{x}'-tz')\ind_{D_2}(\hat{x}'+(1-t)z')
\abs{\grad u(\hat{x})}^2 \abs{z'}^2G(\Phi_u(\hat{x}'-tz',\hat{x}'+(1-}
\dhx'\dz'\dt\\
&\quad\le \int_0^1 \int_{D_2} \abs{\grad u(x')}^2 \lt[\int_{\R^{n-1}}
 \abs{z'}^2G(\Phi_u(x'-tz',x'+(1-t)z'))\dz'\rt]\dx' \dt,
\end{aligned}
\]
where we made the change of variables $z'=y'-x'$, $\hat{x}'=x'+tz'$,  and used that by convexity
of $D_2$, $ D_2(tz')\cap D_2(-(1-t)z')\subsq D_2$. We finally claim that 
 for every fixed $t\in[0,1]$ and $x'\in D_2$,
 \begin{equation}\label{claimtechnical1}
  \int_{\R^{n-1}}
 \abs{z'}^2G(\Phi_u(x'-tz',x'+(1-t)z'))\dz'\le C I^1_G,
 \end{equation}
 which would conclude the proof of~\cref{techlem1:skernel}. For this we set $G(z)=g(|z|)$ for some $g: \R^+\mapsto\R^+$. 
and  write using polar coordinates
\begin{multline*}
   \int_{\R^{n-1}} \abs{z'}^2G(\Phi_u(x'-tz',x'+(1-t)z'))\dz'\\
=\int_{\bbs^{n-2}}\int_0^\infty  r^n g \big(\sqrt{r^2+ |u(x'-t r\sigma )-u(x'+ (1-t)r\sigma)|^2}\big)\dr\dH^{n-2}_\sigma.
\end{multline*}
 We finally notice that for every fixed $t\in [0,1]$, $x'\in D_2$ and  $\sigma\in \bbs^{n-2}$, the function $\Psi(r):=\sqrt{r^2+ |u(x'-tr\sigma)-u(x'+ (1-t)r\sigma)|^2}$ is  Lipschitz continuous with 
\[\frac{\sqrt{5}}{2} r\ge \Psi(r)\ge r \qquad \textrm{and} \qquad \frac{5}{4}\ge
\Psi'(r)\ge \frac{3}{2\sqrt{5}}
\]
so that making the change of variables $s=\Psi(r)$ we find
\[\int_{\R^{n-1}} \abs{z'}^2G(\Phi_u(x'-tz',x'+(1-t)z'))\dz'\le C \int_0^\infty  s^n g (s)\ds=C I^1_G.
\]
This concludes the proof of~\cref{claimtechnical1}.
\end{proof}

\section{Caccioppoli inequality}\label{sec:caccio}

First, we introduce the notion of \textsl{flatness} for sets of finite perimeter.

\begin{dfn}[Flatness]\label{dfn:flatness}
For any set of finite perimeter $E\subsq\R^n$ we define the flatness of
$E$ in~$x\in\partial E$ at scale $r>0$ with respect to the direction $\nu\in\bbsn$ by
\[
\f(E,x,r,\nu) \coloneqq \inf_{c\in\R}~ \frac{1}{r^{n-1}}\int_{\partial^* E\cap\cy(x,r,\nu)}
\frac{\abs{(y-x)\cdot\nu-c}^2}{r^2} \dhn_y.
\]
When $\nu=e_n$, we write $\f_n(E,x,r)$ for $\f(E,x,r,e_n)$ and we write $\f_n(E,r)$ for $\f_n(E,0,r)$.
\end{dfn}

Using the harmonic approximation result given by~\cref{prp:harmapprox}, we will be able to show in~\cref{lem:excess_improv} that there exists a direction $\nu$ such that 
$\f(E,\lambda r,\nu)\lesssim\lambda^2\e_n(E,r)$ for $(\Lambda,r_0)$-minimizers of~$\calF_{\eps,\gamma}$, as long as $r$ is much larger than $\eps$.
To pass this estimate to the excess at scale $\lambda r/2$, we establish in this section a
Caccioppoli-type (or Reverse Poincaré) inequality.
The key argument is to prove first that for sets which are sufficiently flat, the quasi-minimality
condition~\cref{wquasimin:ineq} can be upgraded.

Let us introduce some additional notation. For the convenience of the reader, these also appear in
the \nameref{subsec:notation} section at the beginning of the paper.

For any $t>0$ and $z\in\R^{n-1}$, we define $\ky_t(z)\coloneqq D_t(z)\times\left(-1,1\right)$, and we
simply write $\ky_t$ when $z=0$.
For any cylinder $\ky_t(z)$, any set of locally finite perimeter $E$, and any constant $c\in \R$, we define
the quantities
\[
\scrF(E,\ky_t(z),c)\coloneqq \int_{\ky_t(z)\cap\partial^* E} \frac{(x_n-c)^2}{t^2}\dhn
\]
and
\[
\scrE(E,\ky_t(z))\coloneqq P(E;\ky_t(z))-\calH^{n-1}(D_t(z)).
\]
When $z=0$, we make the abuse of notation $\scrE(E,t)=\scrE(E,\ky_t(0))$ and
$\scrF(E,t,c)=\scrE(E,\ky_t(0),c)$.
Let us point out that these two quantities are respectively linked with the (non-scale-invariant)
flatness and excess of $E$ at scale $t$ in the direction $e_n$.
Indeed, if $0\in \partial E$ and if for some $h\in(0,t)$
\[
\left\{ (x',x_n)\in \ky_t~:~ x_n<-h\right\} \subsq E\cap \cy_t
\subsq \left\{ (x',x_n)\in \ky_t~:~ x_n<h\right\},
\]
then
\[
\f_n(E,t) = \inf_{c\in\R}~ \frac{1}{t^{n-1}} \scrF(E,t,c),
\]
and
\[
\calH^{n-1}(D_t) = \int_{\partial^* E\cap \cy_t} \nu_E\cdot e_n
\]
(see~\cite[Lemma~22.11]{Mag2012}), thus, for any $t\in (0,1)$,
\begin{equation}\label{eq:linkscrEexcess}
\scrE(E,t)
= \int_{\partial^* E\cap \cy_t} (1-\nu_E\cdot e_n)\dH^{n-1}
= \frac{1}{2}\int_{\partial^* E\cap \cy_t} \abs{\nu_E-e_n}^2\dH^{n-1}
=\left( \frac{t^{n-1}}{2}\right)\e_n(E,t).
\end{equation}
 Notice in particular that $\scrE(E,\cdot)$ is nondecreasing in $(0,1)$.
Eventually, recalling the definition of the function $Q$ in~\cref{eq:defQ}, for any $\theta\in
[0,1]$ we define the function~$Q_{1-\theta}$ by
\begin{equation}\label{def:Qtheta}
Q_{1-\theta}(t)\coloneqq Q(t^{1-\theta}),\qquad\forall t>0.
\end{equation}

\subsection{A refined quasi-minimality condition}

We improve the quasi-minimality condition~\cref{wquasimin:ineq} for sets which are sufficiently
flat. For any $\eps>0$, let us define the \enquote{critical} energy functional
\[
\calF_{\eps}(E)\coloneqq \calF_{1,\eps}(E)=(P-P_\eps)(E).
\]

\begin{prp}\label{prp:benoitmin}
Assume that $G$ satisfies~\cref{Hposrad,,Hmoment1}, and let $\eps\in(0,1)$, $\gamma\in(0,1)$,
$\theta\in (0,1]$ and $\Lambda>0$ with $4\Lambda\le 1-\gamma$.
There exists $C=C(n)>0$ such that if $E$ is a $(\Lambda,4)$-minimizer of
$\calF_{\eps,\gamma}$ with
\[
\left\{  x_n<-\frac{1}{4}\right\}\cap \ky_{3}
\subsq E\cap \ky_{3}
\subsq\left\{ x_n<\frac{1}{4}\right\}\cap \ky_{3},
\]
then the following holds.
If $t\in(0,2)$ is such that $\calH^{n-1}(\partial \ky_t\cap\partial E)=0$  then for any set $F$ of finite perimeter such that $E\triangle F\subsq \ky_t$
and 
\[
\left\{  x_n<-\frac{1}{4}\right\}\cap \ky_t
\subsq F\cap \ky_t
\subsq\left\{  x_n<\frac{1}{4}\right\}\cap \ky_t,
\]
we have
\begin{multline}\label{benoitmin:res}
\scrE(E,t) \le \left(\frac{1+\gamma}{1-\gamma}\right)\scrE(F,t)
+\frac{2\gamma}{(1-\gamma)} \big[\scrE(E,t+\eps^\theta)-\scrE(E,t)\big]\\
+\frac{C}{(1-\gamma)}Q_{1-\theta}\left(\frac{1}{\eps}\right)
+\frac{\Lambda}{(1-\gamma)} \abs{E\triangle F} 
+\frac{1+3\gamma}{(1-\gamma)}\hn(\partial^* F\cap \partial \ky_t)	.
\end{multline}
\end{prp}

\begin{proof}
To simplify notation a bit set $\eta\coloneqq \hn(\partial^*F\cap \partial \ky_t)$. Since
$\calH^{n-1}(\partial \ky_t\cap\partial E)=0$ and $E\triangle F\subsq \ky_t$ we have 
\begin{equation}\label{eq:difPKyt}
P(E)-P(F)=P(E;\ky_t)-P(F;\ky_t)-\eta.
\end{equation}
By $(\Lambda,4)$-minimality of $E$ we find
\begin{equation}\label{benoitmin:eq1}
(1-\gamma) P(E;\ky_t)\le (1-\gamma)P(F;\ky_t)
+\gamma\big[\calF_\eps(F)-\calF_\eps(E)\big] +\Lambda\abs{E\triangle F}+(1-\gamma)\eta.
\end{equation}
In the next two steps we prove that 
\begin{equation}\label{eq:claimdeltaCalF}
 \calF_\eps(F)-\calF_\eps(E)\le  2 \scrE(F,t+\eps^\theta) +CQ_{1-\theta}\left(\frac{1}{\eps}\right).
\end{equation}
\medskip

\noindent\textit{Step 1.} In this first step we localize the estimate. Setting for simplicity
\[
\wt{D}_t\coloneqq D_{t+\eps^\theta},\qquad \wt{\ky}_t\coloneqq \ky_{t+\eps^\theta}
\]
and defining the \enquote{localized} functional
\[
\calF_\eps^{\loc}(E) \coloneqq 
P(E;\wt{\ky}_t)- \iint_{(E\cap \wt{\ky}_t)\times (E^\compl\cap \wt{\ky}_t)}G_\eps(x-y)\dx\dy,
\]
we claim that 
\begin{equation}\label{benoitmin:resstep1}
\calF_\eps(F)-\calF_\eps(E)
\le \calF_\eps^\loc(F)-\calF_\eps^\loc(E)+CQ_{1-\theta}(\eps^{-1}).
\end{equation}
Since  $E\triangle F\subsq \ky_t\csubset \wt{\ky}_t$,
$P(E)-P(F)=P(E;\wt{\ky}_t)-P(F;\wt{\ky}_t)$ and thus in order to prove~\cref{benoitmin:resstep1}, we  just
need to consider the nonlocal part. Setting $\Phi(A,B)=\iint_{A\times B} G_\eps(x-y) \dx\dy$ and
using that $E\triangle F\subsq \ky_t$ and $\Phi(A,B)\geq 0$, we have 
\begin{equation*}
 \begin{aligned}
  &P_\eps(E)-P_\eps(F)= \Phi(E,E^c)-\Phi(F,F^c)\\
  &=
\lt[\Phi(E\cap \wt{\ky}_t,E^c\cap \wt{\ky}_t)-\Phi(F\cap\wt{\ky}_t,F^c\cap \wt{\ky}_t)\rt]
+\Phi(E\cap \ky_t,E^c\backslash \wt{\ky}_t)-\Phi(F\cap \ky_t, E^c\backslash \wt{\ky}_t)\\
&\qquad 
+\Phi(E\backslash \wt{\ky}_t,E^c\cap \ky_t)-\Phi(E\backslash \wt{\ky}_t,F^c\cap \ky_t)\\
&\le \lt[\Phi(E\cap \wt{\ky}_t,E^c\cap \wt{\ky}_t)-\Phi(F\cap\wt{\ky}_t,F^c\cap \wt{\ky}_t)\rt]+\Phi(E\cap \ky_t,E^c\backslash \wt{\ky}_t)+\Phi(E\backslash \wt{\ky}_t,E^c\cap \ky_t)\\
&\le \lt[\Phi(E\cap \wt{\ky}_t,E^c\cap \wt{\ky}_t)-\Phi(F\cap\wt{\ky}_t,F^c\cap \wt{\ky}_t)\rt]+2\Phi(\ky_t,(\wt{\ky}_t)^c).
 \end{aligned}
\end{equation*}
In order to prove~\cref{benoitmin:resstep1} it is thus enough to estimate $\Phi(\ky_t,(\wt{\ky}_t)^c)$. For this we write that 
\begin{equation*}
 \begin{aligned}
  \Phi(\ky_t,(\wt{\ky}_t)^c)&=\iint_{\ky_t\times (\wt{\ky}_t)^c} G_\eps(x-y)\dx\dy\\
  &=\int_{\R^n\times\R^n} \ind_{\ky_t}(x)\ind_{\wt{\ky}_t^c}(x+z)G_\eps(z)\dx\dz\\
  &=\int_{\R^n\times\R^n}\ind_{\ky_t}(x)\ind_{\wt{\ky}_t^c}(x+z) \ind_{\R^n\backslash B_{\eps^\theta}}(z)G_\eps(z)\dx\dz\\
  &\le \int_{\R^n\times\R^n}\ind_{\ky_t}(x)\ind_{\ky_t^c}(x+z) \ind_{\R^n\backslash B_{\eps^\theta}}(z)G_\eps(z)\dx\dz\\
  &\le \frac{1}{2}\lt(\int_{\R^n\backslash B_{\eps^\theta}} |z|G_\eps(z) \dz\rt) P(\ky_t)\\
  &\le C Q_{1-\theta}(\eps^{-1}),
 \end{aligned}
\end{equation*}
where we used~\cref{boundpk} with $K=\ind_{\R^n\backslash B_{\eps^\theta}}G_\eps$.\\
\medskip

\noindent\textit{Step 2.} In this step we show
\begin{equation}\label{benoitmin:resstep2}
\calF_\eps^\loc(F)-\calF_\eps^\loc(E) \le 2\scrE(F,t+\eps^\theta)+C Q_{1-\theta}(\eps^{-1}).
\end{equation}

Together with~\cref{benoitmin:resstep1} this would conclude the proof of~\cref{eq:claimdeltaCalF}. To this aim,
 we use a slicing argument  (see \cite[Section~3]{MP2021} or \cite{Lud2014}), rewriting
$P$ and $P_\eps$ as an average over $1$-dimensional slices.
Let us set $\rho(t)\coloneqq \om_{n-1}\abs{t}^{n-1}g(\abs{t})$ and $\rho_\eps(t)\coloneqq \eps^{-2}
\rho(\eps^{-1} t)$ for~$t\in\R\setminus\{0\}$. For every line segment $L\subsq\R^n$, we
define the one-dimensional nonlocal perimeter functional in $L$
\[
\begin{aligned}
P_\eps^{\unD}(E;L)
\coloneqq& \iint_{L\times L} \abs{\ind_E(x)-\ind_E(y)} \rho_\eps(x-y)\dH^1_x\dH^1_y
=\,2\iint_{(E\cap L)\times (E^\compl\cap L)} \rho_\eps(x-y)\dH^1_x\dH^1_y
\end{aligned}
\]
and the one-dimensional critical energy in $L$ by
\begin{equation}\label{eq:defF1D}
\calF_\eps^{\unD}(E;L)\coloneqq
\calH^0(\partial^* E\cap L)-P_\eps^{\unD}(E;L).
\end{equation}
Proceeding as in~\cite[Proposition~3.1 \& Corollary~3.3]{MP2021} (the only difference is the
restriction to $\wt{\ky}_t$) we have
\[
\calF_\eps^{\loc}(E)=\frac{1}{2\om_{n-1}}\int_{\bbsn}\int_{\sigma^\perp} 
\calF_\eps^{\unD}(E;\wt{L}_{\sigma,x})\dhn_x\dhn_\sigma,
\]
where we set
\begin{equation}\label{Ltilde}
L_{\sigma,x}\coloneqq x+\R\sigma
\qquad\text{and}\qquad
\wt{L}_{\sigma,x}\coloneqq L_{\sigma,x}\cap \wt{\ky}_t.
\end{equation}
In particular
\begin{equation}\label{benoitmin:eq5}
\calF_\eps^{\loc}(F)-\calF_\eps^{\loc}(E)
=\frac{1}{2\om_{n-1}}\int_{\bbsn}\int_{\sigma^\perp} 
(\calF_\eps^{\unD}(F;\wt{L}_{\sigma,x})
-\calF_\eps^{\unD}(E;\wt{L}_{\sigma,x}))\dx\dhn_\sigma.
\end{equation}
\smallskip
\textit{Step 2.1.}
We claim that for every $\sigma\in\bbsn$ and $\calH^{n-1}$-a.e $x\in\{\sigma\}^\perp$,
we have
\begin{multline}\label{benoitmin:claim1}
\calF_\eps^{\unD}(F;\wt{L}_{\sigma,x})-\calF_\eps^{\unD}(E;\wt{L}_{\sigma,x})\\
\le 
\left\{\begin{array}{ll}
0&\qquad\text{ it }\partial^* F\cap \wt{L}_{\sigma,x}=\emptyset\text{,}\\
2\left(\calH^0(\partial^* F\cap \wt{L}_{\sigma,x})-1\right)+C Q_{1-\theta}(\eps^{-1})&
\qquad\text{ otherwise.}
\end{array}\right.
\end{multline}
 By the standard properties of one-dimensional restrictions of sets of finite perimeter (see e.g.~\cite{AFP2000}), 
 it is enough to prove~\cref{benoitmin:claim1} for $E,F\subsq \R$ and $\wt{L}_{\sigma,x}=L=(0,a)$ for some $a>0$. Notice that since $E$ and~$F$ are of finite perimeter in $L$,
 they are just a finite union of disjoint intervals.  \\
By~\cite[Remark~3.2]{MP2021}, for any set of finite perimeter $E\subsq \R$, we
have $P^{\unD}_\eps(E;\R)\le \calH^0(\partial E)$, which implies $\calF^{\unD}_\eps(E;L)\ge
\calF^{\unD}_\eps(E;\R)\ge 0$.
Thus, if $\calH^0(\partial F\cap L)=0$ (that is, $\partial F\cap L=\emptyset$), then
$\calF_\eps^{\unD}(F;L)-\calF_\eps^{\unD}(E;L)\le \calF_\eps^{\unD}(F;L)\le -P_\eps^\unD(F;L)\le
0$.
If $\calH^0(\partial F\cap L)\ge 2$, then
\[
\calF^{\unD}_\eps(F;L)-\calF^{\unD}_\eps(E;L)
\le \calF^{\unD}_\eps(F;L)
\le \calH^0(\partial F\cap L)
\le 2\left(\calH^0(\partial F\cap L)-1\right).
\]
There remains to focus on the case where $\calH^0(\partial F\cap L)=1$. In this case we claim that 
\begin{equation}\label{eq:toprovestep2}
 \calF^{\unD}_\eps(F;L)-\calF^{\unD}_\eps(E;L)
\le C Q_{1-\theta}(\eps^{-1}).
\end{equation}
Let $t_F$ be such
that $L\cap\partial F=\{t_F\}$ then either $F\cap L=(0,t_F)$ or $F\cap L=(t_F,a)$. Since
both cases are similar, we treat only the case $F\cap L=(0,t_F)$. We distinguish two
sub-cases.\medskip

\noindent
\textit{Case 1: $\dist(t_F,L^\compl)\ge \eps^\theta$.} In this case we argue somewhat similarly to
\cref{benoitmin:resstep1}. Using the fact that
\begin{equation}\label{benoitmin:eq5b}
2\int_{-\infty}^{c}\int_{c}^\infty\rho_\eps(s-t)\ds\dt=2\int_0^\infty
t\rho_\eps(t)\dt=1,\qquad\forall c\in\R,
\end{equation}
we compute
\[
\begin{aligned}
\calF^\unD_\eps(F; L)
&=1-2\int_0^{t_F} \int_{t_F}^a \rho_\eps(s-t)\ds\dt\\
&=2\int_{-\infty}^{t_F} \int_{t_F}^\infty \rho_\eps(s-t)\ds\dt
-2\int_0^{t_F} \int_{t_F}^a \rho_\eps(s-t)\ds\dt\\
&=2\int_{-\infty}^{t_F}\int_a^\infty \rho_\eps(s-t)\ds\dt+2\int_{-\infty}^0 \int_{t_F}^a
\rho_\eps(s-t)\ds\dt\\
&\le2\int_{-\infty}^{0}\left(\int_{a-t_F}^\infty \rho_\eps(s-t)\ds\dt
+\int_{t_F}^\infty \rho_\eps(s-t)\ds\right)\dt\\
&\le 4\int_{-\infty}^0 \int_{\eps^\theta}^\infty \rho_\eps(s-t)\ds\dt\ = 4\int_{\eps^\theta}^\infty (t-\eps^\theta)\rho_\eps(t)\dt.
\end{aligned}
\]
Thus,
\[
\calF^\unD_\eps(F; L) \le  C \int_{\R^n\backslash B_{\eps^\theta}}\abs{z}G_\eps(z)\dz
= C Q_{1-\theta}(\eps^{-1}),
\]
proving~\cref{eq:toprovestep2} in this case.\medskip\\
\textit{Case 2: $\dist(t_F,L^\compl)< \eps^\theta$.} Either $0<t_F<\eps^\theta< a$ or
$0<a-\eps^\theta<t_F<a$. Both cases are similar, we only treat the case $0<a-\eps^\theta<t_F<a$.

Notice that $E\triangle F\subsq \ky_t$ implies $t_F\in \partial E$ and 
\begin{equation}\label{benoitmin:eq6}
F^\compl\cap (a-\eps^\theta,a)=E^\compl\cap(a-\eps^\theta,a)=(a-\eps^\theta,a).
\end{equation}
Let us write $E\cap (0,a)=\bigcup_{j=1}^k I_j$, where $k\ge 1$ and $I_j\subsq (0,a)$ are open
intervals. Then let $\{s_1,\dotsc,s_{k_1},t_1,\dotsc,t_{k_2}\}$ be the elements of $\partial
E\cap(0,a)$ (notice that $0$ and $a$ are excluded) such that
\begin{itemize}
\item $s_1<s_2<\dotsc<s_{k_1}$ and $t_1<t_2<\dotsc<t_{k_2}=t_F$;
\item for each $j$, $s_j$ is a left endpoint of some $I_i$, and $t_j$ is a right endpoint of some
$I_i$.
\end{itemize}
Note that $s_1$ may not be the left endpoint of $I_1$ (if $I_1=(0,t_1)$), so that $k_1$ may be
different from $k_2$ (see~\cref{fig:fig1}). The fact that $t_{k_2}=t_F$ is due to~\cref{benoitmin:eq6}.

 \begin{figure}[ht]
     \begin{minipage}{0.48\textwidth}
         \begin{center}
 		\captionsetup{width=.95\textwidth}
 		\includestandalone[width=.95\textwidth]{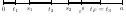}
 		\end{center}
 	\end{minipage}
 	\hspace{10pt}
     \begin{minipage}{0.48\textwidth}
         \begin{center}
 		\captionsetup{width=.95\textwidth}
 		\includestandalone[width=.95\textwidth]{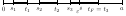}
 		\end{center}
 	\end{minipage}
 	\caption{Two examples of the situation in Step~2.1, Case 2 in the proof of~\cref{prp:benoitmin}. On the left,
 	$k_1=2$ and $k_2=3$, and on the right, $k_1=k_2=3$. The thick segments represent the set $E\cap (0,a)$.}
 	\label{fig:fig1}
 \end{figure}

For $1\le j\le k_1$, we denote by $A_j$ the connected component of $E^\compl\cap L$ which is
immediately on the left side of $s_j$ (that is, its right endpoint is $s_j$), and by $\wt{B}_j$ the
union of all the connected components of $E\cap L$ on the right side of $s_j$.
Similarly, for $1\le j\le k_2$, we denote by $B_j$ the connected component of $E\cap L$ which is
immediately on the left side of $t_j$, and by $\wt{A}_j$ the union of all the connected components of
$E^\compl\cap L$ on the right side of $t_j$. See~\cref{fig:fig2}.
 
\begin{figure}[ht]
\centering
\includestandalone[width=.55\textwidth]{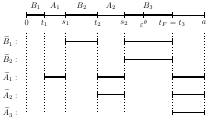}
\caption{An example of the situation in Step~2.1, Case 2 in the proof of~\cref{prp:benoitmin}, when
$k_1=2$ and $k_2=3$, with the representation of the $\wt{A}_j$ and the~$\wt{B}_j$.}
\label{fig:fig2}
\end{figure}

Then using that $\calH^0(\partial E\cap L)=k_1+k_2$, recalling the definition of $\calF_\eps^\unD$
in~\cref{eq:defF1D} and decomposing the domain of integration of $P_\eps^\unD(E;L)$ we see that
\[
\calF_\eps^\unD(E;L)
=\sum_{j=1}^{k_1}\left[1-2\iint_{A_j\times \wt{B}_j} \rho_\eps(s-t)\ds\dt\right]
+\sum_{j=1}^{k_2}\left[1-2\iint_{B_j\times \wt{A}_j} \rho_\eps(s-t)\ds\dt\right].
\]
Each term of each sum is nonnegative by~\cref{benoitmin:eq5b}, and since $\wt{A}_{k_2}=(t_F,a)$ by
\cref{benoitmin:eq6} and $B_{k_2}\subsq (0,t_F)$ this implies in particular
\begin{equation*}\label{benoitmin:eq7}
\calF_\eps^\unD(E;L)
\ge 1-2\iint_{(0,t_F)\times (t_F,a)} \rho_\eps(s-t)\ds\dt
= \calF_\eps^\unD(F;L),
\end{equation*}
concluding the proof of~\cref{eq:toprovestep2} in this case as well.\smallskip

\noindent
\textit{Step 2.2.} For $\sigma\in \bbsn$ we define $\pi_\sigma$ as the projection on $\{\sigma\}^\perp$. We then set
\begin{align}
\nonumber
\Sha(\Sigma;\wt{\ky}_t)
\coloneqq & {\frac{1}{2\om_{n-1}}}\int_{\bbsn}\calH^{n-1}(\pi_{\sigma^\perp}(\Sigma\cap \wt{\ky}_t))\dhn_\sigma\\
\label{defshadow}
=&{\frac{1}{2\om_{n-1}}}\int_{\bbsn}\int_{\sigma^\perp} \ind_{\{\wt{L}_{\sigma,x}\cap\Sigma\neq\emptyset\}}
\dhn_x\dhn_\sigma,
\end{align}
where we recall the definition~\cref{Ltilde} of the segments $\wt{L}_{\sigma,x}$.
Since for any $(n-1)$-rectifiable set $\Sigma\subset\R^{n-1}$ there holds
\begin{equation}\label{Hn-1=In-1}
\calH^{n-1}(\Sigma)=\frac{1}{2\om_{n-1}}\int_{\bbsn}\int_{\sigma^\perp} 
\calH^0(\Sigma\cap L_{\sigma,x})\dhn_x\dhn_\sigma,
\end{equation}
we have
\begin{align*}
P(F;\wt{\ky}_t)
&=\frac1{2\om_{n-1}}\int_{\bbsn}\int_{\sigma^\perp} \calH^0(\partial^* F\cap \wt{L}_{\sigma,x})
\dhn_x\dhn_\sigma\\
&\ge \frac1{2\om_{n-1}} \int_{\bbsn}\int_{\sigma^\perp} \ind_{\{\wt{L}_{\sigma,x}\cap \partial^* F\neq\emptyset\}}
\dhn_x\dhn_\sigma.
\end{align*}
Thus, inserting~\cref{benoitmin:claim1} into~\cref{benoitmin:eq5} and using the fact that 
\[
 \int_{\bbsn}\int_{\sigma^\perp} \ind_{\{\wt{L}_{\sigma,x}\neq\emptyset\}}
\dhn_x\dhn_\sigma\le C
\]
gives
\begin{equation}\label{benoitmin:consclaim1}
\calF_\eps^{\loc}(F)-\calF_\eps^{\loc}(E)
\le 2\lt(P(F;\wt{\ky}_t)-\Sha(\partial^*F;\wt{\ky}_t)\rt)+CQ_{1-\theta}(\eps^{-1}).
\end{equation}
By~\cref{lem_shadow} below, $\Sha(\partial^*F;\wt{\ky}_t)$ is minimal when $F\cap \wt{\ky}_t=\wt{D}_t\times (-1,0)$, which gives
\[
\Sha(\partial^*F;\wt{\ky}_t)\ge \calH^{n-1}(\wt{D}_t).
\]
In view of~\cref{benoitmin:consclaim1}, this concludes the proof of~\cref{benoitmin:resstep2}.
\medskip

\noindent
\textit{Step 3.} We may now conclude the proof of~\cref{benoitmin:res}. By~\cref{eq:claimdeltaCalF} and 
\cref{benoitmin:eq1}, we find
\[
(1-\gamma) P(E;\ky_t)\le (1-\gamma)P(F;\ky_t)
+\gamma\left[2\scrE(F,t+\eps^\theta)+CQ_{1-\theta}(\eps^{-1})
\right] +\Lambda\abs{E\triangle F}+ (1-\gamma)\eta.
\]
Subtracting $(1-\gamma)\calH^{n-1}(D_t)$ from the previous inequality and using that by~\cref{eq:difPKyt}
\[
 \scrE(F,t+\eps^\theta)-\scrE(F,t)=\scrE(E,t+\eps^\theta)-\scrE(E,t)+\eta,
\]
yields
\begin{align*}
(1-\gamma) \scrE(E,t)
&\le (1-\gamma)\scrE(F,t)+2\gamma\scrE(F,t+\eps^\theta)
+C\gamma Q_{1-\theta}(\eps^{-1})
+\Lambda\abs{E\triangle F} +(1+\gamma)\eta\\
&=(1+\gamma)\scrE(F,t)+2\gamma\big[\scrE(F,t+\eps^\theta)-\scrE(F,t)\big]+C\gamma Q_{1-\theta}(\eps^{-1})
+\Lambda\abs{E\triangle F}+(1+\gamma)\eta\\
&=(1+\gamma)\scrE(F,t)+2\gamma\big[\scrE(E,t+\eps^\theta)-\scrE(E,t)\big]+C\gamma Q_{1-\theta}(\eps^{-1})
+\Lambda\abs{E\triangle F}+(1+3\gamma)\eta.
\end{align*}
Dividing by $(1-\gamma)$ concludes the proof of~\cref{benoitmin:res}.
\end{proof}

We now prove that the quantity $\Sha(\partial^*E;\ky_t)$ introduced in~\cref{defshadow} is minimal among obstructing sets (in the sense of~\cref{ineg_shadow_1_dir_lem_1d_0} below) when $\partial^* E$ is a horizontal hyperplane. 

\begin{lem}
\label{lem_shadow}
For any $t>0$, and any set of  finite
perimeter $E\subsq\R^n$ such that
\begin{equation}\label{ineg_shadow_1_dir_lem_1d_0}
\{x_n< -1/4\}\cap \ky_t\ \subsq\ {E}\cap \ky_t\ \subsq\ \{x_n< 1/4\}\cap\ky_t,
\end{equation}	
we have
\begin{equation}\label{lem_shadow:res}
 \Sha(\partial^*E;\ky_t)\ge\Sha(D_t\times\{0\};\ky_t)
=\calH^{n-1}(D_t).
\end{equation}
\end{lem}
The last identity in~\cref{lem_shadow:res} follows from the choice of the factor $1/(2\om_{n-1})$, as in~\cref{Hn-1=In-1}.
\begin{proof}[Proof of~\cref{lem_shadow}]
 We start by setting some notation.
We denote by $S:\R^n\to\R^n$ the symmetry with respect to the vertical line
$\{0_{\R^{n-1}}\}\times\R$, that is, for $\xi=(\xi',\xi_n)$,
\[
S\xi\coloneqq (-\xi',\xi_n).
\]
 We write 
\[
 \Sha(\partial^*E;\ky_t)=\frac{1}{4\om_{n-1}} \int_{\bbsn} \lt[\calH^{n-1}(\pi_{\sigma^\perp}(\partial^{*} E\cap
 \ky_t)) +
 \calH^{n-1}(\pi_{(S\sigma)^\perp}(\partial^{*} E\cap \ky_t))\rt]\dhn_\sigma.
\]
We establish below that for every $\sigma\in \bbsn$,  the {integrand} is minimal when
$\partial^{*} E$ is horizontal in $\ky_t$, that is,
\begin{equation}
\label{ineg_shadow_1_dir}
\calH^{n-1}(\pi_{\sigma^\perp}(\partial^{*} E\cap \ky_t)) + \calH^{n-1}(\pi_{(S\sigma)^\perp}(\partial^{*}
E\cap \ky_t)) \ge \calH^{n-1}(\pi_{\sigma^\perp}(D_t)) +
\calH^{n-1}(\pi_{(S\sigma)^\perp}(D_t)).
\end{equation}
The inequality~\cref{lem_shadow:res} and thus the lemma then follows by integration. The proof of~\cref{ineg_shadow_1_dir} is done in two steps. In a first step we treat the case $n=2$ and in a second step we use slicing to reduce the problem to the previous case.

\medskip
\noindent
\textit{Step 1. Proof~\cref{ineg_shadow_1_dir} for $n=2$.}

\noindent
{\textit{Step 1.1.}}    By~\cite{ACMM2001}, we may decompose {the set of
finite perimeter} $\wt{E}\coloneqq E\cap {\ky_{t}}$  into its (measure theoretic) connected
components. By assumption~\cref{ineg_shadow_1_dir_lem_1d_0} one of these components denoted $\wt{E}_1$
contains $(-t,t)\times(-1,-\frac{1}{4})$. Its external boundary is a Jordan curve $\gamma\in
C^0([0,1),[-t,t]\times[-1,\frac{1}{4}])$ and we have $\gamma([0,1))\subsq \partial^M
\wt{E}_1\subsq \partial \wt{E}_1$ up to a $\calH^1$-negligible set{, where $\partial^M \wt{E}_1$ is the essential
boundary and $\partial \wt{E}_1$ the usual topological boundary of $\wt{E}_1$}.
Moreover, by~\cref{ineg_shadow_1_dir_lem_1d_0} we may assume up to a reparameterization that
$\gamma_{\left|\left[0,\frac{1}{2}\right]\right.}$ is a parameterization of the broken line made of
the three oriented segments joining $(-t,-\frac{1}{4})$ to $(-t,-1)$, then  $(-t,-1)$ to
$(t,-1)$ and $(t,-1)$ to $(t,-\frac{1}{4})$.

 \begin{figure}[ht]
     \begin{center}
 		\captionsetup{width=.95\textwidth}
 		\includestandalone[width=.95\textwidth]{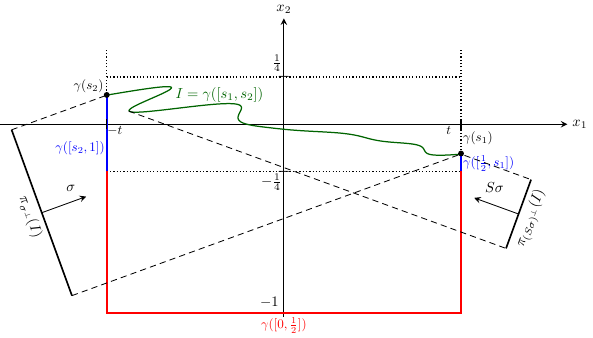}
    \end{center}
 	\caption{The situation in Step~1.1 of the proof of~\cref{lem_shadow}. For readability, the projections
    $\pi_{\sigma^\perp}(I)$ and $\pi_{(S\sigma)^\perp}(I)$ are shifted in the directions
    $\sigma$ and $S\sigma$ respectively.}
 	\label{fig:fig_shadow}
 \end{figure}

Denoting 
\begin{align*}
s_1 & \coloneqq \max\{s\in[{\textstyle\frac{1}{2}},1):\gamma(s)\in \{t\}\times\R\},\\
s_2 & \coloneqq \min \{s\in (s_1,1):\gamma(s)\in \{-t\}\times\R\},
\end{align*}
 we obtain the parameterization of an arc $\gamma:[s_1,s_2]\to
 [-t,t]\times[-\frac{1}{4},\frac{1}{4}]$ with
\[
\gamma(s_1)\in\{t\}\times\R,\qquad\gamma(s_2)\in\{-t\}\times\R,\qquad \gamma((s_1,s_2))\subsq
(-t,t)\times[-{\textstyle\frac{1}{4}},{\textstyle\frac{1}{4}}].
\]
Let $I$ be the segment $[{\gamma(s_1),\gamma(s_2)}]$. See~\cref{fig:fig_shadow}
below.
Obviously, any straight line intersecting $I$ also intersects the arc $\gamma([s_1,s_2])$, hence for
every $\sigma\in \bbs^1$, $\pi_{\sigma^\perp}(\gamma([s_1,s_2]))$ contains $\pi_{\sigma^\perp}(I)$, so
that
\begin{equation}
\label{ineg_shadow_1_dir_1d_1}
\calH^1(\pi_{\sigma^\perp}(I))
\le \calH^1(\pi_{\sigma^\perp}(\gamma([s_1,s_2])))\le \calH^1(\pi_{\sigma^\perp}(\partial^M E\cap \ky_t))
{=\calH^1(\pi_{\sigma^\perp}(\partial^* E\cap \ky_t))}.
\end{equation}

\noindent
{\textit{Step 1.2.}} For  $\sigma\in \bbs^1$, let $\theta\in[0,{\textstyle\frac{\pi}{2}}]$ be such that 
\[
\lt\{\R\sigma,\R (S\sigma)\rt\}=
\lt\{\R\binom{\cos\theta}{\sin\theta},\R\binom{-\cos\theta}{\sin\theta}\rt\}
\]
and  $\varphi\in(-{\textstyle\frac{\pi}{2}},{\textstyle\frac{\pi}{2}})$ be  such that $I$ has direction
$\binom{\cos\varphi}{\sin\varphi}$. We compute
\begin{align*}
\calH^1(\pi_{\sigma^\perp}(I))+\calH^1(\pi_{(S\sigma)^\perp}(I)) &= \dfrac{2t}{\cos\varphi}
\lt(|\sin(\theta-\varphi)| +|\sin(\theta+\varphi)|\rt)\\
&=\begin{cases}
4t |\tan\varphi|\cos\theta&\text{if }0\le \theta\le|\varphi|<\frac{\pi}{2},\\
 4t \sin\theta&\text{if }|\varphi|<\theta\le\frac{\pi}{2}.
\end{cases}
\end{align*}
Since $\tan \vphi$ is increasing in $(\theta,\frac{\pi}{2})$ we have $ |\tan\varphi|\cos\theta\ge \sin\theta$ if $\theta\le|\varphi|<\frac{\pi}{2}$ and thus
\[
\calH^1(\pi_{\sigma^\perp}(I))+\calH^1(\pi_{(S\sigma)^\perp}(I)) 
\ge 4t \sin\theta = \calH^1(\pi_{\sigma^\perp}(I_t))+\calH^1(\pi_{(S\sigma)^\perp}(I_t)).
\]
Together with~\cref{ineg_shadow_1_dir_1d_1} applied both to $\sigma$ and $S\sigma$, this proves the~\cref{ineg_shadow_1_dir} when $n=2$.\medskip\\
\textit{Step 2.} The case $n>2$. There exists $e \in \bbsn\cap[\R^{n-1}\times\{0\}]\sim\bbs^{n-2}$ such that
\[
\sigma= (\sigma\cdot e) e + \sigma_n e_n,\qquad S\sigma= -(\sigma\cdot e) e + \sigma_n e_n.
\]
Denoting $P\coloneqq \mathrm{Span}\{e,e_n\}$, $V\coloneqq P^\perp$ and $P_y\coloneqq y+P$ for $y\in
V$, we have
\begin{align*}
\calH^{n-1}(\pi_{\sigma^\perp}(\partial^{*} E\cap \ky_t)) &= \int_{V\cap B_t}
\calH^1(\pi_{\sigma^\perp}(\partial^{*} E\cap \ky_t\cap P_y))\,d\calH^{n-2}_y,\\
\calH^{n-1}(\pi_{(S\sigma)^\perp}(\partial^{*} E\cap \ky_t)) &= \int_{V\cap B_t}
\calH^1(\pi_{(S\sigma)^\perp}(\partial^{*} E\cap \ky_t\cap P_y))\,d\calH^{n-2}_y.
\end{align*}
Next, for almost every $y\in V$, $E\cap P_y$ is a set with finite perimeter in the plane $P_y$ and
up to a $\calH^1$-negligible set, $\partial^*_{P_y} (E\cap P_y)=(\partial^*_{\R^n} E)\cap
P_y$.

 We notice that for $|y|\ge t$, $\ky_t\cap P_y=\emptyset$ and that for $|y|<t$,
 \[
 \ky_t\cap P_y= y+\lt\{x_1 e+ x_2 e_n \ : \ |x_1|< \sqrt{t^2-|y|^2}, \ |x_2|< 1\rt\}\cong (-
 \sqrt{t^2-|y|^2},\sqrt{t^2-|y|^2})\times(-1,1),
 \]
 where $\cong$ denotes equality up to an isometry. Eventually, we conclude the proof using {\it Step
 1} in $\ky_t\cap P_y$.
 \end{proof}

\begin{rmk}
In~\cref{prp:benoitmin}, we introduced a parameter $\theta\in(0,1]$ to find a proper balance between
the terms
\[
 \big[\scrE(E,t+\eps^\theta)-\scrE(E,t)\big]\qquad\qquad\text{ and }\qquad\qquad
 Q_{1-\theta}\left(\eps^{-1}\right).
\]
As we will see, through an averaging argument, we can roughly estimate
\begin{equation}\label{rmktheta:eq}
 \big[\scrE(E,t+\eps^\theta)-\scrE(E,t)\big]\lesssim \eps^\theta \e(E,2t)
\end{equation}
the first quantity gets smaller the closer $\theta$ is to~$1$.
However, $Q_{1-\theta}(\eps^{-1})$ gets larger as $\theta$ goes to~$1$. In particular when
$\theta=1$, $Q_{1-\theta}(r/\eps)=Q(1)$ is a constant (non-zero unless $G$ is compactly
supported in~$B_1$). This would prevent us from obtaining a decay of the excess through iteration.
We choose $\theta$ later, small enough so that $Q_{1-\theta}(r/\eps)$ stays sufficiently small
down to any scale $\eps^{1-\beta}$ with $\beta\in(0,1)$. As long as $\theta$ is non-zero,
\cref{rmktheta:eq} will be sufficient to proceed with the iteration.
\end{rmk}

\subsection{A Caccioppoli-type inequality}

From the improved quasi-minimality condition given by~\cref{prp:benoitmin}, we first obtain an
intermediate weaker form of a Caccioppoli inequality.
\begin{prp}\label{prp:weakpoincare}
Assume that $G$ satisfies~\cref{Hposrad,,Hmoment1}, and let $\eps\in(0,1)$, $\gamma\in(0,1)$,
$\theta\in (0,1]$ and $\Lambda>0$ such that $\eps^\theta\in(0,\textstyle\frac{1}{4})$ and
$4\Lambda\le 1-\gamma$.
Then for every $(\Lambda,4)$-minimizer $E$ of $\calF_{\eps,\gamma}$ with
\[
\left\{  x_n<-\frac{1}{8}\right\}\cap \ky_3
\subsq E\cap \ky_{3}
\subsq\left\{ x_n<\frac{1}{8}\right\}\cap\ky_3,
\]
the following holds.
For every $c\in\R$ such that $\abs{c}<\frac{1}{4}$ and every $t\in(4\eps^\theta,1)$, we have
\begin{equation}\label{weakpoincare:res}
\scrE\lt(E,t/2\rt)
\le C\left(\big(\scrE(E,t)\scrF(E,t,c)\big)^{\frac{1}{2}}
+\left(\frac{\eps^\theta}{t}\right)\scrE(E,t)+Q_{1-\theta}(\eps^{-1})+\Lambda t^{n-1}\right),
\end{equation}
where $C$ depends only on $n$ and $\gamma$.
\end{prp}

\begin{proof}
For almost every
$t\in(4\eps^\theta,1)$, we have
\begin{equation}\label{weakpoincare:eq0a}
\calH^{n-1}(\partial \ky_{t}\cap\partial E)=0.
\end{equation}
Let us fix such a $t$.
If $\scrF(E,t,c)\ge \frac{1}{16}\scrE(E,t)$, then using the fact that
$\scrE(E,\cdot)$ is nondecreasing, we have
\[
\scrE(E,t/2) \le \scrE(E,t)\le \sqrt{\scrE(E,t)}\sqrt{\scrE(E,t)}
\le 4\big(\scrE(E,t)\scrF(E,t,c)\big)^{\frac{1}{2}}
\]
thus~\cref{weakpoincare:res} holds. Hence, we now assume $\scrF(E,t,c)<
\frac{1}{16}\scrE(E,t)$, and set $\lambda\coloneqq
\sqrt{\frac{\scrF(E,t,c)}{\scrE(E,t)}}\in(0,\textstyle\frac{1}{4})$. 
As in~\cite[Lemma 24.9]{Mag2012} we want to use for $s\in (0,1)$ the construction of~\cite[Lemma~24.6]{Mag2012} as competitor inside $\ky_{st}$ for~\cref{prp:benoitmin}. 
To this aim using for instance~\cite[Theorem~13.8]{Mag2012}, we  approximate $E$ by smooth sets $E_k$ with $|E\triangle E_k|\to 0$, $P(E_k)\to P(E)$ and 
\begin{equation}\label{weakpoincare:eq0b}
\left\{x_n<-\frac{1}{4}\right\}\cap \ky_3
\subsq E_k\cap \ky_3
\subsq\left\{ x_n<\frac{1}{4}\right\}\cap \ky_3.
\end{equation}
By the Morse--Sard lemma, for almost every $s\in(0,1)$,
\[
\partial \ky_{st}\cap\partial E_k \text{ is a } (n-2)\text{-dimensional hypersurface}.
\]
For every such $s$ we may apply~\cite[Lemma~24.6]{Mag2012} with $a=(1-\lambda)st$ and $b=st$, and  use the inequalities
$\sqrt{1+t^2}\le 1+t^2$ and $1-(1-\lambda)^{n-1}\le (n-1)\lambda$, to construct an open set of finite
perimeter $F_{s}$ such that~\cref{weakpoincare:eq0b} holds for $F_s$,
\begin{equation}\label{weakpoincare:eq3}
F_{s}\cap\partial \ky_{st}=E_{k}\cap\partial \ky_{st},
\end{equation}
and
\begin{equation}\label{weakpoincare:eq5}
\scrE(F_{s},st)\le C\left(\lambda st V_{\scrE}(st)+\frac{1}{\lambda
st}V_{\scrF}(st)\right),
\end{equation}
where we have set
\[
V_{\scrE}(a)\coloneqq \frac{\dd}{\dd a}(\scrE(E_k,a))=\calH^{n-2}(\partial \ky_a\cap\partial
E_k)-\calH^{n-2}(\partial D_a)
\]
and
\[
V_{\scrF}(a)\coloneqq \frac{\dd}{\dd a}(a^2\scrF(E_k,a,c))
=\int_{\partial \ky_a\cap\partial E_k} (x_n-c)^2 \dH^{n-2}.
\]
Applying
\cref{prp:benoitmin} with, $F=(F_s\cap \ky_{st})\cup (E\backslash \ky_{st})$ and noticing that by~\cref{weakpoincare:eq3} and~\cite[Theorem~16.16]{Mag2012}, for a.e. $s$,
$\hn(\partial^* F_s\cap \partial \ky_{st})=\hn( (E\triangle E_k)\cap \partial \ky_{st})$
 we find for such $s$,
 \begin{multline}\label{weakpoincare:eq7}
\scrE(E,st) \le C\Big(\scrE(F_s,st)
+ \big[\scrE(E,st+\eps^\theta)-\scrE(E,st)\big] +\Lambda \abs{\ky_{st}}
+Q_{1-\theta}(\eps^{-1})\\
\hphantom{\le C\Big(\scrE(F_s,st) + \big[\scrE(E,st+\eps^\theta)-\scrE(E,st)\big]+\Lambda\abs{\ky_{st}}}
+\hn( (E\triangle E_k)\cap \partial \ky_{st})\Big)\\
\stackrel{\mathclap{\cref{weakpoincare:eq5}}}{\le}C\lt(\lambda st V_{\scrE}(st)+\frac{1}{\lambda
st}V_{\scrF}(st)
+\big[\scrE(E,st+\eps^\theta)-\scrE(E,st)\big] +\Lambda t^{n-1}
+Q_{1-\theta}(\eps^{-1})\rt.\\
\hphantom{\le C\Big(\scrE(F_s,st) + \big[\scrE(E,st+\eps^\theta)-\scrE(E,st)\big]+\Lambda\abs{\ky_{st}}}
\lt.\vphantom{\frac{1}{\lambda st}V_{\scrF}(st)}
+\hn( (E\triangle E_k)\cap \partial \ky_{st})\rt).
\end{multline}
We now integrate~\cref{weakpoincare:eq7} for $s\in(1/2,3/4)$. First, since $\scrE(E,\cdot)$ is nondecreasing, we have
\begin{equation}\label{weakpoincare:eq9}
\frac{1}{4}\scrE(E,t/2)\le \int_{\frac{1}{2}}^{\frac{3}{4}} \scrE(E,st)\ds.
\end{equation}
Second, we compute
\begin{multline}\label{weakpoincare:eq10}
\int_{\frac{1}{2}}^{\frac{3}{4}} 
[\scrE(E,st+\eps)-\scrE(E,st)] \ds
=\frac{1}{t}\int_{\frac{t}{2}}^{\frac{3t}{4}} [\scrE(E,a+\eps^\theta)-\scrE(E,a)]\dd a\\
=\frac{1}{t}\left(\int_{\frac{t}{2}+\eps^\theta}^{\frac{3t}{4}+\eps^\theta} \scrE(E,a) \dd a
-\int_{\frac{t}{2}}^{\frac{3t}{4}} \scrE(E,a)\dd a\right)
\le \frac{1}{t}\int_{\frac{3t}{4}}^{\frac{3t}{4}+\eps^\theta} \scrE(E,a) \dd a
\le \left(\frac{\eps^\theta}{t}\right) \scrE(E,t),
\end{multline}
where for the last inequality we used again that $\scrE(E,\cdot)$ is
nondecreasing.
Third,
\begin{equation}\label{weakpoincare:eq11}
\int_{\frac{1}{2}}^{\frac{3}{4}} st V_{\scrE}(st)\ds
\le \frac{3}{4} \int_{\frac{t}{2}}^{\frac{3t}{4}} V_{\scrE}(a) \dd a
=\frac{3}{4} \lt(\scrE\left(E_k,\frac{3t}{4}\right)-\scrE\lt(E_k,\frac{t}{2}\rt)\rt)\le\frac{3}{4}\scrE(E_k,t).
\end{equation}
Finally, with a similar argument using that  $a\mapsto a^2\scrF(E_k,a,c)$ is nondecreasing, we have 
\begin{equation}\label{weakpoincare:eq12}
\int_{\frac{1}{2}}^{\frac{3}{4}} \frac{1}{st} V_{\scrF}(st)\ds
\le 2\scrF(E_k,t,c).
\end{equation}
Inserting~\cref{weakpoincare:eq9},~\cref{weakpoincare:eq10},~\cref{weakpoincare:eq11,weakpoincare:eq12}
into~\cref{weakpoincare:eq7} yields
\[
\scrE(E,t/2)
\le C\left[\lambda\scrE(E_k,t)+\frac{1}{\lambda} \scrF(E_k,t,c)
+\left(\frac{\eps^\theta}{t}\right) \scrE(E,t)
+\Lambda t^{n-1} +Q_{1-\theta}(\eps^{-1})+|E\triangle E_k|\right].
\]
By~\cref{weakpoincare:eq0a}  we can send $k\to \infty$ to obtain 
\[
 \scrE(E,t/2)
\le C\left[\lambda\scrE(E,t)+\frac{1}{\lambda} \scrF(E,t,c)
+\left(\frac{\eps^\theta}{t}\right) \scrE(E,t)
+\Lambda t^{n-1} +Q_{1-\theta}(\eps^{-1})\right].
\]
Recalling that
$\lambda=\sqrt{\frac{\scrF(E,t,c)}{\scrE(E,t)}}$ concludes the proof of~\cref{weakpoincare:res} for a.e. $t\in(4\eps^\theta,1)$. By the left-continuity of $\scrE(E,\cdot)$ and
$\scrF(E,\cdot,c)$ this actually holds for every $t\in(4\eps^\theta,1)$. 
\end{proof}

We now post-process~\cref{weakpoincare:res} to obtain the desired stronger Caccioppoli inequality.
The main difference with~\cite[Theorem~24.1]{Mag2012} is that in our case we cannot apply~\cref{weakpoincare:res} at scales which are smaller than $\eps^{\theta}$. 

\begin{prp}[Caccioppoli inequality]\label{prp:strongcaccio}
Assume that $G$ satisfies~\cref{Hposrad,,Hmoment1}, and let $\eps\in(0,1)$, $\gamma\in(0,1)$,
$\Lambda>0$ and $r_0>0$ with $\Lambda r_0\le 1-\gamma$.
There exist  constants $\tau_{\mathrm{cac}}=\tau_{\mathrm{cac}}(n)>0$ and  $M_{\mathrm{cac}} > 1$ such that the following holds.
Let  $E$ be a $(\Lambda,r_0)$-minimizer of $\calF_{\eps,\gamma}$ and assume that  $0\in\partial
E$. If for some  $\nu\in\bbsn$, $r_0>M_{\mathrm{cac}} r>0$ and $\theta\in (0,1]$, 
\[
\e(E,M_{\mathrm{cac}} r,\nu)+\left(\frac{\eps}{r}\right)^{\theta}\le \tau_{\mathrm{cac}},
\]
then
\begin{equation}\label{strongcaccio:res}
\e(E,r/2,\nu) \le
C\left(\f(E,r,\nu)+\left(\frac{\eps}{r}\right)^{\theta}\e(E,r,\nu)+\Lambda r
+Q_{1-\theta}\left(\frac{r}{\eps}\right)\right),
\end{equation}
where $C=C(n,G,\gamma)$.
\end{prp}

\begin{proof}
Up to  a rotation and rescaling, we may assume that  $\nu=e_n$ and $r=1$. Therefore, up to choosing~$M_{\mathrm{cac}}$ large enough and
$\tau_{\mathrm{cac}}$ small enough,  $E$ is a $(\Lambda,4)$-minimizer of
$\calF_{\eps,\gamma}$ with $4\Lambda\le 1-\gamma$ and $16\eps^{\theta}<1$. 
Thus,~\cref{strongcaccio:res} amounts to proving
\begin{equation}\label{strongcaccio:res0}
\e_n(E,{\textstyle\frac{1}{2}}) \le
C\left(\f_n(E,1)+\eps^\theta\e_n(E,1)+\Lambda
+Q_{1-\theta}\left(\frac{1}{\eps}\right)\right).
\end{equation}
By~\cref{prp:heightbound}, choosing $M_{\mathrm{cac}}$ even larger if necessary and $\tau_{\mathrm{cac}}=\tau_{\mathrm{cac}}(n)$ small
enough, we have
\[
\left\{ (x',x_n)\in \cy_{4}~:~ x_n<-\frac{1}{8}\right\}
\subsq E\cap \cy_{4}
\subsq\left\{ (x',x_n)\in \cy_{4}~:~ x_n<\frac{1}{8}\right\}.
\]
Thus for every $z\in D_1$, we have
\[
\left\{ (x',x_n)\in \ky_3(z)~:~ x_n<-\frac{1}{8}\right\}
\subsq E\cap \ky_3(z)
\subsq\left\{ (x',x_n)\in \ky_3(z)~:~ x_n<\frac{1}{8}\right\},
\]
so we can apply~\cref{prp:weakpoincare} to $E+(z,0)$ with $2s$ for every
$s\in(2\eps^\theta,\frac{1}{2})$.
For every $s\in (2\eps^\theta,\frac{1}{2})$ such that $D_{2s}(z)\subsq D_1$, we get that
\begin{equation}\label{strongcaccio:eq3}
\scrE(E,\ky_{s}(z))
\le C\left(\big(\scrE(E,\ky_{2s}(z))\scrF(E,\ky_{2s}(z),c)\big)^{\frac{1}{2}}
+\left(\frac{\eps}{s}\right)^\theta\scrE(E,\ky_{2s}(z))\right.
\left.+Q_{1-\theta}\left(\frac{1}{\eps}\right)+\Lambda s^{n-1}\right).
\end{equation}
Setting
\[
h\coloneqq \inf_{\abs{c}<\frac{1}{4}} \int_{\cy_1\cap\partial^* E} (x_n-c)^2\dhn
\]
multiplying~\cref{strongcaccio:eq3} by $s^2$ and taking the infimum over $\abs{c}<\frac{1}{4}$,
using the fact that
\[
s^2\scrF(E,\ky_{2s}(z),c)\le \frac{\scrF(E,\ky_1,c)}{4} \le \frac{h}{4}
\]
for every $s\in (2\eps^\theta,\frac{1}{2})$ with $D_{2s}(z)\subsq D_1$, we find
\begin{equation}\label{strongcaccio:eq4}
s^2\scrE(E,\ky_s(z))
\le C\left(\big(s^2\scrE(E,\ky_{2s}(z))h\big)^{\frac{1}{2}}
+\eps^\theta\scrE(E,\ky_{2s}(z))
+s^2Q_{1-\theta}\left(\frac{1}{\eps}\right)+\Lambda\right).
\end{equation}
Set 
\[
 \Psi\coloneqq \sup \left\{s^2\scrE(E,\ky_{s}(z))~:~ D_{2s}(z)\subsq D_1\text{ and }
s\in \left(4\eps^\theta,\frac{1}{2}\right)\right\}.
\]
If 
\[
 \Psi=\sup\left\{s^2\scrE(E,\ky_{s}(z))~:~ D_{2s}(z)\subsq D_1\text{ and }
s\in \left(4\eps^\theta,8\eps^\theta\rt)\right\}
\]
then~\cref{strongcaccio:res0} holds. Indeed, using the left-continuity of $t\mapsto \scrE(E,\ky_t)$
and the fact that $\scrE(E,\ky_s(z))\le \scrE(E,\ky_1)$ whenever $D_{2s}(z)\subsq D_1$, in that case
we find
\[
\scrE(E,\ky_{\frac{1}{2}})\le \frac{\Psi}{4}\le C \eps^{2\theta} \scrE(E,\ky_1),
\]
which gives~\cref{strongcaccio:res0} recalling that
$\e_n(E,{\textstyle\frac{1}{2}})=2\scrE(E,\ky_{\frac{1}{2}})$ (see~\cref{eq:linkscrEexcess}).
We can thus take the supremum over $s> 4 \eps^{\theta}$ or $s> 8 \eps^{\theta}$ for $\Psi$.
For any $z$ and $s$ such that $D_{2s}(z)\subsq D_1$ and $s\in(\textstyle
8\eps^{\theta},\frac{1}{2})$, we cover $D_s(z)$ by $N=N(n)$ balls $D_{\frac{s}{4}}(z_k)$ with
centers $z_k\in D_s(z)$. Then since $\frac{s}{4}>2\eps^\theta$ and $D_{\frac{s}{2}}(z_k)\subsq D_1$, we can apply~\cref{strongcaccio:eq4} to each
$\left(\frac{s}{4}\right)^2\scrE(E,\ky_{\frac{s}{4}}(z_k))$.
Thus, by the subadditivity of $\scrE$, and by definition of $\Psi$, for such $z$ and
$s\in(8\eps^\theta,\frac{1}{2})$, we deduce
\begin{equation}\label{strongcaccio:eq5}
\begin{aligned}
s^2\scrE(E,\ky_s(z))
&\le \frac{1}{16} \sum_{k=1}^N \left(\frac{s}{4}\right)^2\scrE(E,\ky_{\frac{s}{4}}(z_k))\\
&\le C\sum_{k=1}^N 
\left(\big(s^2\scrE(E,\ky_{\frac{s}{2}}(z_k))h\big)^{\frac{1}{2}}
+\eps^\theta\scrE(E,\ky_{\frac{s}{2}}(z_k))
+s^2 Q_{1-\theta}\left(\frac{1}{\eps}\right)+\Lambda \right)\\
&\le C \left(\sqrt{\Psi h}
+\eps^\theta \Psi
+ Q_{1-\theta}\left(\frac{1}{\eps}\right)+\Lambda\right).
\end{aligned}
\end{equation}
Recall that $\Psi$ is in fact obtained by taking the supremum over the $s$, $z$ such that
$D_{2s}(z)\subsq D_1$ and $s\in(\textstyle 8\eps^{\theta},\frac{1}{2})$ by the above discussion.
Therefore,~\cref{strongcaccio:eq5} yields
\begin{equation}\label{strongcaccio:eq6}
\begin{aligned}
\Psi
&\le C \left(\sqrt{\Psi h}
+\eps^\theta \Psi
+ Q_{1-\theta}\left(\frac{1}{\eps}\right)+\Lambda \right).
\end{aligned}
\end{equation}
If
\[
\eps^{\theta} \Psi +Q_{1-\theta}\left(\frac{1}{\eps}\right)+\Lambda
<\sqrt{\Psi h},
\]
then~\cref{strongcaccio:eq6} implies $\Psi\le C h$.
Otherwise,~\cref{strongcaccio:eq6} implies
\[
\Psi
\le C\left(\eps^\theta \Psi +Q_{1-\theta}\left(\frac{1}{\eps}\right)+\Lambda \right)
\le C\left(\eps^\theta \scrE(E,\ky_1) +Q_{1-\theta}\left(\frac{1}{\eps}\right)+\Lambda \right).
\]
Recalling $\e_n(E,{\textstyle\frac{1}{2}})=2\scrE(E,\ky_{\frac{1}{2}})$, the left-continuity of
$t\mapsto\scrE(E,\ky_t)$ and the definition of $h$, combining the different cases yields
\cref{strongcaccio:res0}. This concludes the proof.
\end{proof}

\section{Uniform regularity}\label{sec:reg}

\subsection[sec]{\for{toc}{Excess decay \texorpdfstring{for $r\les\eps$}{at small scales}}\except{toc}{Excess decay \texorpdfstring{for \forcebold{$r\les\eps$}}{at small scales}}}

If $G$ satisfies assumptions~\cref{Hposrad,Hmoment1,Hint0}, by
\cref{prp:almostregsmallscale,rmk:regsmallscales}, it is standard to obtain power decay of the
excess at small scales. Let us recall the following well-known result.
\begin{prp}\label{classicalexcessdecay}
Let  $\omega>0$, $\alpha\in (0,1)$ and $r_0>0$ be fixed. Then, there exists $\tau=\tau(n,\alpha)>0$ and $C=C(n,\alpha)$ such that the following holds. If
$E$ is such that for every $r\le r_0$ and every $E\triangle F\subset B_r$, 
 \[
  P(E;B_r)\le P(F;B_r)+ \omega r^{n-1+\alpha},
 \]
assuming that $0\in \partial E$ and 
\[
 \e(E,R)+\omega R^{\alpha}\le \tau 
\]
for some $R\le r_0$, then 
\[
 \e(E,r)\le C \lt(R^{-\alpha}\e(E,R)+ \omega \rt)r^\alpha \qquad\forall
r\in(0,R).
\]
\end{prp}
\begin{proof} Although this result is standard and can be reconstructed from e.g.~\cite{Tam1982}, 
we provide a short proof since this precise  statement is not easily accessible in the literature. By scaling we may assume without loss of generality that $R=1$ (replacing $\omega$ by $\omega R^\alpha$).
We thus want to prove that there exist $\tau=\tau(n)>0$ and $C=C(n,\alpha)>0$ such that provided 
\begin{equation}\label{hyptamanini}\e(E,1)+\omega\le \tau 
\end{equation}
 then 
\begin{equation}\label{toproveCampanatoTamanini}
 \e(E,r)\le C\lt(\e(E,1)+ \omega\rt)r^\alpha  \qquad\forall
r\in(0,1).
\end{equation}
We recall that by the tilt Lemma  (see for instance~\cite[Lemma 4.6]{GNR2021} applied to $\Lambda =\omega r^\alpha$ and $\ell=0$ combined with 
the beginning of the proof of~\cite[Proposition 4.1]{GNR2021}), for every $\lambda$ small enough,
there exists $\tau_{\rm tilt}=\tau_{\rm tilt}(n,\lambda)>0$, $C_{\rm tilt}= C_{\rm tilt}(n)>0$ and $C_\lambda>0$ such that provided 
\begin{equation}\label{hyptiltTamanini}
 \e(E,r)+ \omega r^\alpha \le \tau_{\rm tilt},
\end{equation}
we have 
\begin{equation}\label{iterTamanini}
 \e(E,\lambda r)\le C_{\rm tilt} \lambda^2\e(E,r) + C_\lambda \omega r^{\alpha}.
\end{equation}
We now choose $\lambda$ such that $C_{\rm tilt}\lambda^2\le \lambda^\alpha/2$ and then set $C=2C_\lambda/\lambda^\alpha$. Letting for $k\ge 0$, $r_k= \lambda^k$, we claim that for every $k\ge 0$,
\begin{equation}\label{claimTamanini}
 \e(E,r_k)\le C  (\e(E,1)+ \omega)r_k^\alpha.
\end{equation}
By the scaling properties of the excess, see~\cref{prp:scaleexcess}, this would conclude the proof of~\cref{toproveCampanatoTamanini}.\\

We start by noting that the statement holds for $k=0$. Now if it holds for $k-1$, up to choosing
$\tau$ small enough,~\cref{hyptiltTamanini} holds for $r=r_{k-1}$ by~\cref{hyptamanini}. Therefore, by~\cref{iterTamanini} and the induction hypothesis,
\[
 \e(E,r_k)\le C_{\rm tilt} \lambda^2\e(E,r_{k-1}) + C_\lambda \omega r_{k-1}^{\alpha}\le \frac{C}{2} (\e(E,1)+\omega) r_k^\alpha+\frac{C_\lambda}{\lambda^\alpha} \omega r_k^\alpha\le C (\e(E,1)+\omega) r_k^\alpha.
\]
This proves~\cref{claimTamanini}.
\end{proof}
As a consequence, we have the following power decay of the excess for small scales.

\begin{prp}[Excess decay at small scales]\label{prp:smallscales_decexcess}
Assume that $G$ satisfies~\cref{Hposrad,,Hmoment1,,Hint0}, and let $\gamma\in(0,1)$, $\eps>0$, 
$\Lambda>0$ and $r_0>0$ with $\Lambda (r_0+\eps)\le 1-\gamma$.
 There exist positive constants
$\tau^{\mathbf{s}}_{\mathrm{dec}}=\tau^{\mathbf{s}}_{\mathrm{dec}}(n,G,\gamma)$ and $C=C(n,G,\gamma)$
such that the following holds. If $E$ is a $(\Lambda,r_0)$-minimizer of $\calF_{\eps,\gamma}$ with $0\in \partial E$ and such that 
 for some  $R\le r_0$,
\begin{equation}\label{smallinit}
\e(E,R)+ \frac{R}{\eps} \le \tau^{\mathbf{s}}_{\mathrm{dec}},
\end{equation}
then we have (recall that $s_0$ is given by~\cref{Hint0})
\begin{equation}\label{excessdecaysmall}
\e(E,r)\le C \lt(\frac{r}{R}\rt)^{1-s_0} \lt(\e(E,R) +\lt(\frac{R}{\eps}\rt)^{1-s_0}\rt)\qquad\forall
r\in(0,R).
\end{equation}
\end{prp}
\begin{proof}
By scaling (recall~\cref{prp:rescaling}) we can assume that $\eps=1$. By~\cref{prp:almostregsmallscale}, we have for $r\le R$,
 \[
  P( E;B_r)\le P( F; B_r) + C \lt(r^{n-s_0} +\Lambda  r^n\rt) \qquad \forall E\triangle F \subset B_r.
 \]
Since on the one hand, $\Lambda\le 1$ and on the other hand, up to choosing $ \tau^{\mathbf{s}}_{\mathrm{dec}}$ small enough, $R\le 1$ by~\cref{smallinit}, this reduces to 
\[
  P(E;B_r)\le P(F; B_r) + C r^{n-s_0} \qquad \forall E\triangle F \subset B_r.
\]
We can then apply~\cref{classicalexcessdecay} with $\alpha=1-s_0$ to conclude the proof.
\end{proof}

\subsection[sec]{\for{toc}{Excess decay \texorpdfstring{for $r\gg\eps$}{at large scales}}\except{toc}{Excess decay \texorpdfstring{for \forcebold{$r\gg\eps$}}{at large scales}}}

Starting with a small excess at a given scale much larger than $\eps$, we show that the excess is
smaller at a smaller scale, up to tilting the direction.

\begin{lem}[Tilt lemma]\label{lem:excess_improv}
Assume that $G$ satisfies~\cref{Hposrad,,Hmoment1,,Hmoment2}, and let $\gamma\in(0,1)$, $\eps>0$,
$\Lambda>0$ and $r_0>0$ with $\Lambda r_0\le 1-\gamma$.
Then, there exists a positive constant $\lambda_{\mathrm{tilt}}$ such that for every
$\lambda\in(0,\lambda_{\mathrm{tilt}})$, there exists $\tau_{\mathrm{tilt}}=\tau_{\rm tilt}(n,G,\gamma,\lambda)>0$ such that the following holds.
If $E$ is a $(\Lambda,r_0)$-minimizer of $\calF_{\eps,\gamma}$  with $0\in \partial E$ which satisfies, for some
 $0<r\le r_0$ and $\theta\in(0,1]$,
\begin{equation}\label{hyp:epstilt}
\e_n(E,r) +\Lambda r+\left(\frac{\eps}{r}\right)^\theta\le \tau_{\mathrm{tilt}},
\end{equation}
then there exists $\nu\in\bbsn$ such that (recall the definition~\cref{def:Qtheta} of $Q_{1-\theta}$)
\begin{equation}\label{excess_improv:res}
\begin{aligned}
\e(E,\lambda r,\nu) &\le C\lt(\lambda^2\e_n(E,r)+\lambda \Lambda r
+ Q_{1-\theta}\left(\frac{\lambda r}{\eps}\right)\rt),
\end{aligned}
\end{equation}
where $C=C(n,G,\gamma)$.
\end{lem}

\begin{proof}
We follow relatively closely the proof of~\cite[Theorem~25.3]{Mag2012}.
Let $\lambda\in(0,\lambda_{\mathrm{tilt}})$, with $\lambda_{\mathrm{tilt}}$ and
$\tau_{\mathrm{tilt}}$ to be chosen later.
Up to rescaling, we may assume that $r=4$, $\e_n(E,4)+4\Lambda+\eps^\theta\le\tau_{\mathrm{tilt}}$,
 and $E$ is a $(\Lambda,4)$-minimizer of $\calF_{\eps,\gamma}$ with $4\Lambda\le 1-\gamma$.
In the rest of the proof, we shall write $\e_n(r)$ for $\e_n(E,r)$ and $\f(r,\nu)$ for
$\f(E,r,\nu)$.

Assuming that $\tau_{\mathrm{tilt}}\le\tau_{\mathrm{lip}}$, we can apply~\cref{thm:lipapprox1} with
$r=1$. Let $C_1=C_1(n,G,\gamma)$ be a large constant, and set
\[
L\coloneqq C_1\left(\e_n(4)+Q\left(\frac{1}{\eps}\right)+\Lambda \right).
\]
We proceed in two steps.\smallskip\\
\textit{Step 1. }
We claim that if
\begin{equation}\label{excess_improv:s1eq1}
L\le\min(\lambda^{(n-1)(n+3)},\sigma^2),
\end{equation}
where $\sigma(n,G,\gamma,\lambda)$ is the constant given by~\cref{prp:harmapprox} with
$\tau=\lambda^{n+3}$, then there exists $\nu\in\bbs^{n-1}$ such that
\begin{equation}\label{excess_improv:claim1}
\f(\lambda,\nu) \le C\lambda^2 L,
\end{equation}
where $C=C(n,G,\gamma)$.
Let us assume that~\cref{excess_improv:s1eq1} holds, and let us set $u_0\coloneqq
\frac{u}{\sqrt{L}}$. By~\cref{thm:lipapprox1}, $u_0$ satisfies
\[
\int_{D_{2}} \abs{\grad u_0}^2 \le 1
\]
and, choosing $C_1$ large enough, for all $\vphi\in C^\infty_c(D_{1})$,
\[
\int_{D_{1}} \grad u_0\cdot\grad\vphi
-\gamma \iint_{D_{2}\times D_{2}} (u_0(x')-u_0(y'))(\vphi(x')-\vphi(y'))G_\eps(x'-y',0) \dx'\dy'\\
\le \sqrt{L}\norm{\grad\vphi}_{L^\infty}.
\]
Assuming $\tau_{\mathrm{tilt}}\le\eps_{\mathrm{harm}}$, then since $\sqrt{L}\le\sigma$ by
assumption,~\cref{prp:harmapprox} gives the existence of a harmonic function $v_0\in H^1(D_1)$ such
that
\[
\int_{D_1} \abs{\grad v_0}^2\le 1\quad\text{ and }\quad
\int_{D_1} \abs{u_0-v_0}^2 \le \lambda^{n+3}.
\]
Setting $v\coloneqq \sqrt{L}v_0$, $v$ is a harmonic function in $D_1$ such that
\begin{equation}\label{excess_improv2:eq1}
\int_{D_1} \abs{\grad v}^2\le L\quad\text{ and }\quad
\int_{D_1} \abs{u-v}^2 \le \lambda^{n+3}L.
\end{equation}
Consider $w(z)\coloneqq v(0)+\grad v(0)\cdot z$ the tangent map of $v$ at the origin.
Then since $v$ is harmonic, up to choosing $\lambda_{\mathrm{tilt}}$ small enough we have
\[
\norm{v-w}_{L^\infty(D_{\lambda})}^2\le C\lambda^4\norm{\grad v}_{L^2(D_1)}^2
\le C\lambda^4 L,
\]
thus with~\cref{excess_improv2:eq1}, this implies
\begin{equation}\label{excess_improv2:eq2}
\frac{1}{\lambda^{n+1}} \int_{D_{\lambda}} \abs{u-w}^2 \le C \lambda^2 L.
\end{equation}
Defining the new direction
\[
\nu \coloneqq \dfrac1{\sqrt{1+\abs{\grad v(0)}^2}}{(-\grad v(0),1)},
\]
using the estimates~\cref{excess_improv2:eq1},\cref{excess_improv2:eq2} and the consequences of~\cref{thm:lipapprox1},
and proceeding exactly as in Step~1 of the proof of~\cite[Theorem~25.3, pp~343]{Mag2012}, we obtain the
claim~\cref{excess_improv:claim1}.\medskip\\
\textit{Step 2.}
For $\lambda$ fixed, we can assume that $\tau_{\mathrm{tilt}}$ is chosen small enough depending on
$n$, $G$, $\gamma$ and $\lambda$ to enforce~\cref{excess_improv:s1eq1}.
Then, a key observation is that with that choice of $\nu$, we have
\[
\abs{\nu-e_n}^2 \le C\int_{D_1}\abs{\grad v}^2\le CL.
\]
Thus, since $\cy(0,r,\nu)\subsq \cy(0,{\sqrt{2}r},e_n)$, by~\cref{prp:changedir,prp:scaleexcess},
if $\lambda_{\mathrm{tilt}}$ is small enough so that $M_{\mathrm{cac}}\sqrt{2}\lambda < 4$, 
\begin{equation}\label{excess_improv2:eq3}
\e(M_{\mathrm{cac}}\lambda,\nu)
\le C\left(\e_n(M_{\mathrm{cac}}\sqrt{2}\lambda)+\abs{\nu-e_n}^2\right)
\le C\left(\frac{1}{\lambda^{n-1}}\e_n(4)+L\right)\le C \frac{L}{\lambda^{n-1}}.
\end{equation}
Whence, by~\cref{excess_improv:s1eq1}, up to choosing $\lambda_{\mathrm{tilt}}$ even smaller if
necessary
\[
\e(M_{\mathrm{cac}}\lambda,\nu)\le C\lambda^{(n-1)(n+2)}\le \tau_{\mathrm{cac}}.
\]
As a consequence, we can apply~\cref{prp:strongcaccio}, which yields (recall that $Q(1/\eps)\le Q_{1-\theta}(\lambda/\eps)$)
\begin{equation}\label{excess_improv2:eq4}
\e(\lambda/2,\nu) \le
C_2\left(\f(\lambda,\nu)+\left(\frac{\eps}{\lambda}\right)^{\theta}\e(\lambda,\nu)+\lambda\Lambda
+Q_{1-\theta}\left(\frac{\lambda}{\eps}\right)\right),
\end{equation}
where $C_2=C_2(n,G,\gamma)$.
Since $\eps^\theta\le\tau_{\mathrm{tilt}}$, up to choosing
$\tau_{\mathrm{tilt}}$ even smaller if necessary depending on $\lambda$, we have
\[
\left(\frac{\eps}{\lambda}\right)^{\theta}\e(\lambda,\nu)
\le C\left(\frac{\eps}{\lambda}\right)^{\theta}\e(M_{\mathrm{cac}}\lambda,\nu)
\stackrel{\cref{excess_improv2:eq3}}{\le} \frac{L\tau_{\mathrm{tilt}}}{\lambda^{n-1+\theta}}
\le \lambda^2 L.
\]
Thus, for $\lambda_{\mathrm{tilt}}$ small enough,~\cref{excess_improv:claim1,excess_improv2:eq4}
give
\[
\e(\lambda/2,\nu)
\le C\lt(\lambda^2 L+\lambda\Lambda+Q_{1-\theta}\left(\frac{\lambda}{\eps}\right)\rt).
\]
Since $Q$ is nonincreasing, this gives~\cref{excess_improv:res} with $r=4$ and $\lambda/2$ in place
of $\lambda$, which concludes the proof.
\end{proof}

As a corollary, iterating properly~\cref{lem:excess_improv}, we get the following power decay of the
excess down to scales which are large compared to $\eps$ (the constant $\eta$ below is typically large).

\begin{prp}\label{prp:largescales_decexcess}
Assume that $G$ satisfies~\cref{Hposrad,,Hmoment1,,Hmoment2}. Let $\gamma\in(0,1)$, $\eps>0$
$\Lambda>0$, $r_0>0$ and $\eta\ge 1$ with $\Lambda r_0\le 1-\gamma$.
Given any   $\theta\in(0,1)$,  there exists a  positive
constant
$\tau_{\mathrm{dec}}^{\mathbf{\ell}}=\tau_{\mathrm{dec}}^{\mathbf{\ell}}(n,G,\gamma, \theta)$ such that the following holds. If $E$ is a
$(\Lambda,r_0)$-minimizer of $\calF_{\eps,\gamma}$ with $0\in \partial E$  satisfying, for some
 $\eta \eps\le R\le r_0$,
\begin{equation}\label{hyp:eps:large}
\e(E,R)+\Lambda R +\eta^{-1}\le \tau_{\mathrm{dec}}^{\mathbf{\ell}}
\end{equation}
then, for all $r\in[\eta \eps,R]$, we have
\begin{equation}\label{largescales_decexcess:eqres1}
\begin{aligned}
\e(E,r)
&\le C\left[ \frac{r}{R}\big(\e(E,R)+\Lambda
R\big)+Q_{1-\theta}\left(\frac{r}{\lambda\eps}\right) \right]
\end{aligned}
\end{equation}
where $\lambda$ and $C$ depend only on $n,G,\gamma$.
\end{prp}

\begin{proof}
With~\cref{lem:excess_improv} at hand, the proof of~\cref{largescales_decexcess:eqres1} is very similar to the proof of~\cref{prp:smallscales_decexcess}.
By scaling (recall~\cref{prp:rescaling}) we may assume that $R=1$. Arguing as in~\cite[Proposition
4.1]{GNR2021} we may use the scaling properties of the excess (see~\cref{prp:scaleexcess})
to post-process~\cref{lem:excess_improv} and   replace the cylindrical excess by the spherical excess both in the hypothesis~\cref{hyp:epstilt} and the conclusion~\cref{excess_improv:res}.\\
Let $\lambda=\lambda(n,G,\gamma)\le 1$ to be chosen later and set for $k\ge 0$, $r_k=\lambda^k$. By the scaling properties of the excess (see~\cref{prp:scaleexcess}) and the monotonicity of $Q_{1-\theta}$, in order to prove~\cref{largescales_decexcess:eqres1}, it is enough to prove that there exists $C>0$ such that if $r_k\ge \eta \eps$ then 
\begin{equation}\label{claimlargescales}
 \e(E,r_k)\le C \lt( r_k (\e(E,1)+\Lambda)+ Q_{1-\theta}\lt(\frac{r_k}{\eps}\rt)\rt).
\end{equation}
Let $C_{\rm tilt}>0$ be the constant given in~\cref{excess_improv:res}. We then choose $\lambda$ such that $C_{\rm tilt}\lambda^2\le \lambda/2$ and set $C=2C_{\rm tilt}$. Since~\cref{claimlargescales} holds for $k=0$ 
it is enough to show that provided it holds for $k-1$ then it also holds for $k$. \\
By~\cref{hyp:eps:large}, the induction hypothesis and the fact that $Q$ vanishes at infinity, up to choosing $ \tau_{\mathrm{dec}}^{\mathbf{\ell}}$ small enough,~\cref{hyp:epstilt} is satisfied for $r=r_{k-1}$ 
(notice that $(\eps/r_{k-1})^{\theta}\le \eta^{-\theta}$). Therefore, by~\cref{excess_improv:res} and the monotonicity of $Q_{1-\theta}$,
\begin{multline*}
 \e(E,r_k)\le C_{\rm tilt}\lt( \lambda^2 \e(E,r_{k-1}) +\Lambda r_k + Q_{1-\theta}\lt(\frac{r_k}{\eps}\rt)\rt)\\
 \stackrel{\cref{claimlargescales}}{\le} \frac{C}{2}\lt( r_k (\e(E,1)+\Lambda)+ Q_{1-\theta}\lt(\frac{r_k}{\eps}\rt)\rt) +C_{\rm tilt} \lt(\Lambda r_k + Q_{1-\theta}\lt(\frac{r_k}{\eps}\rt)\rt)\\
 \le C\lt(  r_k (\e(E,1)+\Lambda)+ Q_{1-\theta}\lt(\frac{r_k}{\eps}\rt)\rt).
\end{multline*}
This concludes the proof.
\end{proof}

\subsection[sec]{\for{toc}{\texorpdfstring{$C^{1,\alpha}$-regularity}{Regularity}}\except{toc}{\texorpdfstring{{\forcebold{$C^{1,\alpha}$}}-regularity}{Regularity}}}

Combining~\cref{prp:smallscales_decexcess,prp:largescales_decexcess} we obtain that the excess is controlled by a power of the radius down to arbitrarily small scales.

\begin{thm}\label{thm:decexcess}
Assume that $G$ satisfies~\crefrange{Hposrad}{Hdec}, and let $\gamma\in(0,1)$,
$\Lambda>0$ and $r_0>0$ with $\Lambda r_0\le 1-\gamma$.
Then, for every $\alpha\in(0, \alpha_* )$ with 
\[\alpha_*=\frac{1}{2}\frac{p_0}{(n-s_0)(n+p_0)+ p_0}(1-s_0),\] there exist
$\beta=\beta(n,G,\alpha)$, $\tau_{\mathrm{reg}}=\tau_{\mathrm{reg}}(n,G,\gamma,\alpha)$ and
$\eps_{\rm reg}=\eps_{\rm reg}(n,G,\gamma,\alpha,\Lambda)$ such that the following holds.
If $E$ is a $(\Lambda,r_0)$-minimizer of $\calF_{\eps,\gamma}$ with $\eps\in(0,\eps_{\rm reg})$ and $0\in \partial E$ satisfying, for
some  $\eps^{1-\beta}\le R\le r_0$,
\[
\e(E,R)+\Lambda R \le \tau_{\mathrm{dec}}
\]
then
\begin{equation}\label{decexcess:res}
\begin{aligned}
\e(E,r) &\le C \lt( \frac{r}{R} \lt( \e(E,R) +\Lambda R\rt) + r^{2\alpha}\rt)
\qquad\textrm{for } 0<r\le R,
\end{aligned}
\end{equation}
where $C=C(n,G,\gamma,\alpha)$.	
\end{thm}

Before giving the proof, let us explain how we deduce~\cref{mainthm:minball} and
\cref{mainthm:holdreg}.

\begin{proof}[Proof of~\cref{mainthm:holdreg}]
The proof is a standard consequence of the excess decay proven in~\cref{thm:decexcess} and Campanato's criterion for Hölder-continuous functions. We refer to~\cite{Mag2012} for more details.
\end{proof}

\begin{proof}[Proof of~\cref{mainthm:minball}]
By~\cite[Theorem~B]{Peg2021} (see also~\cite[proof~of~Theorem~2.7]{MP2021} about the hypothesis
$G\in L^1$), if $\eps$ is small enough we have existence of  minimizers $E_\eps$ for~\cref{cminpb}.
Moreover, still by~\cite[Theorem~B]{Peg2021}, they converge up to translation to $B_1$ as $\eps\to 0$. 
The convergence is meant here both in $L^1$ for the sets  and in the Hausdorff distance for the
boundaries. In addition, we have convergence of the perimeters. This yields continuity of the excess
(see e.g.~\cite{Mag2012} for this fact).
Since $B_1$ is smooth, this implies that~\cref{mainthm:holdreg} may be applied at every point of the boundary of $E_\eps$ at a scale $R$ which is uniform in  $\eps$ (and  the point).
By a standard covering argument, see e.g.~\cite{CL2012} this upgrades the  Hausdorff convergence of the boundaries to~$C^{1,\alpha}$. In particular, for $\eps$ small enough, $\partial E_\eps$ are small $C^{1,\alpha}$ graphs over $\partial B_1$.    
 Then,~\cref{mainthm:minball} is an immediate consequence of~\cite[Theorem~3]{MP2021}.
\end{proof}

Eventually, we prove the power decay of the excess.

\begin{proof}[Proof of~\cref{thm:decexcess}]
Starting from a scale $R$, the idea of the proof is to use~\cref{largescales_decexcess:eqres1} to obtain the decay of the excess up to a scale $r_+$.
Then, in order to use~\cref{excessdecaysmall} we use the scaling properties of the excess (see~\cref{prp:rescaling}) to jump to a scale $r_-$.
Setting $L=\e(E,R)+\Lambda R$, since we want that in particular $\e(E,r_+)\le C ((r_+/R) L+ r_+^{2\alpha})$, in light of~\cref{largescales_decexcess:eqres1} we need to take $r_+=\eps^{1-\beta'}$ for some $\beta'\in(0,1)$.
Similarly, since we want $\e(E,r_-)\le C r_-^{2\alpha}$,~\cref{excessdecaysmall} imposes $r_-=\eps^{1+\beta''}$ for some $\beta''>0$. Therefore, when relying on~\cref{prp:rescaling} to pass from $r_+$ to $r_-$
we will lose a factor in the estimates. To compensate that we need to assume that the starting scale $R$ is much larger than $r_+$ i.e. $R= \eps^{1-\beta}$ for some $1>\beta>\beta'$.\\

We start by choosing $\beta$ arbitrarily close to $1$, $p\in (0,p_0)$ arbitrarily close to $p_0$,  and $\theta\in (0,1)$ such that 
\[
 (1-\theta)(n-1+p_0)=(n-1+p).
\]
For  $\beta'$ to be chosen later, if $r_+=\eps^{1-\beta'}$ we have that for $\eps$ small enough depending on $\beta'$ that~\cref{hyp:eps:large} is satisfied (with $\eta=\eps^{-\beta'}$) so that~\cref{largescales_decexcess:eqres1} together with~\cref{Hdec} yield
\[
 \e(E,r)\le  C \lt(\lt(\frac{r}{R}\rt) L + \lt(\frac{\eps}{r}\rt)^{n-1+p}\rt) \qquad \textrm{for } R\ge r\ge r_+.
\]
 This gives~\cref{decexcess:res} provided 
 \[
  \lt(\frac{\eps}{r}\rt)^{n-1+p}\le r^{2\alpha} \qquad \textrm{for } R\ge r\ge r_+
 \]
which is equivalent to 
\[
  \lt(\frac{\eps}{r_+}\rt)^{n-1+p}\le r_+^{2\alpha}.
\]
Since $r_+=\eps^{1-\beta'}$, this gives the condition
\begin{equation}\label{firstcondalpha}
 2\alpha\le  \frac{\beta'}{1-\beta'}(n-1+p).
\end{equation}
Thanks to~\cref{prp:rescaling} and $L\le 1$, we then have 
\begin{equation}\label{eqdecroismeso}
 \e(E,r)\le \lt(\frac{r_+}{r}\rt)^{n-1}\e(E,r_+)\le C \lt(\frac{r_+}{r}\rt)^{n-1}\lt(\frac{r_+}{R} + \lt(\frac{\eps}{r_+}\rt)^{n-1+p}\rt) \qquad \textrm{for } r_+\ge r\ge r_-.
\end{equation}
In order to have $\e(E,r)\le C r^{2\alpha}$ in this range it is enough to have this estimate for $r=r_-$ which  means 
\[
 \eps^{-(\beta'+\beta'')(n-1)}(\eps^{\beta-\beta'}+\eps^{\beta'(n-1+p)})\le \eps^{2\alpha(1+\beta'')}.
\]
In particular, the optimal choice is 
\begin{equation}\label{condbeta'}
 \beta'= \frac{\beta}{n+p},
\end{equation}
for  which  the condition becomes
\begin{equation}\label{secondcondalpha}
 2\alpha\le \frac{\beta' p-(n-1)\beta''}{1+\beta''}.
\end{equation}
Let us point out that~\cref{secondcondalpha} is stronger than~\cref{firstcondalpha}. Notice that under~\cref{secondcondalpha}, we have in particular that~\cref{smallinit} is satisfied at $R=r_-$ provided $\eps$ is small enough. We may thus use~\cref{excessdecaysmall} to obtain
\[
 \e(E,r)\le C\lt(\frac{r}{r_-}\rt)^{1-s_0} \lt(\e(E,r_-) +\lt(\frac{r_-}{\eps}\rt)^{1-s_0}\rt)\qquad \textrm{for } r_-\ge r>0.
\]
By~\cref{eqdecroismeso} (with the choice~\cref{condbeta'}) this reduces to 
\[
 \e(E,r)\le C\lt(\frac{r}{r_-}\rt)^{1-s_0} \lt(\eps^{\beta'p-(n-1)\beta''} +\eps^{(1-s_0)\beta''}\rt)\qquad \textrm{for } r_-\ge r>0. 
\]
In particular $\e(E,r)\le C r^{2\alpha}$ for $r\le r_-$ provided it holds for $r=r_-$ i.e. 
\[
 \eps^{\beta'p-(n-1)\beta''} +\eps^{\beta''(1-s_0)}\le \eps^{2\alpha(1+\beta'')}
\]
which gives the condition 
\[
 2\alpha \le \min\lt( \frac{\beta'p -(n-1)\beta''}{1+\beta''},\frac{\beta''}{1+\beta''} (1-s_0)\rt).
\]
This constraint implies~\cref{secondcondalpha} (which itself implies~\cref{firstcondalpha}). Since the first term is decreasing in $\beta''$ and the second one is increasing, we see that the optimal choice of $\beta''$ is obtained when both terms are equal, that is 
\[
 \beta''=\frac{\beta p}{(n-s_0) (n+p)}
\]
which gives 
\[
 2\alpha\le \frac{\beta p}{(n-s_0)(n+p)+\beta p}(1-s_0).
\]
Since $p$ can be  taken arbitrarily close to $p_0$ and $\beta$ arbitrarily close to $1$, this concludes the proof.
\end{proof}

\noindent
{\bf Acknowledgments.}
M. Goldman was partially supported by the ANR grant SHAPO.
B. Merlet and M. Pegon are partially supported by the Labex CEMPI (ANR-11-LABX-0007-01).

\printbibliography 

\makeatletter
\providecommand\@dotsep{5}
\makeatother
\listoftodos\relax

\end{document}